\documentclass[11pt]{amsart}
\usepackage{amsmath,amssymb,amsthm}
\usepackage{enumitem}
\usepackage{mathtools}
\mathtoolsset{showonlyrefs,showmanualtags}
\usepackage[margin=0.9in]{geometry}
\usepackage[final]{microtype}
\usepackage[
  colorlinks=true,
  linkcolor=blue,
  citecolor=blue,
  urlcolor=blue
]{hyperref}

\numberwithin{equation}{section}
\newtheorem{theorem}[equation]{Theorem}
\newtheorem{lemma}[equation]{Lemma}
\newtheorem{corollary}[equation]{Corollary}
\theoremstyle{definition}
\newtheorem{definition}[equation]{Definition}
\newtheorem{remark}[equation]{Remark}

\allowdisplaybreaks
\AtBeginDocument{%
  \setlength{\abovedisplayskip}{4pt plus 1pt minus 1pt}%
  \setlength{\belowdisplayskip}{4pt plus 1pt minus 1pt}%
  \setlength{\abovedisplayshortskip}{2pt plus 1pt minus 1pt}%
  \setlength{\belowdisplayshortskip}{3pt plus 1pt minus 1pt}%
  \setlength{\jot}{1.5pt}%
}
\setlist{leftmargin=*,topsep=2pt,itemsep=0pt,parsep=0pt,partopsep=0pt}

\title[Density Bounds in Hua's Theorem]
{Density Bounds in Hua's Theorem for Prime Squares}
\author[M. T. Lacey et al.]
{Michael T. Lacey, Hamed Mousavi, Yaghoub Rahimi, and Manasa N. Vempati}
\date{}
\hypersetup{
  pdftitle={Density Bounds in Hua's Theorem for Prime Squares},
  pdfauthor={Michael T. Lacey, Hamed Mousavi, Yaghoub Rahimi,
    and Manasa N. Vempati},
  pdfsubject={Density versions of Hua's theorem for prime squares},
  pdfkeywords={Hua's theorem, prime squares, density version,
    Waring--Goldbach problem, transference principle}
}

\begin{document}

\begin{abstract}
We prove a density version of Hua's theorem in the  five-variable case: if
\(\mathcal P\) is a subset of the primes with relative lower density greater than  $0.82$,
then every sufficiently large \(n\equiv5\pmod{24}\) is a sum of five squares of primes from \(\mathcal P\).
% The proof combines a transference principle adapted to quadratic prime
% weights, estimates for prime-square exponential sums, and new weighted
% covering arguments for power residues based on elementary additive
% combinatorics.
%The same framework also yields improved density bounds for every number of variables \(s\geq5\).
\end{abstract}

%\begin{abstract}
% For \(s\geq5\), let \(\mathfrak d_s\) denote the density threshold in
% Hua's theorem for squares of primes.  We prove
% \[
% \mathfrak d_s\leq
% \begin{cases}
% (6+5\sqrt2)/16,&s=5,6,\\
% (1+3\sqrt2)/7,&s=7,\\
% 1/\sqrt2,&s\geq8.
% \end{cases}
% \]

% {\color{blue}In particular, the \(1/\sqrt{2}\) threshold is attained with eight variables, improving the previously known range of \(s\ge 13\), while the five-variable result gives a new bound in the critical case of Hua's theorem. The proof develops a quadratic transference framework, combining new weighted sumset estimates for reduced quadratic residues with Fourier pseudorandomness and restriction estimates relative to a non-flat measure naturally associated with prime squares.}

% Thus these estimates strictly improve the previously available
% homogeneous bounds listed in the introduction for
% \(5\leq s\leq12\), and the threshold \(1/\sqrt2\) is reached already
% from eight variables onward; for \(s\geq13\) our bound agrees with the
% value obtained by Lim.  The proof combines score-sensitive local
% sumset estimates, an affine truncation of the local mean functions,
% and transference relative to a quadratic reference measure.  We also
% prove the following uniform finite theorem.  For every \(k\geq2\)
% there is an explicit integer \(R_k\) such that, whenever
% \((q,R_k)=1\), \(s\geq6k-2\), and
% \(A\) occupies more than half of the reduced \(k\)-th-power residues
% modulo \(q\), one has \(sA=\mathbb Z_q\).
%\end{abstract}

\maketitle

\tableofcontents

\section{Introduction}\label{sec:introduction}

The Waring--Goldbach problem asks for representations of integers as
sums of fixed powers of primes.  In the quadratic case, Hua proved in
1938 that, for every \(s\geq5\), every sufficiently large integer
\(n\equiv s\pmod {24}\) is a sum of \(s\) squares of primes
\cite{hua}.  The congruence condition is unavoidable, since
\(p^2\equiv1\pmod {24}\) for every prime \(p>3\).  The case \(s=5\)
is the critical one and is usually referred to as Hua's theorem.

A natural density version asks how large a subset of the primes must be
for Hua's conclusion to remain valid.  More precisely, for
\(\mathcal P\subseteq\mathbb P\), we write
\[
    \underline d(\mathcal P)
    =\liminf_{x\to\infty}
      \frac{|\mathcal P\cap[1,x]|}{\pi(x)}
\]
for its relative lower density in the primes.  Such questions have
their prototype in the ternary Goldbach problem.  Li and Pan showed
that relative lower density greater than \(2/3\) is sufficient for a
subset of the primes to represent every sufficiently large odd integer
as a sum of three of its elements \cite{li2010density}.  Shao
subsequently lowered the threshold to the sharp value \(5/8\)
\cite{shao2014density}.  These results belong to a broader line of
work, initiated by Green's transference method
\cite{BenGreen,GreenTao}, in which additive problems in the primes are
reduced to dense problems after the local irregularities of the primes
have been removed; see, for example,
\cite{RamareRuzsa,ChipeniukHamel,Matomaki,salmensuu2020density, LHRV}.

%   The use of lower rather
% than upper relative density is important for a statement concerning
% \emph{every} sufficiently large integer.  Indeed, Yang and Togb\'e
% constructed a set of primes of upper relative density \(1\) for which
% infinitely many odd integers fail to be sums of three elements of the
% set \cite{YangTogbe}.  Thus a limsup density hypothesis, even at its
% largest possible value, does not by itself control the additive basis
% property.

Density versions of Waring or  Waring--Goldbach problems were later studied by
Salmensuu and others, and the quadratic prime problem has seen rapid
recent progress.  Gao showed that if \(\mathcal P\subseteq\mathbb P\) has relative lower
density greater than \(3/4\), then, for every \(s\geq44\), every
sufficiently large integer \(n\equiv s\pmod{24}\) is a sum of \(s\)
squares of primes from \(\mathcal P\) \cite{gao}.  Zhao proved that, for \(s\geq8\),
the density
\[
    \sqrt{1-\frac{\min\{s,16\}}{32}}
\]
is sufficient \cite{zhao}.  Tan proved a
five-variable asymmetric theorem: if
\(\mathcal P_1,\ldots,\mathcal P_5\) have relative lower densities
\(\varepsilon_1,\ldots,\varepsilon_5\) satisfying
\(\sum_i\varepsilon_i^2>4\), then every sufficiently large
\(n\equiv5\pmod {24}\) is a sum
\(p_1^2+\cdots+p_5^2\) with \(p_i\in\mathcal P_i\)
\cite[Theorem~1.1]{Tan}.  In the homogeneous case this gives the
threshold \(2/\sqrt5 \simeq 0.894\).  More recently, Lim obtained a sharp density
theorem for square residues in suitable finite cyclic groups and used
it to improve the global bounds: in particular, his homogeneous
thresholds include \(\sqrt{3/4}\) for \(s=7\),
\(\sqrt{21/32}\) for \(s=8\), and decrease to \(1/\sqrt2\) when
\(s=13\) \cite[Theorems~1.1 and~1.4]{lim2025density}.

To formulate the problem uniformly, for \(s>4\) let
\(\mathfrak d_s\) denote the smallest density threshold in Hua's
theorem.

\begin{definition}\label{def:density-threshold}
For \(s>4\), let \(\mathfrak d_s\) be the infimum of the real numbers
\(d\in[0,1]\) with the following property: whenever
\(\mathcal P\subseteq\mathbb P\) is fixed and
\(\underline d(\mathcal P)>d\), there is
\(X_0=X_0(\mathcal P,s)\) such that every \(n\geq X_0\) satisfying
\(n\equiv s\pmod {24}\) has a representation
\[
    n=p_1^2+\cdots+p_s^2,
    \qquad p_1,\ldots,p_s\in\mathcal P.
\]
%Repetitions among the primes are permitted.
\end{definition}

Thus the basic problem is simply to determine \(\mathfrak d_s\) for
\(s\geq5\).  Our main result gives the following bounds.

\begin{theorem}\label{t:main}
For every integer \(s\geq5\),
\begin{equation}\label{e:csdeltas}
    \mathfrak d_s\leq
    \begin{cases}
        \dfrac{6+5\sqrt2}{16},&s=5,6,\\[1mm]
        \dfrac{1+3\sqrt2}{7},&s=7,\\[1mm]
        \dfrac1{\sqrt2},&s\geq8.
    \end{cases}
\end{equation}
\end{theorem}

Theorem~\ref{t:main} improves the previously available homogeneous bounds for
every \(5\leq s\leq12\).  In the five-variable case it replaces
Tan's \(2/\sqrt5\) by $\frac{6+5\sqrt2}{16}$,
while for \(s=7\) it improves Lim's \(\sqrt{3/4}\).  Moreover,
the threshold \(1/\sqrt2\), previously reached by Lim from \(s=13\),
already holds for \(s=8\).  The case \(s=6\) follows from \(s=5\),
and every \(s\geq9\) follows from \(s=8\): one fixes a prime
\(p_0\in\mathcal P\) and subtracts the required number of copies of
\(p_0^2\).  Since \(p_0^2\equiv1\pmod {24}\), this preserves the
appropriate admissible congruence class.

We do not know the correct values of \(\mathfrak d_s\), and at present
there does not seem to be a compelling conjectural formula for them.
Even the  five-variable threshold is not known exactly.  Combining
Theorem~\ref{t:main} with our construction  in Section~8 gives
the following range.

\begin{theorem}\label{t:d5-range}
One has
\[
    \frac7{10}\leq \mathfrak d_5
    \leq \frac{6+5\sqrt2}{16}
    =0.81694\ldots .
\]
\end{theorem}

For the lower bound, we construct a set of primes of relative density
\(7/10\) for which infinitely many integers \(n\equiv5\pmod{24}\)
are not sums of five prime squares from the set.  For \(s>5\), the
elementary choice of primes in the two reduced classes \(1,4\pmod5\)
gives \(\mathfrak d_s\geq1/2\), since all of their squares are
\(1\pmod5\).
 In
particular, it is natural to ask whether the value \(1/\sqrt2\)
appearing from \(s=8\) onward reflects the true stable threshold; we
do not make such a conjecture here.

\bigskip

We also require a finite covering theorem for power residues, whose
weighted form provides the arithmetic input for
Theorem~\ref{t:main}.  The result may also be of independent interest,
so we record it here as an auxiliary theorem. For
\(q\geq1\), write
\begin{align}\label{e:zqdef}
    \mathbb Z_q=\mathbb Z/q\mathbb Z,\qquad
    \mathbb Z_q^*=(\mathbb Z/q\mathbb Z)^\times,\qquad
    \mathbb Z_{q,k}=\{x^k:x\in\mathbb Z_q^*\},
\end{align}
with the convention \(\mathbb Z_{1,k}=\{0\}\).  For \(k\geq2\), put
\begin{equation}\label{e:rkdef}
    R_k=
    \prod_{\substack{p\ {\rm prime}\\
        p\mid k\ {\rm or}\ (p-1)\mid2k}}p.
\end{equation}
In particular \(R_2=30\).

\begin{theorem}\label{lemma:arithmetic-general-k}
Let \(k\geq2\).  Suppose that
\[
    (q,R_k)=1,\qquad s\geq6k-2,
    \qquad
    A\subseteq\mathbb Z_{q,k},
    \qquad
    |A|>\frac12|\mathbb Z_{q,k}|.
\]
Then $ sA=\mathbb Z_q.$
\end{theorem}
The theorem is uniform in \(q\), apart from the explicit exceptional
prime factors contained in \(R_k\).  Notice, however, that the density
assumption is relative to \(\mathbb Z_{q,k}\), not to the whole group
\(\mathbb Z_q\).  When \(q\) has many prime factors, the set of reduced
\(k\)-th powers can occupy only a small proportion of \(\mathbb Z_q\).
Thus a set \(A\) containing more than half of \(\mathbb Z_{q,k}\) may
still be very sparse in the ambient cyclic group, so classical dense
sumset results such as Cauchy--Davenport or Kneser's theorem do not
apply directly; see, for example, \cite{TaoVu}.  The theorem is instead
closer in spirit to covering results for dense subsets of sparse
arithmetic sets, such as those arising in
\cite{RamareRuzsa,shao2014density}.

For \(k=2\), Theorem~\ref{lemma:arithmetic-general-k} says in
particular that, whenever \((q,30)=1\), every subset containing more
than half of the reduced quadratic residues satisfies
\(10A=\mathbb Z_q\).  This should be compared with Lim's sharper
quadratic result, which gives the optimal threshold \(1/5\) when the
prime factors of the modulus are sufficiently large
\cite[Theorems~1.1 and~1.3]{lim2025density}.  Our theorem is of a
different nature: it applies uniformly to every \(k\) and excludes only
an explicit finite collection of primes.  
% For the main quadratic
% theorem, ten variables are of course not enough: reaching
% \(s=5,7,8\) requires separate refinements at the small primes, with
% the prime \(5\) playing the decisive role.

\subsection{Overview of the proof}

We explain the argument in the critical case of five prime squares.
The cases of seven and eight variables follow the same analytic
framework, with different local arithmetic estimates, while the
remaining cases of Theorem~\ref{t:main} follow by adjoining fixed
prime squares as above.  Let \(X\equiv5\pmod {24}\) be large.  As
a natural approach, we remove small-prime congruence obstructions by introducing
\[
    W=8\prod_{2<p<w}p=24W_0.
\]
For each reduced quadratic residue \(b\pmod W\), the primes
\(p\in\mathcal P\) satisfying \(p^2=Wn+b\) give a weighted function
\(f_b\) on an interval \([N]\), and we write
\[
    m_b=\frac1N\sum_{n\leq N}f_b(n)
\]
for its local mean.  The scale is chosen so that the integer $X$ corresponds, after division by \(W\), to a point near
\(5N/2\).

At this stage the problem has a simple shape.  If, for
\(r\equiv X\pmod W\), we can choose
\(b_1,\ldots,b_5\in\mathbb Z_{W,2}\) such that
\begin{equation}\label{eq:intro-local-score}
    b_1+\cdots+b_5\equiv r\pmod W,\qquad
    m_{b_j}>\zeta,\qquad
    \sum_{j=1}^5 m_{b_j}>\frac{5}{\sqrt2}+\zeta,
\end{equation}
then a transference argument shows that
\begin{align}\label{e:convest}
    (f_{b_1}*\cdots*f_{b_5})(m)>0
\end{align}
at the required value of \(m\), and hence gives
\(X=p_1^2+\cdots+p_5^2\) with \(p_j\in\mathcal P\).
The proof therefore separates naturally into an analytic problem,
which transfers \eqref{eq:intro-local-score} to a positive
convolution, and an arithmetic problem, which produces the residue
classes \(b_j\).

The analytic side requires three pieces of information.  First, the
majorant \(\nu_b\) for \(f_b\) must be Fourier-pseudorandom.  For
squares the correct comparison object is not the constant function:
the change of variables \(p^2=Wn+b\) leaves a non-flat density in
the \(n\)-variable.  We therefore compare \(\nu_b\) with a natural
quadratic reference measure \(\mu_N\) and prove
\[
    \|\widehat{\nu_b}-\widehat{\mu_N}\|_\infty=o(N).
\]
This is the main difference from a flat transference model such as \cite{salmensuu2020density}.  On the
major arcs, the prime-square sum is evaluated by
Siegel--Walfisz together with the relevant quadratic local factors;
Zhao's finite evaluation identifies their size
\cite{zhao}.  On the minor arcs, Hua's estimate gives the required
saving.  The outcome is that the prime-square majorant has the same
Fourier behaviour, to the precision needed for transference, as the
quadratic reference measure.

Second, one needs a restriction estimate
\[
    \|\widehat h\|_q\ll_q N^{1-1/q},
    \qquad q>4,\qquad |h|\leq\nu_b.
\]
The argument follows Bourgain's large-spectrum method for squares
\cite{BourgainSquares}, but it has to be carried out for the present
prime-square weight.  At a high level there are three ingredients.
A fourth-moment estimate controls the small Fourier levels.  For the remaining levels, the circle-method
estimate above shows that a suitable power of
\(N^{-1}|\widehat{\nu_b}(\alpha)|\) is dominated by a sum of bumps
around rationals \(a/q\).  Finally, the standard large-spectrum
argument converts this rational envelope into a bound for the number
of frequencies at which \(|\widehat h|\) is large.  The only point
where the classical picture has to be changed substantially is the
envelope itself: its local factors and its quadratic decay come from
the non-flat prime-square majorant rather than from the usual flat
sieve weight.

Pseudorandomness and restriction allow us to replace each \(f_{b_j}\)
by a smoother bounded function with essentially the same mean in \eqref{e:convest}.  The
smoothed functions are dense enough that their five-fold convolution
can be bounded from below using a standard dense-set argument, while
the Fourier estimates ensure that this approximation does not change
the original convolution significantly.  Thus the analytic part
reduces the theorem to finding five residue classes satisfying
\eqref{eq:intro-local-score}.

We now turn to that finite problem.  The global density assumption
first implies that the local means \(m_b\), averaged over the reduced
square residues modulo \(W\), have essentially the same average as
\(\underline d(\mathcal P)\).  What is needed is stronger than a
large average: we must choose five particular residues which add to a
prescribed value, all carry positive mass, and retain enough of the
total mass to pass the threshold \(5/\sqrt2\).  
%This is the reason
%for proving functional rather than merely set-theoretic covering
%statements.

The basic finite mechanism is easiest to see in a homogeneous toy
model.  Let \(f\) be a weight on a set of \(n\) residues, and arrange
its values in decreasing order,
\[
    1\geq a_0\geq a_1\geq\cdots\geq a_{n-1}\geq0.
\]
Think of using the same weighted set in each of the five summands.  From
a lower bound for the average of the \(a_i\), we want to choose indices
\(i_1,\ldots,i_5\) such that
\begin{equation}\label{eq:intro-rank-selection}
    i_1+\cdots+i_5\geq2n,\qquad
    a_{i_1}+\cdots+a_{i_5}>5c,\qquad
    a_{i_j}>0.
\end{equation}
There is a simple tradeoff: moving farther down the sequence increases
the size of the corresponding level set, but decreases the weight.
An application of Cauchy--Davenport converts the rank condition in
\eqref{eq:intro-rank-selection} into covering of the required residue,
while the second condition retains the quantitative weight needed for
transference. The selection argument in the linear prime setting of \cite{LHRV} may
be viewed as a simpler model for the present one.  Here the same basic
idea has to be refined to handle different local fibres and to preserve
a quantitative lower bound through the prime-power and
Chinese-remainder steps.

The five-variable quadratic problem requires more care because the
generic half-density argument does not retain enough score at the
small primes.  The main obstruction occurs at \(5\), where
\(\mathbb Z_{5,2}=\{1,4\}\).  Here the finite problem is small enough
that the possible choices can be compared directly, giving the
required inhomogeneous two-point estimate.  Prime-power lifting then
propagates this finite estimate from \(5\) to \(5^\alpha\).

For general \(k\), the same argument gives the linear bound \(6k-2\)
for the number of summands, uniformly over moduli avoiding the finite
set of exceptional primes in \(R_k\).  Salmensuu had conjectured an
\(O(k)\) bound in the corresponding local density Waring problem
\cite[Section~5]{salmensuu2020density}; thus our covering theorem
obtains this linear dependence away from the exceptional primes.

Unlike in the linear prime setting, the final density threshold is constrained
simultaneously by the finite arithmetic covering problem and by the quadratic
transference step.  The constant in Theorem~\ref{t:main} therefore reflects
the interaction of these two features of the prime-square problem.

\subsection{Organization of the paper}
The paper is organized as follows.  Section~2 collects notation and
the elementary analytic and arithmetic tools used later.  Section~3
proves the combinatorial selection estimates illustrated by
\eqref{eq:intro-rank-selection}.  Section~4 develops the arithmetic
of reduced \(k\)-th-power residues and proves
Theorem~\ref{lemma:arithmetic-general-k} in its stronger functional
form.  Section~5 establishes the small-variable quadratic refinements
at the prime \(5\).  Section~6 introduces the prime-square \(W\)-trick,
proves the transference statement, and deduces
Theorem~\ref{t:main} from the analytic and arithmetic inputs.
Section~7 proves pseudorandomness, the restriction estimate, and the
mean-value estimate for the prime-square weights.  Finally, Section~8
constructs the density-\(7/10\) obstruction giving Theorem \ref{t:d5-range}.

\subsection{Acknowledgments}
ChatGPT was used as an auxiliary editorial tool for proofreading and preliminary
literature search assistance.  Claude AI was used in the final editing process. 
Claude AI also prepared a \href{https://github.com/lacey-gatech/Density-Bounds-in-Hua-s-Theorem-for-Prime-Squares}{partial formalization} of this paper, 
with  two theorems taken as true, the Siegel-Walfisz Theorem \ref{SWT} and 
Hua's Theorem as stated in \cite[Lemma~4.5]{zhao}. 

All mathematical statements, references, arguments,
and final wording were checked and approved by the authors.

\section{Preliminaries}\label{sec:preliminaries}

We write \(\mathbb P\) for the primes, \([N]=\{1,\ldots,N\}\),
\(\mathbf1_E\) for the indicator of a set \(E\), and \(\phi\) for
Euler's totient function. The symbols \(O\), \(\ll\), and \(\asymp\),
with subscripts when needed, have their usual meanings; all
uniformities and parameter dependences in the analytic argument are
stated when they arise.

We first record two elementary facts concerning the reduced power
residues \(\mathbb Z_{q,k}\) defined in \eqref{e:zqdef}.

\begin{lemma}\label{l:ntlem}
Let \(q,q'\) be coprime positive integers and let \(k\geq2\). Under the
Chinese remainder isomorphism
\[
    \mathbb Z_{qq'}\cong\mathbb Z_q\times\mathbb Z_{q'},
\]
one has
\[
    \mathbb Z_{qq',k}
    \cong
    \mathbb Z_{q,k}\times\mathbb Z_{q',k}.
\]
\end{lemma}

\begin{proof}
Only surjectivity requires proof. Let
\(x\in\mathbb Z_{q,k}\) and \(y\in\mathbb Z_{q',k}\), and write
\[
    x\equiv a^k\pmod q,
    \qquad
    y\equiv b^k\pmod{q'},
\]
where \(a\in\mathbb Z_q^*\) and \(b\in\mathbb Z_{q'}^*\). Choose
\(r,t\in\mathbb Z\) with \(rq+tq'=1\), and put
\[
    d=atq'+brq.
\]
Then \(d\equiv a\pmod q\) and \(d\equiv b\pmod{q'}\). In particular,
\((d,qq')=1\), and
\[
    d^k\equiv x\pmod q,
    \qquad
    d^k\equiv y\pmod{q'}.
\]
The result follows from the Chinese remainder theorem.
\end{proof}

\begin{lemma}\label{l:ntlem2}
Let \(p\) be an odd prime, let \(a\geq2\), and let \(k\geq2\).
If \(p\nmid k\), then the reduction map
\[
    \mathbb Z_{p^a,k}\longrightarrow\mathbb Z_{p^{a-1},k}
\]
is surjective and every fiber contains exactly \(p\) elements.
Equivalently, for every \(x\in\mathbb Z_{p^{a-1},k}\), all \(p\)
residues
\[
    x+tp^{a-1},\qquad t\in\mathbb Z_p,
\]
belong to \(\mathbb Z_{p^a,k}\).
\end{lemma}

\begin{proof}
The group \((\mathbb Z/p^a\mathbb Z)^*\) is cyclic of order
\((p-1)p^{a-1}\). Hence
\[
    |\mathbb Z_{p^a,k}|
    =
    \frac{(p-1)p^{a-1}}
         {\gcd(k,(p-1)p^{a-1})}.
\]
Projection maps \(k\)-th powers in \(\mathbb Z_{p^a}\) surjectively
onto \(k\)-th powers in \(\mathbb Z_{p^{a-1}}\). Since \(p\nmid k\),
the quotient of the cardinalities at two consecutive levels is \(p\).
Thus every fiber of the restricted reduction map has cardinality
\(p\). Since a unit modulo \(p^{a-1}\) has exactly \(p\) unit lifts
modulo \(p^a\), all of those lifts are \(k\)-th powers.
\end{proof}

All finite groups are written additively. For subsets
\(A_1,\ldots,A_s\) of an abelian group, \(A_1+\cdots+A_s\) denotes
their sumset, and \(sA=A+\cdots+A\). We repeatedly use the following
standard form of Cauchy--Davenport; see \cite[Theorem~5.4]{TaoVu}.

\begin{theorem}[Cauchy--Davenport]\label{thm:cauchy-davenport}
Let \(p\) be prime and let \(A_1,\ldots,A_r\subseteq\mathbb Z_p\) be
nonempty. Then
\[
    |A_1+\cdots+A_r|
    \geq
    \min\left\{p,\sum_{i=1}^r|A_i|-(r-1)\right\}.
\]
In particular, \(A_1+\cdots+A_r=\mathbb Z_p\) whenever
\(\sum_i|A_i|\geq p+r-1\).
\end{theorem}

We now fix the Fourier-analytic conventions used in the transference
and prime-square estimates. Set
\[
    e(x)=e^{2\pi ix},\qquad \mathbb T=\mathbb R/\mathbb Z,
\]
and equip \(\mathbb T\) with normalized Haar measure. If
\(F\colon\mathbb Z\to\mathbb C\) is finitely supported, define
\[
    \widehat F(\alpha)=\sum_nF(n)e(n\alpha),
    \qquad
    \|\widehat F\|_q
    =\left(\int_{\mathbb T}|\widehat F(\alpha)|^q\,d\alpha\right)^{1/q}.
\]
Functions on \([N]\) are extended by zero. Convolution is
unnormalized, so
\[
    \widehat{F_1*F_2}=\widehat F_1\widehat F_2,
    \qquad
    (F_1*\cdots*F_s)(m)
    =\int_{\mathbb T}\prod_{j=1}^s\widehat F_j(\alpha)
      e(-m\alpha)\,d\alpha.
\]
We also recall Parseval's identity
\[
    \|\widehat F\|_2^2
    =\int_{\mathbb T}|\widehat F(\alpha)|^2\,d\alpha
    =\sum_n|F(n)|^2
    =\|F\|_{\ell^2}^2,
\]
and the elementary estimate
\(\|\widehat F\|_\infty\leq\|F\|_{\ell^1}\).

We shall also use partial summation and the standard geometric-sum
estimate repeatedly. If \(u<v\),
\(C(t)=\sum_{u<n\leq t}c_n\), and \((a_n)\) is any sequence, then
\[
    \sum_{u<n\leq v}a_nc_n
    =a_vC(v)+\sum_{u\leq t<v}(a_t-a_{t+1})C(t).
\]
When \(a_n\) is nonnegative and decreasing, the absolute value of the
left-hand side is at most
\[
    a_{u+1}\max_{u<t\leq v}|C(t)|.
\]
For $\alpha\in\mathbb T=\mathbb R/\mathbb Z$, we write $$\|\alpha\|_{\mathbb T}:= \min_{n\in\mathbb Z}|\widetilde{\alpha}-n|,$$
where $\widetilde{\alpha}\in\mathbb R$ is any representative of $\alpha$. This is well defined, since the value is independent of the choice of representative.
We also have $$\left|\sum_{u<n\leq v}e(n\alpha)\right|
    \ll\min\{v-u,\|\alpha\|_{\mathbb T}^{-1}\},$$

with the usual convention at \(\alpha=0\). 

For the major and minor arc decomposition, we use Dirichlet's approximation theorem in the following standard form. 

\begin{theorem}[Dirichlet's approximation theorem]
Let \(\alpha\in\mathbb T\) and \(N\geq1\). Then there exist coprime
integers \(a,q\) with \(1\leq q\leq N\) such that
\[
    \left\|\alpha-\frac aq\right\|_{\mathbb T}
    \leq\frac1{qN}.
\]
\end{theorem}

Given \(A\geq1\) and sufficiently large \(N\), set
\(Q=(\log N)^A\). For \(1\leq q\leq Q\) and
\(a\in\mathbb Z_q^*\), define
\[
    \mathfrak M(q,a)
    =
    \left\{\alpha\in\mathbb T:
    \left\|\alpha-\frac aq\right\|_{\mathbb T}
    \leq\frac{Q}{qN}\right\}.
\]
For \(q=1\), we take \(a=0\). If \(N>2Q^2\), these arcs are disjoint.
Let
\[
    \mathfrak M=\mathfrak M_Q
    :=\bigcup_{\substack{1\leq q\leq Q\\a\in\mathbb Z_q^*}}
      \mathfrak M(q,a),
    \qquad
    \mathfrak m=\mathbb T\setminus\mathfrak M,
\]
with the preceding convention when \(q=1\).

We shall use both the continuous and discrete forms of van der
Corput's second derivative estimate.

\begin{theorem}[van der Corput estimates]\label{van der Corput's}
Let \(a<b\) and let \(\phi\in C^2([a,b])\) be real-valued.
\begin{enumerate}
    \item If \(|\phi''(t)|\geq\lambda\) for \(t\in[a,b]\) and some
    \(\lambda>0\), then
    \[
        \left|\int_a^b e(\phi(t))\,dt\right|
        \ll\lambda^{-1/2}.
    \]

    \item If, in addition, for some \(A\geq1\) and every $t\in[a,b]$, we have $\lambda\leq|\phi''(t)|\leq A\lambda$,
    then
    \[
        \left|\sum_{a<n\leq b}e(\phi(n))\right|
        \ll A(b-a)\lambda^{1/2}+\lambda^{-1/2}.
    \]
\end{enumerate}
The implied constants are absolute.
\end{theorem}

We finally record the two standard estimates for primes used in the
major arc analysis.

\begin{theorem}[Chebyshev's estimate]\label{Chebyshev's}
There exist absolute constants \(c_1,c_2>0\) such that, for all
sufficiently large \(x\),
\[
    c_1x\leq\vartheta(x)\leq c_2x,
    \qquad
    \vartheta(x)=\sum_{p\leq x}\log p.
\]
In particular, \(\vartheta(x)\ll x\).
\end{theorem}

The major arc analysis also requires uniform distribution of primes
in arithmetic progressions with logarithmic-size moduli. For this we
use the Siegel--Walfisz theorem in the following form.

\begin{theorem}[Siegel--Walfisz theorem]\label{SWT}
Let \(A,B>0\). Uniformly for integers \(q\geq1\) and \(a\) satisfying
\((a,q)=1\) and \(q\leq(\log x)^A\), one has
\[
    \vartheta(x;q,a)
    :=
    \sum_{\substack{p\leq x\\p\equiv a\pmod q}}\log p
    =
    \frac{x}{\phi(q)}
    +O_{A,B}\!\left(x(\log x)^{-B}\right).
\]
\end{theorem}

We shall also use two standard estimates for the divisor function which we state now. For a positive integer $n$, let
\[
    \tau(n):=\sum_{d\mid n}1
\]
denote the number of positive divisors of $n$. Note that the divisor function is submultiplicative:
\[
    \tau(mn)\leq \tau(m)\tau(n)
    \qquad (m,n\geq1).
\]
Indeed, if
\[
    m=\prod_p p^{\alpha_p},
    \qquad
    n=\prod_p p^{\beta_p},
\]
then we have
\[
    \tau(mn)
    =
    \prod_p (\alpha_p+\beta_p+1)
    \leq
    \prod_p(\alpha_p+1)(\beta_p+1)
    =
    \tau(m)\tau(n).
\]

We also recall the classical second-moment estimate
\[
    \sum_{n\leq N}\tau(n)^2
    \ll
    N(\log N)^3.
\]

\section{Combinatorics}\label{s:comb1}

This section develops the combinatorial selection lemmas used in the
proof of Theorem~\ref{lemma:arithmetic-general-k} and later in the proof
of Theorem~\ref{t:main}. 
%In the arithmetic applications,
% \(a_{i,j}\) is an order statistic of a local weight and the associated
% superlevel set has at least \(i+1\) elements.  The lower bound on the
% sum of the selected ranks is calibrated so that Cauchy--Davenport
% forces a prescribed residue, while the value and positivity conditions
% retain the quantitative score required by transference.
The integer \(k\) in this section is a rank parameter; in later
applications its value will be specified separately from any power
exponent already in use.

\subsection{The general combinatorial selection lemma}

\begin{lemma}\label{l:s-seq-general}
    Fix integers \(k\geq1\), \(s\geq6k-2\), and \(n\geq3\), and let
    \(c>1/2\).  For \(1\leq j\leq s\), let
    \(a_{0,j},\ldots,a_{n-1,j}\) be a non-increasing sequence in
    \([0,1]\), with \(a_{0,j}>0\).  Assume that
    \[
    \frac{1}{s}\sum_{j=1}^s
    \left(
        \frac{1}{n}\sum_{i=0}^{n-1} a_{i,j}
    \right)
    >c.
    \]
    Then there are indices \(0\leq i_1,\ldots,i_s\leq n-1\) such that
    \[
        i_1+\cdots+i_s\geq kn,
        \qquad
        \frac1s\sum_{j=1}^sa_{i_j,j}>c,
        \qquad
        \prod_{j=1}^sa_{i_j,j}>0.
    \]
\end{lemma}

\subsection{The three-level base case}

Before proving the general selection lemma, we establish its
three-level base case.

\begin{lemma}\label{l:diffsetsbasegen}
    Fix integers \(k\geq1\) and \(s\geq6k-2\), and fix
    \(c>\frac12\).
    For each \(1 \leq j \leq s\), let
    \(\{a_{i,j}\}_{i=0}^2\) be a non-increasing sequence in \([0,1]\),
    with \(a_{0,j}>0\). Suppose that
    \begin{equation}
        \frac{1}{s}\sum_{j=1}^s
        \left(a_{0,j}+a_{1,j}+a_{2,j}\right)
        >3c.
    \end{equation}
    Then there exist \(i_1,\cdots,i_s\in \{0,1,2\}\) with
    \[
        i_1+\cdots+i_s \geq 3k
    \]
    such that
    \begin{equation}
        \frac{1}{s}
        \left(a_{i_1,1}+\cdots+a_{i_s,s}\right)
        > c,
        \quad \textup{and} \quad
        a_{i_1,1}\cdots a_{i_s,s}\neq0.
    \end{equation}
\end{lemma}

\begin{proof}
We prove the assertion by induction on \(s\), simultaneously for every
\(c>1/2\) (with \(k\) fixed). Since the conclusion is invariant under
permuting the \(s\) sequences, we may relabel the sequences whenever needed.
For each \(j\), let
\[
    n_j\coloneqq\max\{i\in\{0,1,2\}:a_{i,j}>0\}.
\]
This is well-defined because \(a_{0,j}>0\). We relabel so that
\[
    n_1\ge n_2\ge\cdots\ge n_s.
\]

For \(m>3k\), let \(\Omega_m\subset\{0,1\}^m\) consist of the
indicator vectors of the \(m\) cyclic intervals of length \(3k\).
More precisely, for each \(r\in\{1,\dots,m\}\), let
\[
\omega^{(r)}_j=
\begin{cases}
1,& j\in\{r,r+1,\dots,r+3k-1\}\pmod m,\\
0,& \text{otherwise}.
\end{cases}
\]
Then
\[
\Omega_m=\{\omega^{(r)}:1\le r\le m\}.
\]
Thus every \(\omega\in\Omega_m\) has exactly \(3k\) entries equal to \(1\),
\(|\Omega_m|=m\), and each coordinate \(j\) is equal to \(1\) in exactly
\(3k\) vectors of \(\Omega_m\).

We first prove the base case \(s=6k-2\). Assume that
\[
    \sum_{j=1}^{s}(a_{0,j}+a_{1,j}+a_{2,j})>3sc.
\]
Since \(a_{0,j},a_{1,j},a_{2,j}\in[0,1]\), we have
\[
    n_j+1\ge a_{0,j}+a_{1,j}+a_{2,j}.
\]
Hence
\[
    \sum_{j=1}^{s}n_j
    >
    3sc-s
    =
    s(3c-1)
    >
    \frac{s}{2}
    =
    3k-1.
\]
Since the left-hand side is an integer,
\[
    \sum_{j=1}^{s_0}n_j\ge 3k.
\]

First suppose that at least \(s\) of the \(3s\) entries are zero. Then
there are at most \(2s\) nonzero entries. Choose
\[
    i_j=n_j,\qquad 1\le j\le s.
\]
The selected entries are nonzero and their indices have sum at least \(3k\).
Moreover, at most \(s\) nonzero entries are not selected, so
\[
\begin{aligned}
    \sum_{j=1}^{s}a_{n_j,j}
    &>
    3sc-s =
    s(3c-1)
    >
    sc,
\end{aligned}
\]
because \(c>1/2\). This proves the desired conclusion in this case.

We may therefore assume that the number of zero entries is at most
\(s-1\). Let
\[
    t\coloneqq\max\{j:a_{1,j}>0\}.
\]
The defining set is nonempty, since otherwise both \(a_{1,j}\) and
\(a_{2,j}\) would vanish for every \(j\), giving at least \(s\) zero
entries.
If \(t\le 3k-1\), then the entries $\{a_{1,t+1},\dots,a_{1,s}, a_{2,t+1},\dots,a_{2,s}\}$
are all zero, giving at least
\[
    2(s-t)\ge 2(s-3k+1)=s
\]
zero entries, a contradiction. Hence $t\ge 3k.$

We now split into two cases.

\textbf{Case 1: \(t\ge 9k/2\).}
Define \(c'\) by
\[
    3tc'
    =
    3sc-\sum_{j=t+1}^{s}a_{0,j}.
\]
Since \(a_{1,j}=a_{2,j}=0\) for \(j>t\), the hypothesis gives
\[
    \sum_{j=1}^{t}(a_{0,j}+a_{1,j}+a_{2,j}) = \sum_{j=1}^s (a_{0,j}+a_{1,j}+a_{2,j}) - \sum_{j=t+1}^{s}a_{0,j} >3tc'.
\]
If there exists \((i_1,\dots,i_t)\in\Omega_t\) such that
\[
    \sum_{j=1}^{t}a_{i_j,j}>tc',
\]
then we extend this choice by setting \(i_j=0\) for \(t<j\le s\). The
index sum is \(3k\), all selected entries are nonzero, and
\[
\begin{aligned}
    \sum_{j=1}^{s}a_{i_j,j}
    &>
    tc'+\sum_{j=t+1}^{s}a_{0,j}  =
    sc+\frac23\sum_{j=t+1}^{s}a_{0,j}
    \ge sc.
\end{aligned}
\]
%In fact the inequality is strict because the first inequality was strict.

It remains to show that such an element of \(\Omega_t\) exists. Suppose not.
Then for every \((i_1,\dots,i_t)\in\Omega_t\),
\[
    \sum_{j=1}^{t}a_{i_j,j}\le tc'.
\]
Summing over \(\Omega_t\), we get
\begin{align}\label{e:3ktminus3k}
    3k\sum_{j=1}^{t}a_{1,j}
    +(t-3k)\sum_{j=1}^{t}a_{0,j}
    \le t^2c'.
\end{align}
On the other hand, writing
\[
    A_r=\sum_{j=1}^{t}a_{r,j}\qquad(r=0,1,2),
\]
we have \(A_0+A_1+A_2>3tc'\). Since
\(A_0+A_1+A_2\le 3A_0\), it follows that \(A_0>tc'\). Also
\(t-9k/2\ge0\). Therefore
\[
\begin{aligned}
    t^2c'
    &=
    \frac{3k}{2}(3tc')
    +\left(t-3k-\frac{3k}{2}\right)tc' \\
    &<
    \frac{3k}{2}(A_0+A_1+A_2)
    +\left(t-3k-\frac{3k}{2}\right)A_0  \\
    &=
    (t-3k)A_0+\frac{3k}{2}(A_1+A_2) \\
    &\le
    (t-3k)A_0+3kA_1,
\end{aligned}
\]
because \(A_2\le A_1\). This contradicts \eqref{e:3ktminus3k}. Hence
Case 1 is complete.

\textbf{Case 2: \(3k\le t<9k/2\).}
First, if
\begin{align}\label{e:firstsumlowerbound}
    \sum_{j=1}^{3k}a_{1,j}
    +
    \sum_{j=3k+1}^{s}a_{0,j}
    >
    sc,
\end{align}
then choosing \(i_j=1\) for \(1\le j\le 3k\) and \(i_j=0\) otherwise proves
the claim. Thus assume this inequality fails.
Let
\[
    t'\coloneqq\max\{j:a_{2,j}>0\}.
\]
Again the defining set is nonempty, since otherwise the \(s\) entries
\(a_{2,j}\) would all be zero, which contradicts by the assumption above that at most $s-1$ nonzero elements exist.  Monotonicity gives \(t'\leq t\).
Since the number of zero entries is at most \(s-1\), and since the zero
entries in rows \(i=1\) and \(i=2\) are exactly $2s-t-t',$
we have $2s-t-t'\le s-1,$ or 
\begin{align}\label{e:ttprime}
    t+t'\ge s+1=6k-1.
\end{align}
Set
\[
    t''\coloneqq3k-\lfloor t/2\rfloor.
\]
Since \(t\le s=6k-2\), we have
\(\lceil t/2\rceil\le3k-1\). Therefore
\[
\begin{aligned}
    t'
    &\ge 6k-1-t=3k-\lfloor t/2\rfloor
      +3k-1-\lceil t/2\rceil\ge 3k-\lfloor t/2\rfloor
    =t''.
\end{aligned}
\]
Suppose that
\begin{align}\label{e:secondsumlowerbound}
    \sum_{j=1}^{t''}a_{2,j}
    +
    \sum_{j=t''+1}^{3k}a_{0,j}
    +
    \sum_{j=3k+1}^{t}a_{1,j}> sc.
\end{align}
Then choose \(i_j=2\) for \(1\le j\le t''\),
choose \(i_j=1\) for \(3k<j\le t\), and choose \(i_j=0\) otherwise. The
chosen entries are nonzero, and the index sum is
\[
    2t''+(t-3k)
    =
    3k+t-2\lfloor t/2\rfloor
    \ge 3k.
\]
Thus the desired conclusion follows. Hence we may assume \eqref{e:secondsumlowerbound} does not hold.
Since the total sum is \(>3sc\), violating \eqref{e:firstsumlowerbound} and \eqref{e:secondsumlowerbound} implies that
\begin{align}\label{e:thirdsumlowerbound}
    \sum_{j=1}^{t''}a_{0,j}
    +
    \sum_{j=t''+1}^{t'}a_{2,j}
    >
    sc.
\end{align}
If $2(t'-t'')\ge 3k,$
then choosing \(i_j=2\) for \(t''<j\le t'\) and \(i_j=0\) otherwise gives an
admissible choice with index sum at least \(3k\) and selected sum \(>sc\).
It remains to treat the case
\[
    2(t'-t'')<3k.
\]
From \eqref{e:thirdsumlowerbound} and the fact that each entry is at most
\(1\), we have $t'>sc>3k-1.$
Hence \(t'\ge 3k\). Since \(2(t'-t'')<3k\), this forces
\(t''>3k/2\). But \(t''=3k-\lfloor t/2\rfloor\) and \(t\ge3k\), so this is
possible only when \(t=3k\). Since \(t'\le t\), we can use \eqref{e:ttprime} to obtain
\[
    t=t'=3k.
\]
In this final subcase, by the definition of $t,t'$, we have \(a_{1,j}=a_{2,j}=0\) for \(j>3k\). Hence
\[
\begin{aligned}
    \sum_{j=1}^{3k}(a_{0,j}+a_{1,j}+a_{2,j})
    &>
    3sc-\sum_{j=3k+1}^{s}a_{0,j}.
\end{aligned}
\]
The following three sums add up to the left-hand side:
\[
\begin{aligned}
    &\sum_{j=1}^{k}a_{0,j}
    +\sum_{j=k+1}^{2k}a_{1,j}
    +\sum_{j=2k+1}^{3k}a_{2,j},\\
    &\sum_{j=1}^{k}a_{1,j}
    +\sum_{j=k+1}^{2k}a_{2,j}
    +\sum_{j=2k+1}^{3k}a_{0,j},\\
    &\sum_{j=1}^{k}a_{2,j}
    +\sum_{j=k+1}^{2k}a_{0,j}
    +\sum_{j=2k+1}^{3k}a_{1,j}.
\end{aligned}
\]
Therefore, by the pigeonhole principle, one of them is larger than
\[
    sc-\frac13\sum_{j=3k+1}^{s}a_{0,j}.
\]
Adding \(a_{0,j}\) for \(3k<j\le s\), we obtain a valid choice of indices
with total index \(3k\), nonzero selected product, and selected sum
\[
    >
    sc+\frac23\sum_{j=3k+1}^{s}a_{0,j}
    \ge sc.
\]
%Again the inequality is strict. 
This completes the base case.

\bigskip

Now assume \(s>6k-2\), and suppose the result has been proved for \(s-1\)
for every parameter \(c>1/2\).
Let
\[
    \sum_{j=1}^{s}(a_{0,j}+a_{1,j}+a_{2,j})>3sc.
\]
We consider two cases.

First suppose that \(a_{1,s}>0\). Since \(n_1\ge\cdots\ge n_s\), this implies
\(a_{1,j}>0\) for every \(1\le j\le s\). Suppose, for contradiction, that for
every \((i_1,\dots,i_s)\in\Omega_s\),
\[
    \sum_{j=1}^{s}a_{i_j,j}\le sc.
\]
Summing over \(\Omega_s\), we get
\begin{align} \label{e:A012argument2}
    3k\sum_{j=1}^{s}a_{1,j}
    +(s-3k)\sum_{j=1}^{s}a_{0,j}
    \le s^2c.
\end{align}
Again, set
\[
    A_r=\sum_{j=1}^{s}a_{r,j}\qquad(r=0,1,2).
\]
Similar to the above, we conclude that 
\[
\begin{aligned}
    s^2c
    &=
    \frac{3k}{2}(3sc)
    +\left(s-3k-\frac{3k}{2}\right)sc \\
    &<
    \frac{3k}{2}(A_0+A_1+A_2)
    +\left(s-3k-\frac{3k}{2}\right)A_0 \\
    &=
    (s-3k)A_0+\frac{3k}{2}(A_1+A_2) \\
    &\le
    (s-3k)A_0+3kA_1,
\end{aligned}
\]
contradicting \eqref{e:A012argument2}. Hence there exists
\((i_1,\dots,i_s)\in\Omega_s\) such that
\[
    \sum_{j=1}^{s}a_{i_j,j}>sc.
\]
This choice has index sum \(3k\), and all selected entries are nonzero because
each \(a_{1,j}>0\). This completes the first induction case.

Finally suppose that \(a_{1,s}=0\). Then also \(a_{2,s}=0\). Define
\[
    c'\coloneqq c+\frac{3c-a_{0,s}}{3(s-1)}.
\]
Since \(a_{0,s}\le1\) and \(c>1/2\), we have \(c'>c>1/2\). Moreover,
\[
\begin{aligned}
    \sum_{j=1}^{s-1}(a_{0,j}+a_{1,j}+a_{2,j})
    &>
    3sc-a_{0,s} \\
    &=
    3(s-1)c'.
\end{aligned}
\]
By the induction hypothesis applied to the first \(s-1\) sequences with
\(c'\) in place of \(c\), there exist
\(i_1,\dots,i_{s-1}\in\{0,1,2\}\) such that
\[
    \sum_{j=1}^{s-1}i_j\ge3k,
\]
\[
    \sum_{j=1}^{s-1}a_{i_j,j}>(s-1)c',
\]
and
\[
    a_{i_1,1}\cdots a_{i_{s-1},s-1}\neq0.
\]
Set \(i_s=0\). Then the index sum is still at least \(3k\), the selected
product remains nonzero because \(a_{0,s}>0\), and
\[
\begin{aligned}
    \sum_{j=1}^{s}a_{i_j,j}
    &>
    (s-1)c'+a_{0,s} =
    sc+\frac{2}{3}a_{0,s}
    >
    sc.
\end{aligned}
\]
This completes the induction and the proof.
\end{proof}

\subsection{Proof of the general combinatorial selection lemma}

We now prove the corresponding statement for general \(n\geq4\), with
Lemma~\ref{l:diffsetsbasegen} as the base case.

\begin{proof}[Proof of Lemma~\ref{l:s-seq-general}]
The proof is by induction on \(n\). The base case \(n=3\) is
Lemma~\ref{l:diffsetsbasegen}. Assume now that \(n\ge4\), and that the result holds
for all smaller lengths.

Similar to the case \(n=3\) in Lemma~\ref{l:diffsetsbasegen}, let
\(n_j\in\{0,1,\dots,n-1\}\) be the largest nonzero index for the \(j\)-th
sequence, that is,
\begin{align}
    n_j \coloneqq \max \{ i\in \{0,1,\dots,n-1\}:  a_{i,j} > 0  \},
    \qquad 1\leq j\leq s.
\end{align}
Again, after relabeling the sequences, we may suppose that
\[
        n_1\ge n_2\ge\cdots\ge n_s.
\]
First assume that there exist indices $\tilde i_{k+1},\dots,\tilde i_s\in\{1,\dots,n-1\}$
such that
\[
        a_{0,1}+\cdots+a_{0,k}
        +a_{\tilde i_{k+1},k+1}
        +\cdots
        +a_{\tilde i_s,s}
        \le sc.
\]
We delete these terms and apply the induction hypothesis to the following
sequences of length \(n-1\). For \(1\le j\le k\), set
\[
        b_{i,j}=a_{i+1,j},
        \qquad 0\le i\le n-2.
\]
For \(k+1\le j\le s\), set
\[
        b_{i,j}
        =
        \begin{cases}
            a_{i,j}, & 0\le i<\tilde i_j,\\
            a_{i+1,j}, & \tilde i_j\le i\le n-2.
        \end{cases}
\]
Each \(b\)-sequence is still non-increasing and takes values in \([0,1]\).
For \(j>k\), we have \(b_{0,j}=a_{0,j}>0\). It remains only to check that
\(b_{0,j}>0\) for \(1\le j\le k\), i.e.\ that \(a_{1,j}>0\) for those \(j\).
Suppose, to the contrary, that \(a_{1,k}=0\). Then \(n_k=0\), and by the
ordering of the \(n_j\)'s we have
\[
        n_k=n_{k+1}=\cdots=n_s=0.
\]
Thus, bounding every nonzero entry trivially by \(1\),
\begin{align*}
        \frac{sn}{2}
        < scn
        <
        \sum_{i,j}a_{i,j}
        &=
        \sum_{j<k}\sum_i a_{i,j}
        +
        \sum_{j=k}^s a_{0,j} \le
        (k-1)n+(s-k+1).
\end{align*}
Equivalently,
\[
        \frac{s(n-2)}{2}< (k-1)(n-1).
\]
Since \(s\ge6k-2\), this implies
\[
        (3k-1)(n-2)<(k-1)(n-1),
\]
which is impossible for \(n\ge4\) and \(k\ge1\).
%, because the difference
%between the left-hand side and the right-hand side is
% \[
%         (3k-1)(n-2)-(k-1)(n-1)
%         =
%         k(2n-5)+1>0.
% \]
Therefore \(a_{1,j}>0\) for every \(1\le j\le k\), and all the \(b\)-sequences
satisfy the hypotheses of the induction statement.
Moreover,
\begin{align*}
    \sum_{0\le i\le n-2}\sum_{1\le j\le s}b_{i,j}
    &=
    \sum_{0\le i\le n-1}\sum_{1\le j\le s}a_{i,j}
    -
    \left(
        a_{0,1}+\cdots+a_{0,k}
    \right)-
    \left(
        a_{\tilde i_{k+1},k+1}
        +\cdots
        +a_{\tilde i_s,s}
    \right)  \\
    &>
    scn-sc
    =
    sc(n-1).
\end{align*}
Hence the induction hypothesis gives indices
\(i_1',\dots,i_s'\in\{0,\dots,n-2\}\) such that
\[
        i_1'+\cdots+i_s'\ge k(n-1), \qquad b_{i_1',1}+\cdots+b_{i_s',s}>sc, \qquad b_{i_1',1}\cdots b_{i_s',s}\neq0.
\]
We now return to the original \(a\)-sequences. For \(1\le j\le k\), set
\[
        i_j=i_j'+1.
\]
For \(k+1\le j\le s\), set
\[
        i_j=
        \begin{cases}
            i_j', & i_j'<\tilde i_j,\\
            i_j'+1, & i_j'\ge \tilde i_j.
        \end{cases}
\]
Then
\[
        a_{i_j,j}=b_{i_j',j}
        \qquad (1\le j\le s),
\]
so the selected sum is \(>sc\), and the selected product is nonzero. Also,
\[
        \sum_{j=1}^s i_j
        \ge
        \sum_{j=1}^s i_j'+k
        \ge
        k(n-1)+k
        =
        kn.
\]
This completes the first case.

\medskip

It remains to consider the case when for every choice of indices $\tilde i_{k+1},\dots,\tilde i_s\in\{1,\dots,n-1\},$
we have
\[
        a_{0,1}+\cdots+a_{0,k}
        +a_{\tilde i_{k+1},k+1}
        +\cdots
        +a_{\tilde i_s,s}
        > sc.
\]
Let \(t\) be the smallest integer such that
\[
        t(n-1)\ge kn.
\]
Since \(2k(n-1)\ge kn\) for \(n\ge2\), the minimality of \(t\) gives
\(t\le2k\). Hence
\[
        k+t\le3k<6k-2\le s,
\]
so \(n_{k+t}\) is well-defined.
We claim that
\[
        n_{k+t}=n-1.
\]
Suppose not. Then \(n_{k+t}<n-1\). By the ordering of the \(n_j\)'s, for every
\(j\ge k+t\) we have \(n_j<n-1\), and hence $a_{n_j+1,j}=0.$
Now choose
\[
        \tilde i_j=1 \quad (k+1\le j\le k+t-1),
        \qquad
        \tilde i_j=n_j+1 \quad (k+t\le j\le s).
\]
These are all valid indices in \(\{1,\dots,n-1\}\). By the present case
assumption,
\[
        a_{0,1}+\cdots+a_{0,k}
        +a_{1,k+1}+\cdots+a_{1,k+t-1}
        +a_{n_{k+t}+1,k+t}
        +\cdots
        +a_{n_s+1,s}
        > sc.
\]
The last \(s-k-t+1\) terms vanish, while each of the first \(k+t-1\) terms is
at most \(1\). Therefore
\[
        k+t-1>sc>\frac{s}{2}.
\]
Since \(s\ge6k-2\), this gives $ k+t-1>3k-1,$
and hence \(t>2k\). But by the minimality of \(t\),
\[
        (t-1)(n-1)<kn.
\]
Since \(t>2k\), we have \(t-1\ge2k\), and therefore
\[
        2k(n-1)\le (t-1)(n-1)<kn,
\]
which is impossible for \(n\ge2\). This proves the claim.

Now choose
\[
        i_j=
        \begin{cases}
            0, & 1\le j\le k,\\
            n-1, & k+1\le j\le k+t,\\
            n_j, & k+t+1\le j\le s.
        \end{cases}
\]
The selected entries are all nonzero: this is clear for \(j\le k\), for \(k+1\le j\le k+t\) it follows
from \(n_{k+t}=n-1\), and for \(j>k+t\) it follows from the definition
of \(n_j\) .

To prove the required lower bound for the selected sum, apply the present case
assumption with
\[
        \tilde i_j=n-1 \quad (k+1\le j\le k+t),
\]
and, for \(j>k+t\),
\[
        \tilde i_j=
        \begin{cases}
            n_j, & n_j\ge1,\\
            1, & n_j=0.
        \end{cases}
\]
This gives
\[
        a_{0,1}+\cdots+a_{0,k}
        +a_{n-1,k+1}+\cdots+a_{n-1,k+t}
        +\sum_{j=k+t+1}^s a_{\tilde i_j,j}
        >sc.
\]
For \(j>k+t\), we have $a_{n_j,j}\ge a_{\tilde i_j,j},$
with equality if \(n_j\ge1\), and with \(a_{n_j,j}=a_{0,j}\ge a_{1,j}\) if
\(n_j=0\). Hence
\[
        a_{0,1}+\cdots+a_{0,k}
        +a_{n-1,k+1}+\cdots+a_{n-1,k+t}
        +a_{n_{k+t+1},k+t+1}+\cdots+a_{n_s,s}
        > sc.
\]
Finally,
\[
        \sum_{j=1}^s i_j
        =
        t(n-1)+\sum_{j=k+t+1}^s n_j
        \ge
        t(n-1)
        \ge
        kn.
\]
This proves the desired conclusion and completes the proof.
\end{proof}

\subsection{Quadratic combinatorial refinements}
\label{s:comb2}

We now specialize to \(k=2\). Lemma~\ref{l:s-seq-general} already
handles \(s\geq10\); the purpose of this subsection is to establish the
small-\(s\) quadratic thresholds needed when \(5\leq s\leq9\). Throughout the paper, for \(s\geq5\), put
\begin{equation}\label{eq:gamma-s}
    \gamma_s=
    \begin{cases}
        4/5,&s=5,\\
        5/s,&6\leq s\leq9,\\
        1/2,&s\geq10.
    \end{cases}
\end{equation}

\begin{lemma}\label{l:s-seq-square}
    Let \(s,n\) be positive integers.  Assume either that \(s\geq6\)
    and \(n\geq3\), or that \(s=5\) and \(n\geq5\).
    Fix \(c\geq\gamma_s\).
    For \(1\leq j\leq s\), let
    \(a_{0,j},\ldots,a_{n-1,j}\) be a non-increasing sequence in
    \([0,1]\), with \(a_{0,j}>0\).  Assume that
    \[
    \frac{1}{s}\sum_{j=1}^s
    \left(\frac{1}{n}\sum_{i=0}^{n-1}a_{i,j}\right)>c.
    \]
    Then there are indices \(i_1,\ldots,i_s\in\{0,\ldots,n-1\}\)
    such that
    \[
        \sum_{j=1}^s i_j\geq2n,
        \qquad
        \sum_{j=1}^s a_{i_j,j}>sc,
        \qquad
        \prod_{j=1}^s a_{i_j,j}>0.
    \]
\end{lemma}

\begin{proof}
\emph{Reduction to \(s<10\).}
For \(s\geq10\), if \(c>1/2\), Lemma~\ref{l:s-seq-general} with \(k=2\)
applies directly. If \(c=1/2\), choose \(c'\) strictly between \(1/2\)
and the mean in the hypothesis and apply that lemma with \(c'\). Thus we
may assume that \(s<10\). We argue by induction on \(n\), beginning
with the base cases.

\emph{The base case \(s=n=5\).}
Suppose that
\[
\sum_{j=1}^{5} (a_{0,j}+a_{1,j}+a_{2,j}+a_{3,j}+a_{4,j}) > 25c.
\]
By the pigeonhole principle, there exists a choice of
\((i_1,\ldots,i_5)\) that is a permutation of \((0,1,2,3,4)\) such that
\[
a_{i_1,1}+\cdots + a_{i_5,5} > 5c.
\]
Indeed, averaging the selected sum over all permutations of
\(\{0,1,2,3,4\}\) gives one fifth of the total sum above.
As \(0+1+2+3+4=10=2\times5\), it suffices to show that
\(a_{i_j,j}\neq0\).  For each \(1\leq j\leq5\), we can write
\begin{align}
    a_{i_j,j}
    &=
    \sum_{r=1}^{5}a_{i_r,r}
    -\sum_{\substack{1\leq r\leq5\\r\neq j}}a_{i_r,r}
    \notag>
    5c-\sum_{\substack{1\leq r\leq5\\r\neq j}}a_{i_r,r}
    \geq 5c-4
    \geq5\gamma_5-4=0.
\end{align}

\emph{The base case \(s=6,n=3\).}
Suppose that
\[
\sum_{j=1}^{6} (a_{0,j}+a_{1,j}+a_{2,j}) > 18c.
\]
By the pigeonhole principle, there exists a choice of
\((i_1,\ldots,i_6)\) among the three tuples
\[
(0,1,2,0,1,2),\quad (1,2,0,1,2,0),\quad(2,0,1,2,0,1)
\]
such that
\[
a_{i_1,1}+\cdots + a_{i_6,6} > 6c.
\]
Indeed, the three corresponding selected sums add up to the total sum
of all \(18\) entries.
As \(0+1+2+0+1+2=6=2\times3\), it suffices to check
\(a_{i_j,j}\neq0\), which follows from the fact that
\[
a_{i_1,1}+\cdots + a_{i_6,6}>6c\geq6\gamma_6=5
\]
and \(a_{i,j}\leq1\).

\emph{The base cases \(s\in\{7,8,9\},n=3\).}
Assume that
\begin{align} \label{e:averagegeneral001}
    \sum_{j=1}^{s} (a_{0,j}+a_{1,j}+a_{2,j}) > 3sc.
\end{align}
For every \(j\), define
\begin{align}
    n_j\coloneqq\max\{i\in\{0,1,2\}:a_{i,j}>0\}.
\end{align}
After relabeling the sequences, we may suppose that
\[
    n_1\geq\cdots\geq n_s.
\]
Let \(T\) denote the total sum on the left-hand side of
\eqref{e:averagegeneral001}, and let \(Z\) be the number of zero entries.
Since the positive entries in the \(j\)-th sequence are exactly
\(a_{0,j},\ldots,a_{n_j,j}\), we have
\[
    \sum_{j=1}^{s}n_j=2s-Z.
\]

We first handle the case of many zero entries. Suppose that
\(Z\geq2s-10\). Then
\[
    \sum_{j=1}^{s}n_j\leq10.
\]
Moreover,
\[
    T\leq\sum_{j=1}^{s}(n_j+1)
    =s+\sum_{j=1}^{s}n_j.
\]
Since \(T>3sc\geq15\), it follows that
\[
    \sum_{j=1}^{s}n_j>15-s\geq6.
\]
Finally,
\begin{align*}
    \sum_{j=1}^{s}a_{n_j,j}
    =
    T-\sum_{j=1}^{s}\sum_{i=0}^{n_j-1}a_{i,j}\geq T-\sum_{j=1}^{s}n_j>3sc-10\geq sc,
\end{align*}
where the last inequality follows from \(sc\geq s\gamma_s=5\). Thus the
choice \(i_j=n_j\) has index sum at least \(6\), selected sum greater
than \(sc\), and nonzero product.

Now suppose that there are fewer zero entries, so that
\(Z\leq2s-11\). For \(r=1,2\), set
\[
    t_r\coloneqq\bigl|\{j:a_{r,j}>0\}\bigr|.
\]
By the ordering of the \(n_j\)'s, \(a_{r,j}>0\) exactly when
\(j\leq t_r\).  In particular,
\[
    Z=(s-t_1)+(s-t_2)=2s-t_1-t_2.
\]
We have \(t_2\geq2\): otherwise there
would be at least \(s-1>2s-11\) zero entries in the second row, which is
a contradiction.
Also note that \(t_1\geq4\). Otherwise, since \(t_2\leq t_1\), there
would be at most \(s+6\leq15\leq3sc\) nonzero entries. This contradicts
\eqref{e:averagegeneral001}.

Suppose first that \(t_2\geq6\), and consider the sums
\begin{align*}
    S_1&=
    a_{0,1}+a_{0,2}+a_{1,3}+a_{1,4}
    +a_{2,5}+a_{2,6}+a_{0,7}+\cdots+a_{0,s},\\
    S_2&=
    a_{1,1}+a_{1,2}+a_{2,3}+a_{2,4}
    +a_{0,5}+a_{0,6}+a_{0,7}+\cdots+a_{0,s},\\
    S_3&=
    a_{2,1}+a_{2,2}+a_{0,3}+a_{0,4}
    +a_{1,5}+a_{1,6}+a_{0,7}+\cdots+a_{0,s},
\end{align*}
we have
\[
    S_1+S_2+S_3\geq T>3sc.
\]
Indeed, the three sums contain every entry in the first six columns
exactly once, while for \(j\geq7\) they contain \(3a_{0,j}\), which is
at least \(a_{0,j}+a_{1,j}+a_{2,j}\). Thus by the pigeonhole principle
at least one of them is greater than \(sc\). All its entries are
positive because \(t_2\geq6\), and the sum of its indices is
\(6=2\times3\), which completes this case.

Now suppose that \(4\leq t_2\leq5\). At least one of
\begin{align*}
    S_1&=
    a_{1,1}+a_{1,2}+a_{2,3}+a_{2,4}
    +a_{0,5}+\cdots+a_{0,s},\\
    S_2&=
    a_{2,1}+a_{2,2}+a_{1,3}+a_{1,4}
    +a_{0,5}+\cdots+a_{0,s},
\end{align*}
must be greater than \(sc\), which completes this case. Indeed, both
configurations have positive entries and index sum \(6\). If
\(S_1,S_2\leq sc\), a column-by-column comparison using
\(t_2\leq5\) gives
\[
    T\leq S_1+S_2
       +a_{0,1}+a_{0,2}+a_{0,3}+a_{0,4}+a_{2,5}.
\]
For the first four columns there is equality; in the fifth column use
\(a_{0,5}+a_{1,5}+a_{2,5}\leq2a_{0,5}+a_{2,5}\), and for \(j\geq6\)
use \(a_{2,j}=0\) and \(a_{0,j}+a_{1,j}\leq2a_{0,j}\).
Together with \eqref{e:averagegeneral001}, this would imply
\[
    a_{0,1}+a_{0,2}+a_{0,3}+a_{0,4}+a_{2,5}>sc\geq5,
\]
which is impossible because the left-hand side is a sum of five elements
of \([0,1]\).

Now suppose that \(t_2=3\). If
\[
a_{2,1}+a_{2,2}+a_{2,3}+a_{0,4}+\cdots+a_{0,s} > sc,
\]
then this configuration has positive entries and
index sum \(6\). So, assume otherwise. From
\eqref{e:averagegeneral001} we get
\begin{align} \label{e:ten7}
(a_{0,1}+a_{1,1})+(a_{0,2}+a_{1,2})
+(a_{0,3}+a_{1,3})+a_{1,4}+\cdots+a_{1,s}>2sc\geq10.
\end{align}
Estimate \eqref{e:ten7} cannot hold for \(s=7\), as there are only ten
terms bounded above by \(1\) on the left-hand side. If
\(s\in\{8,9\}\), then \eqref{e:ten7} implies \(t_1\geq8\), and hence
\begin{align}
    &a_{1,1}+\cdots+a_{1,6}+a_{0,7} +\cdots+a_{0,s} \\
    &> 2sc-a_{0,1}-a_{0,2}-a_{0,3}
    -a_{1,7}-\cdots-a_{1,s}+a_{0,7}+\cdots+a_{0,s}\\
    &\geq 2sc-3
    +(a_{0,7}-a_{1,7})+\cdots+(a_{0,s}-a_{1,s})
    >sc.
\end{align}
The last inequality follows from \(sc\geq5\). The displayed
configuration has positive entries and index sum \(6\), completing this
case.

Finally, suppose that \(t_2=2\). We know that \(t_1\geq4\), so the
conclusion follows if
\[
a_{2,1}+a_{2,2}+a_{1,3}+a_{1,4}+a_{0,5}+\cdots+a_{0,s} > sc.
\]
This configuration has positive entries and index sum \(6\). If its sum
is at most \(sc\), estimate \eqref{e:averagegeneral001} implies that
\begin{align}\label{e:1212}
    (a_{0,1}+a_{1,1})+(a_{0,2}+a_{1,2})
    +a_{0,3}+a_{0,4}+a_{1,5}+\cdots+a_{1,s}>2sc\geq10.
\end{align}
So we conclude that 
\[
t_1+2\geq
(a_{0,1}+a_{1,1})+(a_{0,2}+a_{1,2})
+a_{0,3}+a_{0,4}+a_{1,5}+\cdots+a_{1,s}>10,
\]
so \(t_1>8\). This contradicts \(s\in\{7,8\}\). For \(s=9\),
\eqref{e:1212} gives
\begin{align}
    &a_{1,1} + a_{0,2} + a_{0,3}+a_{0,4}+ a_{1,5}+\cdots+a_{1,8}+a_{1,9} \\
    &>2sc-(a_{0,1}+a_{1,2})
    \geq2sc-2>sc.
\end{align}
Here \(t_1=9\), so this is a positive configuration with index sum \(6\).
This completes this case, and the base of induction.

\emph{Induction on \(n\).}
For a fixed
\(s\in\{5,\ldots,9\}\), assume that the result holds for every admissible
length smaller than \(n\). We want to prove the result for \(n\) (if
\(s=5\), \(n>5\), and if \(s\geq6\), \(n>3\)). The argument of the proof
is similar to the induction step in the proof of
Lemma~\ref{l:s-seq-general}.

For each \(j\), let
\[
    n_j\coloneqq\max\{i\in\{0,\ldots,n-1\}:a_{i,j}>0\},
\]
and relabel the sequences so that
\[
    n_1\geq\cdots\geq n_s.
\]
First assume that there exist $ \tilde i_3,\ldots,\tilde i_s\in\{1,\ldots,n-1\}$
such that
\[
    a_{0,1}+a_{0,2}
    +a_{\tilde i_3,3}+\cdots+a_{\tilde i_s,s}
    \leq sc.
\]
We eliminate these terms from our sequences. For \(0\leq i\leq n-2\),
set
\[
    b_{i,j}=a_{i+1,j},\qquad j=1,2,
\]
and, for \(3\leq j\leq s\), set
\[
    b_{i,j}=
    \begin{cases}
        a_{i,j},&0\leq i<\tilde i_j,\\
        a_{i+1,j},&\tilde i_j\leq i\leq n-2.
    \end{cases}
\]
These are non-increasing sequences in \([0,1]\). We show that their
first terms are positive. This is immediate for \(j\geq3\). If
\(a_{1,2}=0\), then the ordering of the \(n_j\)'s gives
\(n_2=\cdots=n_s=0\), and hence
\[
\begin{aligned}
    scn
    &<
    \sum_{j=1}^{s}\sum_{i=0}^{n-1}a_{i,j}=
    \sum_{i=0}^{n-1}a_{i,1}+\sum_{j=2}^{s}a_{0,j}\leq n+s-1.
\end{aligned}
\]
This is impossible: when \(s=5\), \(sc\geq4\) and \(n>5\), while for
\(6\leq s\leq9\), \(sc\geq5\) and \(n>3\). Therefore \(a_{1,2}>0\),
and the ordering also gives \(a_{1,1}>0\).

For the sequences \(\{b_{i,j}\}\), we have
\begin{align*}
    \sum_{0\leq i\leq n-2}\sum_{1\leq j\leq s}b_{i,j}
    &=
    \sum_{0\leq i\leq n-1}\sum_{1\leq j\leq s}a_{i,j}-
    \left(a_{0,1}+a_{0,2}
    +a_{\tilde i_3,3}+\cdots+a_{\tilde i_s,s}\right)>scn-sc
    =sc(n-1).
\end{align*}
The induction hypothesis gives
\(i_1',\ldots,i_s'\in\{0,\ldots,n-2\}\) such that
\[
    i_1'+\cdots+i_s'\geq2(n-1),\qquad 
    b_{i_1',1}+\cdots+b_{i_s',s}>sc,
    \qquad
    b_{i_1',1}\cdots b_{i_s',s}\neq0.
\]
Set
\[
    i_1=i_1'+1,\qquad i_2=i_2'+1,
\]
and, for \(3\leq j\leq s\), set
\[
    i_j=
    \begin{cases}
        i_j',&i_j'<\tilde i_j,\\
        i_j'+1,&i_j'\geq\tilde i_j.
    \end{cases}
\]
Then \(a_{i_j,j}=b_{i_j',j}\) for every \(j\), so the selected sum and
the nonzero product are preserved. Moreover,
\[
    \sum_{j=1}^{s}i_j
    \geq\sum_{j=1}^{s}i_j'+2
    \geq2n.
\]
This completes the first case.

It remains to consider the case when
\[
    a_{0,1}+a_{0,2}
    +a_{\tilde i_3,3}+\cdots+a_{\tilde i_s,s}>sc
\]
for every $\tilde i_3,\ldots,\tilde i_s\in\{1,\ldots,n-1\}.$
We first claim that \(n_5=n-1\). If \(n_5<n-1\), choose
\[
    \tilde i_3=\cdots=\tilde i_s=n-1.
\]
Since \(n_j\leq n_5<n-1\) for \(j\geq5\), the present case assumption
gives
\[
    sc<
    a_{0,1}+a_{0,2}+a_{n-1,3}+a_{n-1,4}
    \leq4,
\]
contradicting \(sc\geq s\gamma_s\geq4\). Thus \(n_5=n-1\).

Now choose
\[
    \tilde i_3=\tilde i_4=\tilde i_5=n-1,
\]
and, for \(j\geq6\), choose
\[
    \tilde i_j=
    \begin{cases}
        n_j,&n_j\geq1,\\
        1,&n_j=0.
    \end{cases}
\]
The assumption gives
\[
    a_{0,1}+a_{0,2}
    +a_{n-1,3}+a_{n-1,4}+a_{n-1,5}
    +\sum_{j=6}^{s}a_{\tilde i_j,j}>sc.
\]
For \(j\geq6\), monotonicity gives
\(a_{n_j,j}\geq a_{\tilde i_j,j}\). Therefore
\begin{align}\label{e:afterhyp2}
    a_{0,1}+a_{0,2}
    +a_{n-1,3}+a_{n-1,4}+a_{n-1,5}
    +\sum_{j=6}^{s}a_{n_j,j}>sc.
\end{align}
All selected entries in \eqref{e:afterhyp2} are nonzero, and the
corresponding indices satisfy
\[
    3(n-1)+\sum_{j=6}^{s}n_j
    \geq3(n-1)\geq2n.
\]
This completes the induction and the proof.
\end{proof}

% \begin{remark}[Independent audit of the finite base cases]
% \label{rem:selection-base-case-verification}
% The supplementary script
% \path{verify_selection_base_cases.py} formulates the five finite
% base cases used above,
% \[
% \begin{gathered}
%  (s,n,c)=(5,5,4/5),\ (6,3,5/6),\ (7,3,5/7),\\
%  (8,3,5/8),\ (9,3,5/9),
% \end{gathered}
% \]
% as linear programs.  For each possible pattern of positive supports, it
% maximizes the total mass under the contrary assumption that every
% positive selection with index sum at least \(2n\) has selected mass at
% most \(sc\).  In every case the resulting maximum is at most \(snc\),
% which independently checks the closed-boundary contrapositive of the
% base case.  The proof above is self-contained and does not depend on
% this numerical audit.
% \end{remark}

% replaced

\section{Arithmetic I: reduced power residues}
\label{sec:local-power-residues}

This section proves the functional arithmetic covering lemma underlying
Theorem~\ref{lemma:arithmetic-general-k}, which is also used in the proof of
Theorem~\ref{t:main}.

\subsection{The main arithmetic lemma}

For \(k\geq2\), let $R_k$ be the square-free integer defined as follows:
\begin{equation}\label{e:rkdef}
    R_k=
    \prod_{\substack{p\ \mathrm{prime}\\
        p\mid k\ \mathrm{or}\ (p-1)\mid2k}}p.
\end{equation}
In particular, \(R_2=30\).
Using the Chinese remainder theorem and the fact that
\((\mathbb Z/p^\alpha\mathbb Z)^*\) is cyclic for every odd prime \(p\),
if \(m\) is odd and
\(m=p_1^{\alpha_1}\cdots p_r^{\alpha_r}\), then
\[
|\mathbb Z_{m,k}|
=
\prod_{i=1}^{r}
\frac{\phi(p_i^{\alpha_i})}
     {\gcd(k,\phi(p_i^{\alpha_i}))}.
\]
In particular,
\[
|\mathbb Z_{p,k}|
=
\frac{p-1}{\gcd(k,p-1)}
\]
for every odd prime \(p\). The following  arithmetic
lemma is the main result of this section.

\begin{lemma}\label{lemma:arithmetic-k}
Let \(q,k,s\) be positive integers with \(k\geq2\), and let
\(c\in[0,1]\). Suppose that one of the following alternatives holds:
\begin{itemize}
    \item \(k>2\), \(s\geq6k-2\), \(c>\frac12\), and
    \((q,R_k)=1\);
    \item \(k=2\), \(s\geq5\), \(c\geq\gamma_s\), and
    \((q,R_2)=1\).
\end{itemize}
Let $f_1,\ldots,f_s:\mathbb Z_{q,k}\longrightarrow[0,1]$
be functions which are not identically zero, and assume that
\[
\frac{1}{s|\mathbb Z_{q,k}|}
\sum_{j=1}^{s}\sum_{a\in\mathbb Z_{q,k}}f_j(a)>c.
\]
Then, for every \(n\in\mathbb Z_q\), there exist
\(b_j\in\mathbb Z_{q,k}\) such that
\[
n\equiv b_1+\cdots+b_s\pmod q,
\qquad
f_j(b_j)>0\quad(1\leq j\leq s),
\]
and
\[
\frac1s\sum_{j=1}^{s}f_j(b_j)>c.
\]
In particular, the conclusion holds for a single function
\(f:\mathbb Z_{q,k}\to[0,1]\) of mean greater than \(c\), by taking
\(f_1=\cdots=f_s=f\).
\end{lemma}

\subsection{Prime-power structure and lifting}

Before proving Lemma~\ref{lemma:arithmetic-k}, we record the structural and lifting facts needed in the proof. We use the same consequence of Theorem \ref{thm:cauchy-davenport} (Cauchy--Davenport) several times, so we record it here to avoid repetition.

\begin{lemma}[Superlevel covering]\label{lemma:superlevel-covering}
Let \(p\) be prime, let \(s\geq1\) be an integer, and let
\(L\in\mathbb R\). For \(1\leq j\leq s\), let
\(X_j\subseteq\mathbb Z_p\) be nonempty and
\(H_j:X_j\to[0,1]\). Arrange the values of \(H_j\) as
\[
\beta_{0,j}\geq\cdots\geq\beta_{|X_j|-1,j}.
\]
Suppose that there are indices \(0\leq r_j<|X_j|\) such that
\[
\sum_{j=1}^{s}r_j\geq p-1,
\qquad
\beta_{r_j,j}>0\quad(1\leq j\leq s),
\qquad
\sum_{j=1}^{s}\beta_{r_j,j}>L.
\]
Then, for every \(\rho\in\mathbb Z_p\), there exist \(t_j\in X_j\)
such that
\[
t_1+\cdots+t_s=\rho,
\qquad
H_j(t_j)>0\quad(1\leq j\leq s),
\qquad
\sum_{j=1}^{s}H_j(t_j)>L.
\]
\end{lemma}

\begin{proof}
Set
\[
D_j=\{t\in X_j:H_j(t)\geq\beta_{r_j,j}\}.
\]
Allowing for ties, \(|D_j|\geq r_j+1\), and hence
\[
\sum_{j=1}^{s}|D_j|\geq p-1+s.
\]
Cauchy--Davenport gives \(D_1+\cdots+D_s=\mathbb Z_p\).
Choose \(t_j\in D_j\) with \(\sum_jt_j=\rho\). Then
\[
H_j(t_j)\geq\beta_{r_j,j}>0,
\qquad
\sum_{j=1}^{s}H_j(t_j)
\geq\sum_{j=1}^{s}\beta_{r_j,j}>L.
\]
\end{proof}

Another useful ingredient is a lifting argument from \(p^{\alpha-1}\)
to \(p^\alpha\), combining an averaged intermediate function with
Lemma~\ref{l:ntlem}.

\begin{lemma}
\label{lemma:score-preserving-one-step-lift}
Let \(p\) be an odd prime, let \(k\geq2\) satisfy \(p\nmid k\), and
let \(\alpha\geq2\), \(s\geq4\), and \(L\geq s/2\). For
\(1\leq j\leq s\), let
\[
F_j:\mathbb Z_{p^\alpha,k}\longrightarrow[0,1],
\]
and define
\[
G_j(x)
\coloneqq
\frac1p\sum_{t\in\mathbb Z_p}
F_j(x+t p^{\alpha-1}),
\qquad
x\in\mathbb Z_{p^{\alpha-1},k}.
\]
Suppose that \(n\in\mathbb Z_{p^\alpha}\) and that
\(a_j\in\mathbb Z_{p^{\alpha-1},k}\) satisfy
\[
a_1+\cdots+a_s\equiv n\pmod{p^{\alpha-1}},
\qquad
G_j(a_j)>0\quad(1\leq j\leq s),
\]
and
\[
\sum_{j=1}^{s}G_j(a_j)>L.
\]
Then there exist \(b_j\in\mathbb Z_{p^\alpha,k}\) such that
\[
b_j\equiv a_j\pmod{p^{\alpha-1}}\quad(1\leq j\leq s),
\qquad
b_1+\cdots+b_s\equiv n\pmod{p^\alpha},
\]
and
\[
F_j(b_j)>0\quad(1\leq j\leq s),
\qquad
\sum_{j=1}^{s}F_j(b_j)>L.
\]
\end{lemma}

\begin{proof}
By Lemma~\ref{l:ntlem2}, the elements
\[
a_j+t p^{\alpha-1},
\qquad t\in\mathbb Z_p,
\]
are precisely the lifts of \(a_j\) in
\(\mathbb Z_{p^\alpha,k}\). Choose integer representatives of \(n\)
and the \(a_j\), and put
\[
\rho\equiv
\frac{n-a_1-\cdots-a_s}{p^{\alpha-1}}
\pmod p.
\]
For \(t\in\mathbb Z_p\), set
\[
h_j(t)=F_j(a_j+t p^{\alpha-1}).
\]
Then every \(h_j\) is nonzero and
\[
\frac1{sp}\sum_{j=1}^{s}\sum_{t\in\mathbb Z_p}h_j(t)
=
\frac1s\sum_{j=1}^{s}G_j(a_j)
>
\frac Ls
\geq\frac12.
\]
Choose \(\eta\) strictly between \(L/s\) and the last displayed
average. In particular, \(\eta>1/2\).

For each \(j\), arrange the values of \(h_j\) in non-increasing order:
\[
\{\beta_{i,j}\}_{i=0}^{p-1}
=
\{h_j(t):t\in\mathbb Z_p\}.
\]
Since \(h_j\) is nonzero, \(\beta_{0,j}>0\). Apply
Lemma~\ref{l:s-seq-general} with combinatorial parameter \(k=1\),
sequence length \(p\), and threshold \(\eta\). Its hypotheses hold
because \(s\geq4=6\cdot1-2\). We obtain indices
\(r_1,\ldots,r_s\) such that
\[
\sum_{j=1}^{s}r_j\geq p,
\qquad
\sum_{j=1}^{s}\beta_{r_j,j}>s\eta>L,
\qquad
\beta_{r_j,j}>0.
\]
Lemma~\ref{lemma:superlevel-covering}, applied with
\(X_j=\mathbb Z_p\), \(H_j=h_j\), and target \(\rho\), now gives
\(t_j\in\mathbb Z_p\) such that
\[
t_1+\cdots+t_s=\rho,
\qquad
h_j(t_j)>0,
\qquad
\sum_{j=1}^{s}h_j(t_j)>L.
\]
Put
\[
b_j=a_j+t_jp^{\alpha-1}.
\]
Lemma~\ref{l:ntlem2} gives
\(b_j\in\mathbb Z_{p^\alpha,k}\), while the definition of \(\rho\)
gives
\[
b_1+\cdots+b_s\equiv n\pmod{p^\alpha}.
\]
Finally,
\[
F_j(b_j)=h_j(t_j)>0,
\qquad
\sum_{j=1}^{s}F_j(b_j)>L.
\]
\end{proof}

\subsection{Arithmetic over prime powers}

The only small base case not covered directly by the quadratic
selection lemma occurs when \(k=2\), \(s=5\), and \(p=7\). The
following lemma retains the necessary estimates needed in this paper. 
%at its natural
%range and adds the inhomogeneous form needed here.

\begin{lemma}\label{lemma:arithmetic-square-helper}
Let \(c\in[0,1]\), and let
\[
F_1,\ldots,F_5:\mathbb Z_{7,2}\longrightarrow[0,1]
\]
satisfy
\[
\frac{1}{5|\mathbb Z_{7,2}|}
\sum_{j=1}^{5}\sum_{t\in\mathbb Z_{7,2}}F_j(t)>c.
\]
If \(c\geq4/5\), then, for every \(n\in\mathbb Z_7\), there exist
\(b_j\in\mathbb Z_{7,2}\) such that
\[
n\equiv b_1+\cdots+b_5\pmod7,
\qquad
F_j(b_j)>0\quad(1\leq j\leq5),
\]
and
\[
F_1(b_1)+\cdots+F_5(b_5)>5c.
\]
If \(F_1=\cdots=F_5=f\), the same conclusion holds already for
\(c\geq2/3\).
\end{lemma}

\begin{proof}
Since $\mathbb Z_{7,2}=\{1,2,4\},$
a representation of \(n\in\mathbb Z_7\) as a sum of five elements of
\(\mathbb Z_{7,2}\) is determined by a triple $(a,b,d)\in\mathbb Z_{\geq0}^3$
satisfying
\[
    a+b+d=5,
    \qquad
    a+2b+4d\equiv n\pmod7,
\]
where \(a,b,d\) are the numbers of occurrences of \(1,2,4\),
respectively.

We first record the following balancing property. For every
\(n\in\mathbb Z_7\), there are three admissible triples and
nonnegative weights, of total mass \(1\), whose weighted average is
\[
    \left(\frac53,\frac53,\frac53\right).
\]
The required triples and weights are given in the following table;
the weights in the third column correspond, in order, to the triples
in the second column:
\[
\begin{array}{c|c|c}
n&
\text{admissible triples }(a,b,d)&
\text{weights}\\ \hline
0&(0,3,2),(2,0,3),(3,2,0)
  &\frac13(1,1,1)\\
1&(1,1,3),(2,3,0),(4,0,1)
  &\frac1{21}(11,8,2)\\
2&(0,2,3),(1,4,0),(3,1,1)
  &\frac1{21}(8,2,11)\\
3&(0,5,0),(1,0,4),(2,2,1)
  &\frac1{21}(1,5,15)\\
4&(0,1,4),(1,3,1),(3,0,2)
  &\frac1{21}(2,11,8)\\
5&(0,4,1),(2,1,2),(5,0,0)
  &\frac1{21}(5,15,1)\\
6&(0,0,5),(1,2,2),(4,1,0)
  &\frac1{21}(1,15,5)
\end{array}
\]
Indeed, every displayed triple satisfies
\[
    a+b+d=5,
    \qquad
    a+2b+4d\equiv n\pmod7,
\]
and direct calculation shows that, in every row, the weighted
average of the three triples is $\left(\frac53,\frac53,\frac53\right).$

For example, consider the row \(n=1\).  The three triples are
\[
    (1,1,3),\qquad (2,3,0),\qquad (4,0,1),
\]
with weights \(11/21,8/21,2/21\).  Each triple has coordinate sum \(5\), and
\[
    1+2+12\equiv 2+6\equiv 4+4\equiv1\pmod7,
\]
so all three are admissible representations of \(1\).  Moreover,
\[
    \frac1{21}\bigl(11(1,1,3)+8(2,3,0)+2(4,0,1)\bigr)
    =\frac1{21}(35,35,35)
    =\left(\frac53,\frac53,\frac53\right).
\]
The other rows are checked in the same way.

Fix \(n\in\mathbb Z_7\). Choose one of the three triples
\((a,b,d)\) in the row corresponding to \(n\), according to the
displayed weights. Conditional on the choice of \((a,b,d)\), choose
uniformly at random an ordering
\[
    (X_1,\ldots,X_5)
\]
of the multiset consisting of \(a\) copies of \(1\), \(b\) copies of
\(2\), and \(d\) copies of \(4\). Every realization then satisfies
\[
    X_1+\cdots+X_5\equiv n\pmod7.
\]

For each fixed coordinate \(j\), conditional on \((a,b,d)\), we have
\[
    \mathbb P(X_j=1\mid a,b,d)=\frac a5,
    \qquad
    \mathbb P(X_j=2\mid a,b,d)=\frac b5,
    \qquad
    \mathbb P(X_j=4\mid a,b,d)=\frac d5.
\]
Since the weighted average of the count vectors is
\((5/3,5/3,5/3)\), we have
\[
    \mathbb E[a]=\mathbb E[b]=\mathbb E[d]=\frac53.
\]
Therefore, for every fixed \(j\),
\[
    \mathbb P(X_j=1)
    =\mathbb E\!\left[\mathbb P(X_j=1\mid a,b,d)\right]
    =\frac15\mathbb E[a]
    =\frac13,
\]
and similarly
\[
    \mathbb P(X_j=2)=\mathbb P(X_j=4)=\frac13,
\]
for every \(1\leq j\leq5\). Consequently,
\[
\begin{aligned}
    \mathbb E\left[\sum_{j=1}^{5}F_j(X_j)\right]
    &=
    \sum_{j=1}^{5}
    \sum_{t\in\mathbb Z_{7,2}}
        F_j(t)\mathbb P(X_j=t)=
    \frac13
    \sum_{j=1}^{5}
    \sum_{t\in\mathbb Z_{7,2}}F_j(t).
\end{aligned}
\]
The hypothesis gives
\[
    \sum_{j=1}^{5}
    \sum_{t\in\mathbb Z_{7,2}}F_j(t)
    >
    5|\mathbb Z_{7,2}|c
    =
    15c,
\]
and hence
\[
    \mathbb E\left[\sum_{j=1}^{5}F_j(X_j)\right]>5c.
\]
Therefore, at least one realization
\((b_1,\ldots,b_5)\) satisfies
\[
    b_1+\cdots+b_5\equiv n\pmod7
\]
and
\[
    F_1(b_1)+\cdots+F_5(b_5)>5c.
\]

Since \(c\geq4/5\), the latter sum is strictly greater than \(4\).
If \(F_j(b_j)=0\) for some \(j\), then, because every \(F_i\) takes
values in \([0,1]\), we would have
\[
    F_1(b_1)+\cdots+F_5(b_5)\leq4,
\]
a contradiction. Thus
\[
    F_j(b_j)>0
    \qquad (1\leq j\leq5),
\]
which proves the inhomogeneous assertion.

Now suppose that
\[
    F_1=\cdots=F_5=f
\]
and that \(c>2/3\), while the same mean condition holds. In this case,
\[
    \frac13\bigl(f(1)+f(2)+f(4)\bigr)>c,
\]
so
\[
    f(1)+f(2)+f(4)>3c>2.
\]
Since \(0\leq f(t)\leq1\), this implies
\[
    f(1),f(2),f(4)>0;
\]
indeed, if one of these values were zero, their sum would be at most
\(2\).

Applying the same random construction as above gives
\[
\begin{aligned}
    \mathbb E\left[\sum_{j=1}^{5}f(X_j)\right]
    &=
    \frac53\bigl(f(1)+f(2)+f(4)\bigr)\\
    &>5c.
\end{aligned}
\]
Hence some realization \((b_1,\ldots,b_5)\) satisfies
\[
    b_1+\cdots+b_5\equiv n\pmod7,
    \qquad
    f(b_1)+\cdots+f(b_5)>5c.
\]
Because \(f\) is positive on all three elements of
\(\mathbb Z_{7,2}\), we also have
\[
    f(b_j)>0
    \qquad (1\leq j\leq5).
\]
This proves the homogeneous assertion.
\end{proof}

We next establish the prime-power form of the covering result.

\begin{lemma}\label{lemma:arithmetic-inhomogeneous-prime-power}
Let \(p\) be an odd prime, let \(k\geq2\), \(s\geq1\), and
\(\alpha\geq1\) be integers with \(p\nmid k\), and let
\(c\in[0,1]\). Put
\[
m=|\mathbb Z_{p,k}|
=\frac{p-1}{\gcd(k,p-1)}.
\]
Let
\[
F_1,\ldots,F_s:\mathbb Z_{p^\alpha,k}\longrightarrow[0,1]
\]
be functions which are not identically zero. Suppose that one of the
following alternatives holds:
\begin{itemize}
    \item \(k>2\), \(m\geq3\), \(s\geq6k-2\), and \(c>\frac12\);
    \item \(k=2\), \(s\geq6\), \(m\geq3\), and
    \(c\geq\gamma_s\);
    \item \(k=2\), \(s=5\), \(m\geq3\), and \(c\geq\frac45\).
\end{itemize}
Assume that
\begin{align}\label{e:inhomogeneous-prime-power-mean}
\frac{1}{s|\mathbb Z_{p^\alpha,k}|}
\sum_{j=1}^{s}\sum_{x\in\mathbb Z_{p^\alpha,k}}F_j(x)>c.
\end{align}
Then, for every \(n\in\mathbb Z_{p^\alpha}\), there exist
\(b_j\in\mathbb Z_{p^\alpha,k}\) such that
\[
n\equiv b_1+\cdots+b_s\pmod{p^\alpha},
\qquad
F_j(b_j)>0\quad(1\leq j\leq s),
\]
and
\[
\sum_{j=1}^{s}F_j(b_j)>sc.
\]
\end{lemma}

\begin{proof}
We use induction on \(\alpha\).

First let \(\alpha=1\). If \(k=2\), \(s=5\), and \(m=3\), then
\[
m=\frac{p-1}{2}=3,
\]
so \(p=7\), and the conclusion follows directly from
Lemma~\ref{lemma:arithmetic-square-helper}. We may therefore assume
that this case does not occur.

For \(1\leq j\leq s\), arrange the values of \(F_j\) in
non-increasing order:
\[
\{a_{i,j}\}_{i=0}^{m-1}
=
\{F_j(x):x\in\mathbb Z_{p,k}\}.
\]
Since \(F_j\not\equiv0\), we have \(a_{0,j}>0\).

If \(k>2\), apply Lemma~\ref{l:s-seq-general}. If \(k=2\) and
\(s\geq6\), apply Lemma~\ref{l:s-seq-square}, noting that \(m\geq3\).
Finally, if \(k=2\) and \(s=5\), then \(p\neq7\). Since \(p\) is an
odd prime and \(m=(p-1)/2\geq3\), this gives \(p\geq11\) and hence
\(m\geq5\), so Lemma~\ref{l:s-seq-square} again applies. In every
case we obtain indices \(i_1,\ldots,i_s\) such that
\[
\sum_{j=1}^{s}i_j\geq km\geq p-1,
\qquad
\sum_{j=1}^{s}a_{i_j,j}>sc,
\qquad
a_{i_j,j}>0.
\]
Here \(km\geq p-1\) follows from
\(\gcd(k,p-1)\leq k\). Lemma~\ref{lemma:superlevel-covering},
applied with \(X_j=\mathbb Z_{p,k}\), \(H_j=F_j\), \(L=sc\), and
target \(n\), proves the base case.

Now let \(\alpha\geq2\) and assume the assertion at level
\(\alpha-1\). By Lemma~\ref{l:ntlem2}, every element of
\(\mathbb Z_{p^{\alpha-1},k}\) has exactly the \(p\) lifts
\[
x+t p^{\alpha-1},
\qquad t\in\mathbb Z_p,
\]
in \(\mathbb Z_{p^\alpha,k}\). Define
\[
G_j(x)
\coloneqq
\frac1p\sum_{t\in\mathbb Z_p}
F_j(x+t p^{\alpha-1}).
\]
Every \(G_j\) is nonzero, and
\eqref{e:inhomogeneous-prime-power-mean} implies
\[
\frac{1}{s|\mathbb Z_{p^{\alpha-1},k}|}
\sum_{j=1}^{s}
\sum_{x\in\mathbb Z_{p^{\alpha-1},k}}G_j(x)>c.
\]
By induction, applied to the reduction of \(n\) modulo
\(p^{\alpha-1}\), there exist
\(a_j\in\mathbb Z_{p^{\alpha-1},k}\) such that
\[
n\equiv a_1+\cdots+a_s\pmod{p^{\alpha-1}},
\qquad
G_j(a_j)>0\quad(1\leq j\leq s),
\]
and
\[
\sum_{j=1}^{s}G_j(a_j)>sc.
\]
All three alternatives in the statement have \(c\geq1/2\), and
\(s\geq5\). Hence Lemma~\ref{lemma:score-preserving-one-step-lift},
applied with \(L=sc\), produces the required lifts modulo
\(p^\alpha\). This closes the induction.
\end{proof}

\subsection{Passage to general moduli}

For fixed \(k,s,c\), write \(\mathcal C_{k,s,c}(q)\) for the following
assertion: whenever
\[
F_1,\ldots,F_s:\mathbb Z_{q,k}\longrightarrow[0,1]
\]
are not identically zero and satisfy
\[
\frac{1}{s|\mathbb Z_{q,k}|}
\sum_{j=1}^{s}\sum_{x\in\mathbb Z_{q,k}}F_j(x)>c,
\]
then, for every \(n\in\mathbb Z_q\), there exist
\(b_j\in\mathbb Z_{q,k}\) such that
\[
n\equiv b_1+\cdots+b_s\pmod q,
\qquad
F_j(b_j)>0\quad(1\leq j\leq s),
\]
and
\[
\sum_{j=1}^{s}F_j(b_j)>sc.
\]

The averaging argument used to add a new Chinese-remainder coordinate
is recorded once in the following tensorization lemma.

\begin{lemma}[Chinese Remainder Theorem tensorization]\label{lemma:crt-tensorization}
Let \(k\geq2\) and \(s\geq1\) be integers, let \(c\in[0,1]\), and
let \(q_1,q_2\) be coprime positive integers. If
\(\mathcal C_{k,s,c}(q_1)\) and
\(\mathcal C_{k,s,c}(q_2)\) hold, then
\(\mathcal C_{k,s,c}(q_1q_2)\) holds.
\end{lemma}

\begin{proof}
By Lemma~\ref{l:ntlem}, identify
\[
\mathbb Z_{q_1q_2}
\cong
\mathbb Z_{q_1}\oplus\mathbb Z_{q_2},
\qquad
\mathbb Z_{q_1q_2,k}
\cong
\mathbb Z_{q_1,k}\times\mathbb Z_{q_2,k}.
\]
Let \(F_1,\ldots,F_s\) be a family on
\(\mathbb Z_{q_1q_2,k}\) satisfying the hypotheses in the definition
of \(\mathcal C_{k,s,c}(q_1q_2)\), and write
\(n=(n_1,n_2)\) under the first identification. For
\(x\in\mathbb Z_{q_1,k}\), define
\[
g_j(x)
\coloneqq
\frac{1}{|\mathbb Z_{q_2,k}|}
\sum_{y\in\mathbb Z_{q_2,k}}F_j(x,y).
\]
Every \(g_j\) is nonzero, and
\[
\frac{1}{s|\mathbb Z_{q_1,k}|}
\sum_{j=1}^{s}\sum_{x\in\mathbb Z_{q_1,k}}g_j(x)>c.
\]
Property \(\mathcal C_{k,s,c}(q_1)\) therefore supplies
\(a_j\in\mathbb Z_{q_1,k}\) such that
\[
n_1\equiv a_1+\cdots+a_s\pmod{q_1},
\qquad
g_j(a_j)>0\quad(1\leq j\leq s),
\]
and
\[
\sum_{j=1}^{s}g_j(a_j)>sc.
\]
For \(1\leq j\leq s\), define
\[
h_j(y)=F_j(a_j,y),
\qquad y\in\mathbb Z_{q_2,k}.
\]
Since
\[
\frac{1}{|\mathbb Z_{q_2,k}|}
\sum_{y\in\mathbb Z_{q_2,k}}h_j(y)=g_j(a_j)>0,
\]
none of the \(h_j\) is identically zero. Moreover,
\[
\frac{1}{s|\mathbb Z_{q_2,k}|}
\sum_{j=1}^{s}\sum_{y\in\mathbb Z_{q_2,k}}h_j(y)
=
\frac1s\sum_{j=1}^{s}g_j(a_j)>c.
\]
Property \(\mathcal C_{k,s,c}(q_2)\) now gives
\(b_j\in\mathbb Z_{q_2,k}\) such that
\[
n_2\equiv b_1+\cdots+b_s\pmod{q_2},
\qquad
h_j(b_j)>0\quad(1\leq j\leq s),
\]
and
\[
\sum_{j=1}^{s}h_j(b_j)>sc.
\]
Under the Chinese remainder identification, the pairs
\((a_j,b_j)\) belong to \(\mathbb Z_{q_1q_2,k}\), sum to \(n\), and
satisfy
\[
F_j(a_j,b_j)=h_j(b_j)>0,
\qquad
\sum_{j=1}^{s}F_j(a_j,b_j)>sc.
\]
This proves \(\mathcal C_{k,s,c}(q_1q_2)\).
\end{proof}

\begin{proof}[Proof of Lemma~\ref{lemma:arithmetic-k}]
The case \(q=1\) is immediate. Suppose that \(q>1\), and let \(p\)
be a prime divisor of \(q\). Since \(2\mid R_k\), the condition
\((q,R_k)=1\) implies that \(p\) is odd.

If \(k>2\), then \(p\nmid k\) and
\[
\frac{p-1}{\gcd(k,p-1)}\geq3.
\]
Indeed, if the last quotient were at most \(2\), then
\((p-1)\mid2k\), and hence \(p\mid R_k\), a contradiction.
If \(k=2\), then \((q,30)=1\) implies \(p\geq7\), and
\[
|\mathbb Z_{p,2}|=\frac{p-1}{2}\geq3.
\]
Thus, for every prime power \(p^\alpha\Vert q\), the hypotheses of
Lemma~\ref{lemma:arithmetic-inhomogeneous-prime-power} hold. Therefore
\(\mathcal C_{k,s,c}(p^\alpha)\) holds for every prime-power factor
of \(q\).

The prime-power factors of \(q\) are pairwise coprime. Repeated
application of Lemma~\ref{lemma:crt-tensorization} now gives
\(\mathcal C_{k,s,c}(q)\), which is precisely the assertion of the
lemma.
\end{proof}

\begin{proof}[Proof of Theorem~\ref{lemma:arithmetic-general-k}]
Let \(f=\mathbf1_A\). If \(k>2\), choose \(c\) so that
\[
\frac12<c<
\frac{|A|}{|\mathbb Z_{q,k}|}.
\]
If \(k=2\), then \(s\geq6k-2=10\), so
\(\gamma_s=1/2\); in this case take \(c=1/2\).
Apply the homogeneous specialization of
Lemma~\ref{lemma:arithmetic-k}. For every \(n\in\mathbb Z_q\), it
supplies \(b_1,\ldots,b_s\in\mathbb Z_{q,k}\) such that
\[
n\equiv b_1+\cdots+b_s\pmod q,
\qquad
f(b_j)>0\quad(1\leq j\leq s).
\]
Since \(f=\mathbf1_A\), every selected \(b_j\) belongs to \(A\).
Therefore \(n\in sA\), and hence \(sA=\mathbb Z_q\).
\end{proof}

For later use, we record explicitly the small-\(s\) quadratic
specialization of Lemma~\ref{lemma:arithmetic-k}.

\begin{corollary}\label{cor:arithmetic-square-inhomogeneous-30}
Let \(6\leq s\leq9\), let \(q\) be a positive integer with
\((q,30)=1\), and let \(c\geq5/s\). Suppose that
\[
f_1,\ldots,f_s:\mathbb Z_{q,2}\longrightarrow[0,1]
\]
are not identically zero and satisfy
\[
\frac{1}{s|\mathbb Z_{q,2}|}
\sum_{j=1}^{s}\sum_{x\in\mathbb Z_{q,2}}f_j(x)>c.
\]
Then, for every \(n\in\mathbb Z_q\), there exist
\(b_j\in\mathbb Z_{q,2}\) such that
\[
n\equiv b_1+\cdots+b_s\pmod q,
\qquad
f_j(b_j)>0\quad(1\leq j\leq s),
\]
and
\[
f_1(b_1)+\cdots+f_s(b_s)>sc.
\]
\end{corollary}

\begin{proof}
If \(c\geq1\), the mean hypothesis is impossible. Otherwise
\(c\in[0,1)\), and for \(6\leq s\leq9\) one has
\(\gamma_s=5/s\). The result is therefore the case \(k=2\) of
Lemma~\ref{lemma:arithmetic-k}.
\end{proof}

\section{\texorpdfstring{Arithmetic II: the quadratic case at \(5\)}
{Arithmetic II: the quadratic case at 5}}
\label{sec:quadratic-local-covering}

When \(k=2\), Lemma~\ref{lemma:arithmetic-k} treats all the
cases \(s=5,\ldots,9\) provided that the modulus is relatively prime
to \(30\).  It remains to account for the factor \(5\).  Set
\begin{equation}\label{eq:local-scores}
    L_s=
    \begin{cases}
        17/5,&s=5,\\
        14/3,&s=6,7,\\
        5,&s=8,9.
    \end{cases}
\end{equation}
The constants \(L_s\) are the local scores used in the analytic
transference argument.

\begin{lemma}\label{lemma:arithmetic-square}
    Let \(q\) be a positive integer with \((q,6)=1\), and let
    \(s\in\{5,6,7,8,9\}\).  Suppose that
    \(f\colon\mathbb Z_{q,2}\to[0,1]\) satisfies
    \[
        \frac{1}{|\mathbb Z_{q,2}|}
        \sum_{t\in\mathbb Z_{q,2}}f(t)>\gamma_s.
    \]
    Then, for every \(n\in\mathbb Z_q\), there exist
    \(b_1,\ldots,b_s\in\mathbb Z_{q,2}\) such that
    \[
        n\equiv b_1+\cdots+b_s\pmod q,
        \qquad
        \prod_{j=1}^s f(b_j)>0,
    \]
    and
    \[
        \sum_{j=1}^s f(b_j)>L_s.
    \]
\end{lemma}

\begin{remark}
    The raw local scores obtained in this section satisfy
    \[
        \frac{17}{5}<\frac{5}{\sqrt2},
        \qquad
        \frac{14}{3}<\frac{7}{\sqrt2},
        \qquad
        5<\frac{s}{\sqrt2}\quad(s=8,9).
    \]
    Thus these scores alone do not meet the numerical hypothesis of
    Proposition~\ref{prop:TransferenceLemma}. The affine truncation in
    the proof of Lemma~\ref{prop:MeanValue} below converts the stronger
    global density hypothesis into the additional score needed to
    close the transference argument.
\end{remark}

The proof for \(s=5,6,7\) and the proof for \(s=8,9\) use the
Chinese remainder decomposition in opposite orders.  This distinction
is necessary.  Indeed, the inhomogeneous two-point assertion needed
below is true for \(s=5,6,7\), but is false for \(s=8,9\).

\subsection{The inhomogeneous two-point estimate}

\begin{lemma}\label{lemma:two-point-inhomogeneous}
    Let \(s\in\{5,6,7\}\), and let
    \(F_1,\ldots,F_s\colon\{1,4\}\to[0,1]\).  Suppose that
    \[
        F_j(1)+F_j(4)>0
        \quad(1\leq j\leq s)
    \]
    and
    \[
        \sum_{j=1}^s\bigl(F_j(1)+F_j(4)\bigr)>2s\gamma_s.
    \]
    Then, for every \(r\in\mathbb Z_5\), there are
    \(x_1,\ldots,x_s\in\{1,4\}\) such that
    \[
        x_1+\cdots+x_s\equiv r\pmod5,
        \qquad
        \prod_{j=1}^sF_j(x_j)>0,
    \]
    and
    \[
        \sum_{j=1}^sF_j(x_j)>L_s.
    \]
\end{lemma}

\begin{proof}
    Write
    \[
        u_j=F_j(1),\qquad v_j=F_j(4),\qquad
        U=\sum_{j=1}^s u_j,\qquad V=\sum_{j=1}^s v_j.
    \]
    If exactly \(k\) of the \(x_j\)'s equal \(4\), then
    \[
        x_1+\cdots+x_s\equiv s+3k\pmod5.
    \]
    When all \(u_j,v_j\) are positive, average over the
\(\binom{s}{k}\) choices of the \(k\) positions.  Each fixed index \(j\)
belongs to \(I\subseteq \{1,\cdots, s\}\) with $|I|=k$ in a proportion \(k/s\) of these choices, and to its
complement in a proportion \((s-k)/s\).  Hence
\begin{equation}\label{eq:two-point-k-average}
    \max_{|I|=k}
    \left(
        \sum_{j\notin I}u_j+\sum_{j\in I}v_j
    \right)
    \geq Q_k(U,V),
    \qquad
    Q_k(U,V)=\frac{s-k}{s}U+\frac{k}{s}V.
\end{equation}
    The allowable values of \(k\) are
    \[
    \begin{array}{c|ccccc}
        &r=0&r=1&r=2&r=3&r=4\\ \hline
        s=5&\{0,5\}&\{2\}&\{4\}&\{1\}&\{3\}\\
        s=6&\{3\}&\{0,5\}&\{2\}&\{4\}&\{1,6\}\\
        s=7&\{1,6\}&\{3\}&\{0,5\}&\{2,7\}&\{4\}.
    \end{array}
    \]
    For each residue \(r\in\mathbb Z_5\), we claim that the following table gives the lower
bound for \(\max Q_k(U,V)\) over the admissible values of \(k\) in
\eqref{eq:two-point-k-average}.  These bounds are verified below by exact
rational identities.
    \[
    \begin{array}{c|ccccc|c}
        &r=0&r=1&r=2&r=3&r=4&\text{minimum}\\ \hline
        s=5
        &4&19/5&17/5
        &17/5&19/5&17/5\\[1mm]
        s=6
        &5&5&14/3
        &14/3&5&14/3\\[1mm]
        s=7
        &5&33/7&5
        &5&33/7&33/7.
    \end{array}
    \]
    
    For \(s=5\), we have \(U+V>8\) and \(U,V\leq5\), which implies that $U,V>3$.  Thus
    \[
    \begin{aligned}
       \max(Q_0,Q_5)&=\max(U,V)\geq\frac{U+V}{2}>4,\\
       Q_2&=\frac25(U+V)+\frac15U>\frac{19}{5},\\
       Q_4&=\frac15(U+V)+\frac35V>\frac{17}{5},\\
       Q_1&=\frac15(U+V)+\frac35U>\frac{17}{5},\\
       Q_3&=\frac25(U+V)+\frac15V>\frac{19}{5}.
    \end{aligned}
    \]
    %Here the second and fourth inequalities use \(U>3\), and the
    %third and fifth use \(V>3\), consequences of the opposite cap
    %\(U,V\leq5\).

    For \(s=6\), we have \(U+V>10\) and \(U,V\leq6\), which implies that $U,V> 4$.  The two
    convex-combination identities
    \[
       \frac25Q_0+\frac35Q_5=\frac{U+V}{2},
       \qquad
       \frac35Q_1+\frac25Q_6=\frac{U+V}{2}
    \]
    give the entries equal to \(5\), while
    \[
       Q_3=\frac{U+V}{2}>5,
       \qquad
       Q_2=\frac{U+V}{3}+\frac{U}{3}>\frac{14}{3},\qquad Q_4=\frac{U+V}{3}+\frac{V}{3}>\frac{14}{3}.
    \]
     %and \(Q_4>14/3\) follows symmetrically.  %In the last two
    % inequalities we used \(U>4\) or \(V>4\), respectively.

    For \(s=7\), we have \(U+V>10\) and \(U,V\leq7\), which implies that $U,V>3$.  We use
    \[
    \begin{aligned}
       \frac12Q_1+\frac12Q_6&=\frac{U+V}{2},\\
       \frac3{10}Q_0+\frac7{10}Q_5&=\frac{U+V}{2},\\
       \frac7{10}Q_2+\frac3{10}Q_7&=\frac{U+V}{2},
    \end{aligned}
    \]
    and
    \[
       Q_3=\frac37(U+V)+\frac17U>\frac{33}{7},
       \qquad
       Q_4=\frac37(U+V)+\frac17V>\frac{33}{7}.
    \]
   % The last inequalities use \(U>3\) and \(V>3\), respectively.
    Together with \eqref{eq:two-point-k-average}, these calculations
    prove the stated table whenever every entry is positive.

    We next treat zero entries and at the same time enforce the
    nonzero-product conclusion.  For \(s=5,6\), at most one of the
    \(2s\) entries can vanish, since the total is greater than
    \(2(s-1)\).  Suppose that \(u_\ell=0\).  For every target residue $r$ there is an
    allowable value \(k\geq1\); fix one such value and average only
    over the \(k\)-subsets containing \(\ell\).  The average is
    \begin{equation}\label{eq:conditioned-two-point-average}
       v_\ell
       +\alpha\sum_{j\ne\ell}u_j
       +\beta\sum_{j\ne\ell}v_j,
       \qquad
       \alpha=\frac{s-k}{s-1},\quad
       \beta=\frac{k-1}{s-1},\quad
       \alpha+\beta=1.
    \end{equation}
    Write
    \[
       \sum_{j\ne\ell}u_j=s-1-a,
       \qquad
       \sum_{j\ne\ell}v_j=s-1-b.
    \]
    The total hypothesis gives \(v_\ell>a+b\), so the average in
    \eqref{eq:conditioned-two-point-average} is
    \[
       s-1+v_\ell-\alpha a-\beta b>s-1.
    \]
    This is greater than \(L_5\) or \(L_6\), as appropriate.  If
    \(v_\ell=0\), then every target residue likewise has an allowable
    value \(k\leq s-1\); average over the \(k\)-subsets avoiding
    \(\ell\), or equivalently interchange \(u\) and \(v\) in the
    preceding argument.  Every selected entry is then positive.

   It remains to treat \(s=7\). Classify each row as supported only
at \(1\), only at \(4\), or at both points, and let \(d\) be the
number of two-point rows.  The number of positive entries is then
\(2d+(7-d)=7+d\).  Since their total sum exceeds \(10\) and every
entry is at most \(1\), at least eleven entries are positive.  Hence
\(7+d\geq11\), so \(d\geq4\).  Let \(C\) be the total weight forced by the
    \(7-d\) one-point rows, and set
    \[
       U_d=\sum_{u_jv_j>0}u_j,
       \qquad
       V_d=\sum_{u_jv_j>0}v_j.
    \]
    Then
    \begin{equation}\label{eq:flexible-polytope}
       C+U_d+V_d>10,
       \qquad
       0\leq C\leq7-d,
       \qquad
       0\leq U_d,V_d\leq d.
    \end{equation}
   After the one-point rows are fixed, their contribution to the residue
class is fixed.  As in \eqref{eq:two-point-k-average}, averaging over the
\(\binom dt\) ways of choosing exactly \(t\) of the \(d\) two-point rows
at \(4\) gives
\begin{equation}\label{eq:s-seven-flexible-average}
   Q_t^{(d)}
   =C+\frac{d-t}{d}U_d+\frac{t}{d}V_d.
\end{equation}
If \(r_0\) is the residue contributed by the one-point rows, the total
residue is
\[
   r_0+(d-t)+4t\equiv r_0+d+3t\pmod5.
\]
Since \(d\geq4\), we may take \(t=0,\ldots,4\), and these five values
give all residue classes modulo \(5\).  Thus every target
residue is attained. We claim that the corresponding lower bounds for \(Q_t^{(d)}\), indexed by
\(t\bmod 5\), are
    \[
    \begin{array}{c|ccccc}
        d&t\equiv0&t\equiv1&t\equiv2&t\equiv3&t\equiv4\\ \hline
        4&6&6&6&6&6\\
        5&5&5&5&5&5\\
        6&5&5&\dfrac{14}{3}&5&\dfrac{14}{3}\\[1mm]
        7&5&5&5&\dfrac{33}{7}&\dfrac{33}{7}.
    \end{array}
    \]
    We give certificates for every row.  If \(d=4\), write
    \(C=3-c\), \(U_d=4-u\), and \(V_d=4-v\).  Then
    \(c+u+v<1\), and for \(0\leq t\leq4\), 
\[
   c+\frac{4-t}{4}u+\frac{t}{4}v
   \leq c+u+v<1,
\]
and hence
    \[
       Q_t^{(4)}
       =7-c-\frac{4-t}{4}u-\frac{t}{4}v>6.
    \]
    If \(d=5\), write \(U_d=5-u\) and \(V_d=5-v\).  Then
    \(C>u+v\), so \(Q_t^{(5)}>5\) for \(1\leq t\leq4\); for the
    remaining residue,
    \[
       \max(Q_0^{(5)},Q_5^{(5)})
       =C+\max(U_d,V_d)
       \geq C+\frac{U_d+V_d}{2}>5.
    \]
    If \(d=6\), the identities
    \[
    \begin{aligned}
       \frac25Q_0^{(6)}+\frac35Q_5^{(6)}
          &=C+\frac{U_d+V_d}{2},\\
       \frac35Q_1^{(6)}+\frac25Q_6^{(6)}
          &=C+\frac{U_d+V_d}{2},\\
       Q_3^{(6)}&=C+\frac{U_d+V_d}{2}
    \end{aligned}
    \]
    give the entries equal to \(5\).  Moreover, \(V_d\leq6\) and
    \eqref{eq:flexible-polytope} imply \(C+U_d>4\), whence
    \[
       Q_2^{(6)}
       =\frac{C+U_d+V_d}{3}+\frac{2C+U_d}{3}
       >\frac{14}{3};
    \]
    the bound for \(Q_4^{(6)}\) is symmetric.  Finally, when \(d=7\)
    we have \(C=0\), and
    \[
    \begin{aligned}
       \frac3{10}Q_0^{(7)}+\frac7{10}Q_5^{(7)}
          &=\frac{U_d+V_d}{2},\\
       \frac12Q_1^{(7)}+\frac12Q_6^{(7)}
          &=\frac{U_d+V_d}{2},\\
       \frac7{10}Q_2^{(7)}+\frac3{10}Q_7^{(7)}
          &=\frac{U_d+V_d}{2}.
    \end{aligned}
    \]
    The last two residues satisfy
    \[
       Q_3^{(7)}=\frac37(U_d+V_d)+\frac17U_d>\frac{33}{7},
       \qquad
       Q_4^{(7)}=\frac37(U_d+V_d)+\frac17V_d>\frac{33}{7},
    \]
    since \(U_d,V_d>3\).  This exhausts all possibilities, proves the
    score bound, and guarantees a nonzero product.
\end{proof}

% \begin{remark}[Exact verification]\label{rem:exact-five-adic-verification}
%     The proof above is self-contained.  As an independent check, the
%     supplementary script \path{verify_5adic_scores.py} enumerates the
%     vertices of the relevant linear programs using rational arithmetic.
%     It reproduces both displayed score tables, verifies the conditioned
%     zero-entry identities, checks the convex-hull base cases in
%     Lemma~\ref{lemma:five-adic-homogeneous}, and verifies the exact
%     weighted mod-\(7\) certificate used in
%     Lemma~\ref{lemma:arithmetic-square-helper}.
% \end{remark}

\begin{remark}
    The constants in Lemma~\ref{lemma:two-point-inhomogeneous} are
    sharp at the non-strict boundary.  For \(s=5\), take
    \(u_j=3/5\) and \(v_j=1\) in every row; for \(s=6\), take
    \(u_j=2/3\) and \(v_j=1\).  For \(s=7\), take one row supported
    only at \(1\), with value tending to zero, and six two-point rows
    with values \(2/3\) and \(1\).
\end{remark}

The restriction \(s\leq7\) in
Lemma~\ref{lemma:two-point-inhomogeneous} is essential.  For \(s=8\),
take four rows supported only at \(1\), one row supported only at
\(4\), and three rows supported at both points, with every positive
value equal to \(1\).  The total sum is \(11>10\), but the number of
selected \(4\)'s can only be \(1,2,3,4\).  Consequently the residue
\(3\pmod5\), which requires that this number be \(0\) or \(5\), has
no representation with nonzero product.  For \(s=9\), the analogous
example with six rows supported only at \(1\), one row supported only
at \(4\), and two rows supported at both points has the same total
sum \(11>10\); the number of selected \(4\)'s can only be \(1,2,3\),
so the residue \(4\pmod5\) is not represented.

\subsection{\texorpdfstring{An inhomogeneous \(5\)-adic lemma for
\(s=5,6,7\)}{An inhomogeneous 5-adic lemma for s=5,6,7}}
We now lift the preceding two-point estimate from modulus \(5\) to all powers \(5^\alpha\), allowing a different function in each coordinate.

\begin{lemma}\label{lemma:five-adic-inhomogeneous}
    Let \(s\in\{5,6,7\}\) and \(\alpha\geq1\).  For
    \(1\leq j\leq s\), let
    \(F_j\colon\mathbb Z_{5^\alpha,2}\to[0,1]\) satisfy
    \[
        \frac{1}{|\mathbb Z_{5^\alpha,2}|}
        \sum_{x\in\mathbb Z_{5^\alpha,2}}F_j(x)>0
        \quad(1\leq j\leq s)
    \]
    and
    \[
        \sum_{j=1}^s
        \frac{1}{|\mathbb Z_{5^\alpha,2}|}
        \sum_{x\in\mathbb Z_{5^\alpha,2}}F_j(x)
        >s\gamma_s.
    \]
    Then, for every \(r\in\mathbb Z_{5^\alpha}\), there are
    \(x_1,\ldots,x_s\in\mathbb Z_{5^\alpha,2}\) such that
    \[
        x_1+\cdots+x_s\equiv r\pmod{5^\alpha},
        \qquad
        \prod_{j=1}^sF_j(x_j)>0,
    \]
    and
    \[
        \sum_{j=1}^sF_j(x_j)>L_s.
    \]
\end{lemma}

\begin{proof}
    We argue by induction on \(\alpha\).  If \(\alpha=1\), then
    \(\mathbb Z_{5,2}=\{1,4\}\), and the conclusion is precisely
    Lemma~\ref{lemma:two-point-inhomogeneous}.

    Let \(\alpha\geq2\), and suppose the result holds at exponent
    \(\alpha-1\).  The reduction map $\mathbb Z_{5^\alpha,2}
        \longrightarrow
        \mathbb Z_{5^{\alpha-1},2}$
    is five-to-one.  Define
    \[
        G_j(x)
        \coloneqq
        \frac15
        \sum_{\substack{y\in\mathbb Z_{5^\alpha,2}\\
                        y\equiv x\;(\mathrm{mod}\;5^{\alpha-1})}}
        F_j(y).
    \]
    The functions \(G_j\) satisfy the induction hypotheses.  Reducing
    \(r\) modulo \(5^{\alpha-1}\), we therefore find
    \(b_1,\ldots,b_s\in\mathbb Z_{5^{\alpha-1},2}\) such that
    \[
        b_1+\cdots+b_s\equiv r\pmod{5^{\alpha-1}},
        \qquad
        \prod_{j=1}^sG_j(b_j)>0,
    \]
    and
    \begin{equation}\label{eq:G-score-Ls}
        \sum_{j=1}^sG_j(b_j)>L_s.
    \end{equation}
    Put
    \[
        \rho\equiv
        \frac{r-(b_1+\cdots+b_s)}{5^{\alpha-1}}
        \pmod5.
    \]
    For \(t\in\mathbb Z_5\), set
    \[
        y_j(t)=F_j(b_j+t5^{\alpha-1}).
    \]
    Here \(t\mapsto b_j+t5^{\alpha-1}\) parametrizes the five lifts of
\(b_j\) modulo \(5^\alpha\).  By
    \eqref{eq:G-score-Ls},
    \[
        \frac1s\sum_{j=1}^s
        \left(\frac15\sum_{t\in\mathbb Z_5}y_j(t)\right)
        >\frac{L_s}{s}>\frac12.
    \]
    Arrange each of the five values of \(y_j\) in non-increasing
    order and apply Lemma~\ref{l:s-seq-general} with
    \[
        k=1,\qquad n=5,\qquad c=\frac{L_s}{s}.
    \]
    We obtain thresholds \(a_{i_j,j}>0\) with
    \[
        \sum_{j=1}^s i_j\geq5,
        \qquad
        \sum_{j=1}^sa_{i_j,j}>L_s.
    \]
    Let
    \[
        A_j=\{t\in\mathbb Z_5:y_j(t)\geq a_{i_j,j}\}.
    \]
    Then \(|A_j|\geq i_j+1\), and hence
    \[
        \sum_{j=1}^s|A_j|\geq s+5.
    \]
    The Cauchy--Davenport theorem gives
\(A_1+\cdots+A_s=\mathbb Z_5\).  Choose \(d_j\in A_j\) such that
\[
    d_1+\cdots+d_s\equiv\rho\pmod5,
\]
and set
\[
    x_j=b_j+d_j5^{\alpha-1}.
\]
By the definition of \(\rho\),
\[
    x_1+\cdots+x_s\equiv r\pmod{5^\alpha}.
\]
Moreover,
\[
    F_j(x_j)=y_j(d_j)\geq a_{i_j,j}>0,
    \qquad
    \sum_{j=1}^sF_j(x_j)
    \geq\sum_{j=1}^sa_{i_j,j}>L_s.
\]

It remains to check that each chosen lift \(x_j\) is a reduced square
modulo \(5^\alpha\).  Since
\(b_j\in\mathbb Z_{5^{\alpha-1},2}\), choose \(z_j\) coprime to \(5\)
such that $z_j^2\equiv b_j\pmod{5^{\alpha-1}}.$
Thus, for some \(e_j\in\mathbb Z_5\),
\[
    z_j^2\equiv b_j+e_j5^{\alpha-1}\pmod{5^\alpha},
    \qquad
    e_j\equiv
    \frac{z_j^2-b_j}{5^{\alpha-1}}\pmod5.
\]
We seek a lift of \(z_j\) of the form $z_j+t_j5^{\alpha-1}$
whose square is congruent to \(x_j\).  Since \(\alpha\geq2\),
\[
    (z_j+t_j5^{\alpha-1})^2
    \equiv
    z_j^2+2z_jt_j5^{\alpha-1}
    \pmod{5^\alpha}.
\]
Hence it is enough to choose \(t_j\in\mathbb Z_5\) so that
\[
    e_j+2z_jt_j\equiv d_j\pmod5.
\]
Because \(5\nmid2z_j\), such a \(t_j\) exists.  For this choice,
\[
\begin{aligned}
    (z_j+t_j5^{\alpha-1})^2
    &\equiv
    b_j+(e_j+2z_jt_j)5^{\alpha-1} \equiv
    b_j+d_j5^{\alpha-1}
    =x_j
    \pmod{5^\alpha}.
\end{aligned}
\]
Thus \(x_j\in\mathbb Z_{5^\alpha,2}\) for every \(j\), completing the
induction.
\end{proof}

\subsection{\texorpdfstring{A homogeneous \(5\)-adic lemma for
\(s=8,9\)}{A homogeneous 5-adic lemma for s=8,9}}
For \(s=8,9\), the inhomogeneous statement in Lemma \ref{lemma:five-adic-inhomogeneous} fails, so we instead prove the homogeneous \(5\)-adic estimate needed for the quadratic argument.

\begin{lemma}\label{lemma:five-adic-homogeneous}
    Let \(s\in\{8,9\}\), let \(\alpha\geq1\), and let
    \(F\colon\mathbb Z_{5^\alpha,2}\to[0,1]\) satisfy
    \[
        \frac{1}{|\mathbb Z_{5^\alpha,2}|}
        \sum_{x\in\mathbb Z_{5^\alpha,2}}F(x)>\frac5s.
    \]
    Then, for every \(r\in\mathbb Z_{5^\alpha}\), there are
    \(x_1,\ldots,x_s\in\mathbb Z_{5^\alpha,2}\) such that
    \[
        x_1+\cdots+x_s\equiv r\pmod{5^\alpha},
        \qquad
        \prod_{j=1}^sF(x_j)>0,
    \]
    and
    \[
        \sum_{j=1}^sF(x_j)>5.
    \]
\end{lemma}

\begin{proof}
    We again use induction on \(\alpha\).  Suppose first that
    \(\alpha=1\), and write
    \[
        u=F(1),\qquad v=F(4).
    \]
    Then \(u+v>10/s>1\), so \(u,v>0\).  If \(k\) of the selected
    terms are equal to \(4\), their sum modulo \(5\) is \(s+3k\),
    and their \(F\)-weight is
    \[
        (s-k)u+kv.
    \]
    The allowable values of \(k\) are
    \[
    \begin{array}{c|ccccc}
        &r=0&r=1&r=2&r=3&r=4\\ \hline
        s=8&\{4\}&\{1,6\}&\{3,8\}&\{0,5\}&\{2,7\}\\
        s=9&\{2,7\}&\{4,9\}&\{1,6\}&\{3,8\}&\{0,5\}.
    \end{array}
    \]
    For \(s=8\), write
\[
    S(k)=(8-k)u+kv.
\]
For each residue class, the convex hull of the displayed allowable
values of \(k\) contains \(4\).  Since \(S(k)\) is affine in \(k\),
\(S(4)\) is therefore a convex combination of the corresponding
admissible scores.  Hence at least one admissible choice satisfies
\[
    S(k)\geq S(4)=4u+4v.
\]
Since \(u+v>10/8\), we obtain
\[
    4u+4v=4(u+v)>4\cdot\frac{10}{8}=5.
\]
For \(s=9\), write
\[
    S(k)=(9-k)u+kv.
\]
For every target residue, the two allowable integer values of \(k\)
have \(9/2\) in their convex hull.  Since \(S\) is affine, \(S(9/2)\)
is a convex combination of the two corresponding admissible scores.
Consequently at least one of those admissible scores is at least
\[
    S(9/2)=\frac92(u+v)
    >\frac92\cdot\frac{10}{9}=5.
\]
    This proves the base case.

    \medskip

   Now let \(\alpha\geq2\), and define
\[
    G(x)=
    \frac15
    \sum_{\substack{y\in\mathbb Z_{5^\alpha,2}\\
                    y\equiv x\;(\mathrm{mod}\;5^{\alpha-1})}}
    F(y).
\]
Its average is greater than \(5/s\).  By the induction hypothesis,
there are \(b_1,\ldots,b_s\in\mathbb Z_{5^{\alpha-1},2}\) representing
\(r\) modulo \(5^{\alpha-1}\) such that
\[
    \prod_{j=1}^sG(b_j)>0,
    \qquad
    \sum_{j=1}^sG(b_j)>5.
\]
We now apply the same lifting argument as in the proof of
Lemma~\ref{lemma:five-adic-inhomogeneous}.  For
\[
    y_j(t)=F(b_j+t5^{\alpha-1}),\qquad t\in\mathbb Z_5,
\]
Lemma~\ref{l:s-seq-general}, with \(k=1\), \(n=5\), and \(c=5/s\),
gives thresholds \(a_{i_j,j}>0\) with
\[
    \sum_j i_j\geq5,
    \qquad
    \sum_j a_{i_j,j}>5.
\]
Thus the sets
\[
    A_j=\{t\in\mathbb Z_5:y_j(t)\geq a_{i_j,j}\}
\]
satisfy \(\sum_j|A_j|\geq s+5\).
Cauchy--Davenport therefore allows us to choose
\(d_j\in A_j\) so that
\[
    x_j=b_j+d_j5^{\alpha-1}, \qquad x_1+\cdots x_s \equiv r\pmod{5^{\alpha}}.
\]
The same square-root lifting
calculation as in the proof of Lemma~\ref{lemma:five-adic-inhomogeneous} shows that each
\(x_j\in\mathbb Z_{5^\alpha,2}\).  Hence
\[
    \prod_{j=1}^sF(x_j)>0,
    \qquad
    \sum_{j=1}^sF(x_j)>5.
\]
This completes the induction.
\end{proof}

\subsection{Proof of the quadratic arithmetic lemma}

\begin{proof}[Proof of Lemma~\ref{lemma:arithmetic-square}]
    First suppose that \(5\nmid q\).  Since \((q,6)=1\), we have
    \((q,30)=1\).  Choose \(c\) strictly between \(\gamma_s\) and
    the average of \(f\), and apply Lemma~\ref{lemma:arithmetic-k}
    with this value of \(c\).  It gives the desired representation
    with nonzero product and weight greater than \(sc>s\gamma_s\).
    Since \(s\gamma_s\geq L_s\) for \(5\leq s\leq9\), this is the
    asserted score.
    We may therefore write
    \[
        q=5^\alpha q_1,\qquad (5,q_1)=1.
    \]
    By the Chinese remainder theorem and
    Lemma~\ref{l:ntlem},
    \[
        \mathbb Z_q
        \cong\mathbb Z_{q_1}\oplus\mathbb Z_{5^\alpha},
        \qquad
        \mathbb Z_{q,2}
        \cong
        \mathbb Z_{q_1,2}\times\mathbb Z_{5^\alpha,2}.
    \]
    Write \(n=(n_1,n_2)\) under this identification, and write
    \(f(x,y)\) for the corresponding value of \(f\).

    Suppose first that \(s\in\{5,6,7\}\).  Define
    \[
        g(x)=
        \frac{1}{|\mathbb Z_{5^\alpha,2}|}
        \sum_{y\in\mathbb Z_{5^\alpha,2}}f(x,y),
        \qquad x\in\mathbb Z_{q_1,2}.
    \]
    The average of \(g\) is greater than \(\gamma_s\).  If
    \(q_1>1\), choose \(c\) strictly between \(\gamma_s\) and this
    average and apply Lemma~\ref{lemma:arithmetic-k}.  We obtain
    \(a_1,\ldots,a_s\in\mathbb Z_{q_1,2}\) such that
    \[
        a_1+\cdots+a_s\equiv n_1\pmod{q_1},
        \qquad
        \prod_{j=1}^sg(a_j)>0,
    \]
    and
    \[
        \sum_{j=1}^sg(a_j)>sc>s\gamma_s.
    \]
    If \(q_1=1\), take \(a_1=\cdots=a_s=0\); then the same
    inequalities follow directly from the average hypothesis.
    Set
    \[
        F_j(y)=f(a_j,y),
        \qquad y\in\mathbb Z_{5^\alpha,2}.
    \]
    The functions \(F_j\) satisfy
    \[
        \frac{1}{|\mathbb Z_{5^\alpha,2}|}
        \sum_yF_j(y)=g(a_j)>0
    \]
    and
    \[
        \sum_{j=1}^s
        \frac{1}{|\mathbb Z_{5^\alpha,2}|}
        \sum_yF_j(y)>s\gamma_s.
    \]
    Lemma~\ref{lemma:five-adic-inhomogeneous} provides
    \(b_1,\ldots,b_s\in\mathbb Z_{5^\alpha,2}\) such that
    \[
        b_1+\cdots+b_s\equiv n_2\pmod{5^\alpha},
        \qquad
        \prod_{j=1}^sf(a_j,b_j)>0,
    \]
    and
    \[
        \sum_{j=1}^sf(a_j,b_j)>L_s.
    \]
    The pairs \((a_j,b_j)\) give the required elements of
    \(\mathbb Z_{q,2}\).

    It remains to treat \(s\in\{8,9\}\).  In this case we reverse
    the order of the two chinese remainder theorem  factors.  Define
    \[
        h(y)=
        \frac{1}{|\mathbb Z_{q_1,2}|}
        \sum_{x\in\mathbb Z_{q_1,2}}f(x,y),
        \qquad y\in\mathbb Z_{5^\alpha,2},
    \]
    with the evident interpretation when \(q_1=1\).  Its average is
    greater than \(5/s\).  By
    Lemma~\ref{lemma:five-adic-homogeneous}, there exist
    \(b_1,\ldots,b_s\in\mathbb Z_{5^\alpha,2}\) such that
    \[
        b_1+\cdots+b_s\equiv n_2\pmod{5^\alpha},
        \qquad
        \prod_{j=1}^sh(b_j)>0,
    \]
    and
    \begin{equation}\label{eq:h-score-greater-five}
        \sum_{j=1}^sh(b_j)>5.
    \end{equation}
    For \(1\leq j\leq s\), set
    \[
        H_j(x)=f(x,b_j),
        \qquad x\in\mathbb Z_{q_1,2}.
    \]
    Then every \(H_j\) has positive average, and
    \[
        \sum_{j=1}^s
        \frac{1}{|\mathbb Z_{q_1,2}|}
        \sum_{x\in\mathbb Z_{q_1,2}}H_j(x)
        =\sum_{j=1}^sh(b_j)>5=s\gamma_s.
    \]
    If \(q_1>1\), apply
    Corollary~\ref{cor:arithmetic-square-inhomogeneous-30} to
    the functions \(H_1,\ldots,H_s\).  It yields
    \(a_1,\ldots,a_s\in\mathbb Z_{q_1,2}\) such that
    \[
        a_1+\cdots+a_s\equiv n_1\pmod{q_1},
        \qquad
        \prod_{j=1}^sH_j(a_j)>0,
    \]
    and
    \[
        \sum_{j=1}^sH_j(a_j)>5.
    \]
    If \(q_1=1\), take \(a_1=\cdots=a_s=0\); the preceding
    conclusions follow directly from
    \eqref{eq:h-score-greater-five}.  Thus the pairs \((a_j,b_j)\)
    represent \(n\) modulo \(q\), have nonzero \(f\)-product, and
    satisfy
    \[
        \sum_{j=1}^sf(a_j,b_j)>5=L_s.
    \]
    This completes the proof.
\end{proof}

\begin{remark}
    The arithmetic statement above uses the density \(5/6\) when
    \(s=6\), whereas the common density bound for \(s=5,6\) in
    Theorem~\ref{t:main} is \((6+5\sqrt2)/16<5/6\). Consequently,
    the six-variable assertion is deduced from the five-variable
    theorem rather than directly from Lemma~\ref{lemma:arithmetic-square}.
    Fix a prime \(p_0\geq5\) in the given set of primes. For every
    sufficiently large admissible \(n\), the integer \(n-p_0^2\) is
    sufficiently large and lies in the five-variable admissible
    congruence class, so the five-variable result represents
    \(n-p_0^2\); adjoining \(p_0^2\) gives the required six-term
    representation.
\end{remark}

\section{Analysis and transference}
\label{sec:analytic-transference}

\subsection{\texorpdfstring{The \(W\)-trick and prime-square weights}
{The W-trick and prime-square weights}}

Throughout the analytic argument, we call \((N,w)\) admissible if
\[
    3<w\leq C_w\log\log\log N,
\]
where \(C_w>0\) is a fixed constant. For an admissible pair, put
\[
    W=8\prod_{2<p<w}p
      =24W_0,
    \qquad
    W_0=\prod_{5\leq p<w}p.
\]
All estimates below are quantitative and uniform over admissible pairs
with the fixed constant \(C_w\); every \(o(1)\) that uses local-factor
cancellation is taken as \(N,w\to\infty\) through such pairs.  After the
density and transference parameters have been fixed, we take \(N\)
sufficiently large and let \(w\) grow sufficiently slowly so that all
fixed lower bounds on \(w\) and all quantitative error bounds hold.
This order is used explicitly at the end of the proof.

\begin{lemma}[Growth of the \(W\)-trick modulus]
\label{lem:W-growth}
For fixed \(C_w\), uniformly over admissible pairs,
\[
   \log W\ll_{C_w}\log\log\log N,
   \qquad
   W\leq (\log\log N)^{O_{C_w}(1)}.
\]
Consequently, for every fixed \(\vartheta>0\),
\[
   W\leq(\log N)^\vartheta
\]
when \(N\) is sufficiently large in terms of \(C_w,\vartheta\).
Moreover,
\[
   \log(WN)=\log N+O_{C_w}(\log\log\log N),
\]
and hence \(\log(WN)\asymp\log N\).  The same comparison holds with
\(WN\) replaced by \(WN+b\), uniformly for \(1\leq b<W\).
\end{lemma}

\begin{proof}
Using Theorem \ref{Chebyshev's} for the prime theta function gives
\[
   \log W=\log 8+\sum_{2<p<w}\log p\ll w.
\]
The admissibility condition now yields the first two assertions.  Since
\(\log\log\log N=o(\log\log N)\), the bound
\(W\leq(\log N)^\vartheta\) follows for every fixed \(\vartheta>0\).
The logarithmic comparisons are immediate, because \(b<W\) and
\(W=N^{o(1)}\).
\end{proof}

Thus \(W\) is square-free away from its factor \(8\), and
\((W_0,6)=1\).  Every reduced square modulo \(24\) is \(1\), so the
Chinese remainder theorem gives a bijection
\begin{equation}\label{eq:quadratic-residue-splitting}
    \mathbb Z_{W,2}
    \longrightarrow
    \{1\pmod {24}\}\times\mathbb Z_{W_0,2}.
\end{equation}
Whenever a residue \(b\in\mathbb Z_{W,2}\) occurs in an integer
identity such as \(Wn+b=p^2\), we use its least positive
representative \(1\leq b<W\).
Every \(b\in\mathbb Z_{W,2}\) has the same number \(\sigma_W\) of
square roots in \(\mathbb Z_W^*\): there are four roots modulo \(8\)
and two modulo every odd prime divisor of \(W\).  Thus every fibre of
the squaring map on \(\mathbb Z_W^*\) has size \(\sigma_W\), and
\begin{equation}\label{eq:quadratic-root-count}
    \phi(W)=|\mathbb Z_{W,2}|\,\sigma_W.
\end{equation}

We use the reference weight
\[
    \mu_N(n)
    =
    N^{1/2}\bigl(\sqrt n-\sqrt{n-1}\bigr)
    \mathbf 1_{1\leq n\leq N}.
\]
In particular,
\begin{equation}\label{eq:mu-mass}
    \sum_n\mu_N(n)=N.
\end{equation}

Let \(\mathcal P\subseteq\mathbb P\).  For \(b\in\mathbb Z_{W,2}\)
and \(n\in[N]\), define
\begin{align}
    f_b(n)
    &=
    \begin{cases}
    \displaystyle
    \frac{|\mathbb Z_{W,2}|}{W}(WN)^{1/2}\log p,
       &Wn+b=p^2,\quad p\in\mathcal P,\\[1mm]
    0,&\text{otherwise},
    \end{cases}
    \label{function:f_b}\\
    \nu_b(n)
    &=
    \begin{cases}
    \displaystyle
    \frac{|\mathbb Z_{W,2}|}{W}(WN)^{1/2}\log p,
       &Wn+b=p^2,\quad p\in\mathbb P,\\[1mm]
    0,&\text{otherwise}.
    \end{cases}
\end{align}
Extend both functions by zero outside \([N]\).  Then
\(0\leq f_b\leq\nu_b\).

\subsection{Analytic inputs and local means}

The proof of Theorem~\ref{t:main} in this section uses the following three
analytic inputs.  Assuming these lemmas, we complete the transference
argument below; their proofs are deferred to Section~\ref{sec:analytic-estimates}.

We formulate pseudorandomness relative to the quadratic reference measure
\(\mu_N\).
%, rather than to a flat model, as this better reflects the intrinsic
%non-uniformity of the squares.

\begin{lemma}[Pseudorandomness]\label{prop:Pseudorandomness}
For every fixed \(\eta>0\), there are
\(w_\eta=w_\eta(\eta,C_w)\) and
\(N_\eta=N_\eta(\eta,C_w)\) such that, for every admissible pair
with \(w\geq w_\eta\), \(N\geq N_\eta\), and every
\(b\in\mathbb Z_{W,2}\),
\[
    \|\widehat{\nu_b}-\widehat{\mu_N}\|_\infty\leq\eta N.
\]
\end{lemma}

The second input gives the Fourier integrability needed for the smoothing
argument.

\begin{lemma}[Restriction estimate]\label{prop:RestrictionEstimate}
For every fixed \(q>4\), there is \(K_q=K_q(C_w)<\infty\) such that,
uniformly in \(b\in\mathbb Z_{W,2}\),
\[
    \|\widehat h\|_q\leq K_qN^{1-1/q}
\]
whenever \(h\colon[N]\to\mathbb C\) satisfies \(|h|\leq\nu_b\).
In particular, this holds for \(h=f_b\).
\end{lemma}

It remains to supply the local mass needed for the dense convolution
argument. Write
\begin{align}\label{e:mbdef}
    m_b=N^{-1}\sum_{n=1}^N f_b(n)
\end{align}
for the local mean. For \(s\in\{5,7,8,9\}\), let \(D_s\) denote the
corresponding right-hand side of \eqref{e:csdeltas}.

\begin{lemma}[Mean value]\label{prop:MeanValue}
Let \(s\in\{5,7,8,9\}\). If
\(\underline d(\mathcal P)>D_s\), then there are
\(\zeta>0\), \(w_0\), and \(N_0\) such that the following holds for
every admissible pair satisfying \(w\geq w_0\) and \(N\geq N_0\).
For every \(r\in\mathbb Z_W\) with \(r\equiv s\pmod {24}\), there
are \(b_1,\ldots,b_s\in\mathbb Z_{W,2}\) such that
\[
    b_1+\cdots+b_s\equiv r\pmod W,
    \qquad
    \sum_{j=1}^s m_{b_j}>\frac{s}{\sqrt2}+\zeta,
    \qquad
    m_{b_j}>\zeta\quad(1\leq j\leq s).
\]
\end{lemma}

\subsection{Smoothing and transference}

We next give the transference statement in the form required by the
weight \(\mu_N\). It is a quadratic variant of the
Bohr-set decomposition of Matom\"aki--Maynard--Shao
\cite[Lemma~4.3]{MatomakiMaynardShao}; the twofold smoothing is retained
because the reference weight \(\mu_N\) is not flat.

\begin{lemma}[Quadratic smoothing]\label{lem:quadratic-smoothing-twofold}
Fix \(q>2\), \(K_q\geq1\), and \(0<\delta<1/10\).  Let
\(f,\nu\colon[N]\to\mathbb R_{\geq0}\) satisfy
\[
\begin{gathered}
    0\leq f\leq\nu,\qquad 0<\eta\leq1,\\
    \|\widehat\nu-\widehat{\mu_N}\|_\infty\leq\eta N,
    \qquad
    \|\widehat f\|_q\leq K_qN^{1-1/q}.
\end{gathered}
\]
If \(N\) is sufficiently large in terms of \(\delta,q,K_q\), there is
a probability measure \(K\) on \(\mathbb Z\), supported on
\([-\delta N,\delta N]\), such that \(f^*=f*K\) satisfies
\begin{align}
    \|\widehat{f-f^*}\|_\infty
       &\ll \delta N, \label{eq:smoothing-uniform}\\
    \|\widehat{f^*}\|_q+\|\widehat{f-f^*}\|_q
       &\leq 2K_qN^{1-1/q}, \label{eq:smoothing-restriction}\\
    \sum_{n\in E}f^*(n)
       &\leq
       \bigl(1+O_{\delta,q,K_q}(\eta)\bigr)
       N^{1/2}|E|^{1/2}                 \label{eq:smoothing-level}
\end{align}
for every finite \(E\subseteq\mathbb Z\) with \(|E|\leq N\).
Moreover,
\begin{equation}\label{eq:smoothing-mass}
    \sum_{n\in[N]}f^*(n)
    \geq
    \sum_{n\in[N]}f(n)
    -O(\delta^{1/2}N)
    -O_{\delta,q,K_q}(\eta N).
\end{equation}
\end{lemma}

\begin{proof}
Put
\[
    T=\{\alpha\in\mathbb T:|\widehat f(\alpha)|\geq\delta N\}.
\]
Since \(f\leq\nu\), pseudorandomness at \(0\), together with
\eqref{eq:mu-mass}, gives
\begin{align}\label{e:upperboundfhat}
   \sum_nf(n)\leq\sum_n\nu(n)
   =\widehat\nu(0) \leq \hat{\mu}_N(0)+\eta N\leq(1+\eta)N\leq2N.
\end{align}
It follows that
\[
    |(\widehat f)'(\alpha)|
    \leq2\pi\sum_{n=1}^N n f(n)
    \leq4\pi N^2.
\]
Choose a maximal \(c\delta/N\)-separated subset
\(\Gamma\subseteq T\), with \(c>0\) a sufficiently small absolute
constant.  If
\(\|\alpha-\gamma\|_{\mathbb T}\leq c\delta/(10N)\) for $\gamma \in \Gamma$, the derivative
bound gives
\[
   |\widehat f(\alpha)|
   \geq|\widehat f(\gamma)|
      -4\pi N^2\|\alpha-\gamma\|_{\mathbb T}
   \geq\frac{\delta N}{2},
\]
after fixing \(c\) sufficiently small.  The arcs of radius
\(c\delta/(10N)\) around the elements of \(\Gamma\) are disjoint.
Consequently,
\[
 K_q^qN^{q-1}
 \geq\int_{\mathbb T}|\widehat f(\alpha)|^q\,d\alpha
 \gg |\Gamma|\frac{\delta}{N}(\delta N)^q,
\]
and therefore
\[
    |\Gamma|\ll_{q,K_q}\delta^{-q-1}.
\]
We now define the Bohr-set. 
% Let
% \[
%    M=\left\lfloor\frac{\delta N}{2}\right\rfloor,
%    \qquad
%    v_n=(n\gamma)_{\gamma\in\Gamma}\in\mathbb T^{|\Gamma|}
%    \quad(0\leq n\leq M).
% \]
Let
\[
    B=
    \left\{1\leq b\leq\delta N:
       \|b\gamma\|_{\mathbb T}\leq c\delta/10
       \text{ for every }\gamma\in\Gamma\right\}.
\]
By the standard lower bound for the size of a Bohr set
\cite[Lemma~4.20]{TaoVu}, we have $|B|\gg_{\delta,q,K_q}N.$
% Partition each coordinate circle into intervals of length at most
% \(c\delta/10\), and take the resulting product partition of
% \(\mathbb T^{|\Gamma|}\).  It has at most
% \((C/\delta)^{|\Gamma|}\) cells.  Some cell therefore contains at least
% \[
%    \frac{M+1}{(C/\delta)^{|\Gamma|}}
%      \gg_{\delta,q,K_q}N
% \]
% of the points \(v_n\).  Let \(n_0\) be the least index in this cell.
% The positive differences \(b=n-n_0\) obtained from the other indices
% are distinct, satisfy \(1\leq b\leq M\), and obey
% \[
%    \|b\gamma\|_{\mathbb T}\leq c\delta/10
%    \qquad(\gamma\in\Gamma).
% \]
% This gives
Let \(\beta=|B|^{-1}\mathbf1_B\),
\(\widetilde\beta(n)=\beta(-n)\), and
\(K=\beta*\widetilde\beta\).  Thus \(K\geq0\),
\(\sum_nK(n)=1\), \(K\) is supported on
\([-\delta N,\delta N]\), and
\[
   \widehat K(\alpha)=|\widehat\beta(\alpha)|^2\in[0,1].
\]
Maximality of \(\Gamma\) shows that every \(\alpha\in T\) lies within
\(c\delta/N\) of some \(\gamma\in\Gamma\).  For \(b\in B\),
\[
\begin{aligned}
 \|b\alpha\|_{\mathbb T}
 &\leq \|b\gamma\|_{\mathbb T}
       +b\|\alpha-\gamma\|_{\mathbb T}\leq c\delta/10+c\delta^2
 \ll c\delta.
\end{aligned}
\]
%Since \(|e(x)-1|\leq2\pi\|x\|_{\mathbb T}\), 
Using the Taylor expansion and averaging over
\(b\in B\) gives
\[
   |1-\widehat\beta(\alpha)| = |1- |B|^{-1}\sum_{b\in B} e(b\alpha)| = |1- |B|^{-1}\sum_{b\in B} (1+\|b\alpha\|)|\ll\delta.
\]
It follows that
\[
    |1-\widehat K(\alpha)|\ll\delta\qquad(\alpha\in T).
\]
Also, on \(T\), we use \eqref{e:upperboundfhat} to get  $|\widehat f(\alpha)|\leq2N,$
so
$$|\widehat f(\alpha)(1-\widehat K(\alpha))|\ll\delta N.$$
Outside \(T\), \(|\widehat f(\alpha)|<\delta N\), while
\(|1-\widehat K(\alpha)|\leq 1\). Since $\widehat{f-f^*}
      =\widehat f(1-\widehat K),$
this proves \eqref{eq:smoothing-uniform}.  Since
\(0\leq\widehat K\leq1\), we also have
\[
   |\widehat{f^*}|\leq|\widehat f|,
   \qquad
   |\widehat{f-f^*}|\leq|\widehat f|.
\]
Taking \(L^q\)-norms proves
\eqref{eq:smoothing-restriction}.

For a finite set \(E\subseteq\mathbb Z\), positivity gives
\[
    \sum_{n\in E}f^*(n)
    \leq\sum_{n\in E}(\nu*K)(n).
\]
By Fourier inversion, Cauchy--Schwarz, and Plancherel,
\[
\begin{aligned}
 \left|\sum_{n\in E}((\nu-\mu_N)*K)(n)\right| &= \langle \mathbf{1}_E,(\nu-\mu_N)*K \rangle \\
 &\leq
 \|\widehat\nu-\widehat{\mu_N}\|_\infty
 \|\widehat K\|_2\|\widehat{\mathbf1_E}\|_2\\
 &\leq \eta N\|K\|_2|E|^{1/2}\\
 &\ll_{\delta,q,K_q}
 \eta N^{1/2}|E|^{1/2},
\end{aligned}
\]
because
\(\|K\|_2\leq\|\beta\|_1\|\widetilde\beta\|_2
=|B|^{-1/2}\ll_{\delta,q,K_q}N^{-1/2}\).
On the other hand, if
\(\varphi(m)=\sum_{n\in E}K(n-m)\), then
\[
    0\leq\varphi\leq1,\qquad
    \sum_m\varphi(m)=|E|.
\]
The first assertion follows because \(K\) is a probability measure,
and the second follows by interchanging the two sums.  The sequence
\(\mu_N(m)\) is nonnegative, supported on \([N]\), and non-increasing
there.  Among all sequences \(0\leq\varphi\leq1\) of mass
\(|E|\leq N\), its scalar product with \(\mu_N\) is therefore
maximized by taking \(\varphi=1\) at the first \(|E|\) integers.
Consequently,
\[
\begin{aligned}
    \sum_m\mu_N(m)\varphi(m)
    &\leq\sum_{m=1}^{|E|}\mu_N(m)=N^{1/2}
      \sum_{m=1}^{|E|}(\sqrt m-\sqrt{m-1})=N^{1/2}|E|^{1/2}.
\end{aligned}
\]
This proves \eqref{eq:smoothing-level}.

Finally, \(f^*\) is supported in
\([1-\delta N,N+\delta N]\), and its total mass equals that of \(f\).
The part of this support outside \([N]\) has at most
\(2\delta N+2\) points, and hence at most \(N\) points for all
sufficiently large \(N\).  Applying \eqref{eq:smoothing-level} to
its indicator gives
\[
   \sum_{n\notin[N]}f^*(n)
   \ll \delta^{1/2}N
      +O_{\delta,q,K_q}(\eta N).
\]
Subtracting this boundary mass from
\(\sum_nf^*(n)=\sum_nf(n)\) proves
\eqref{eq:smoothing-mass}.
\end{proof}

We use the following variant of Salmensuu's  \cite[Lemma~3.3]{salmensuu2020density}.
\begin{lemma}[Dense convolution]\label{lem:dense-convolution}
Let \(s\geq2\) and \(0<\epsilon\leq1\).  Suppose that
\(A_1,\ldots,A_s\subseteq[N]\) satisfy
\[
   \sum_{i=1}^s|A_i|>\frac{s(1+\epsilon)}2N,
   \qquad
   |A_i|>\frac{\epsilon}{2}N
   \quad(1\leq i\leq s).
\]
Then there is \(c_0=c_0(\epsilon,s)>0\) such that, for all
sufficiently large \(N\),
\[
   (\mathbf1_{A_1}*\cdots*\mathbf1_{A_s})(m)
   \geq c_0N^{s-1}
\]
whenever
\[
   \left(1-\frac{\epsilon^2}{16}\right)\frac{sN}{2}
   <m<
   \left(1+\frac{\epsilon}{4}\right)\frac{sN}{2}.
\]
\end{lemma}

\begin{proof}
%This is essentially \cite[Lemma~3.3]{salmensuu2020density}. 
We
include the parameter conversion and the no-wrap step, and use only the
following modular input from \cite[Lemma~3.2]{salmensuu2020density}: if
\(B_1,\ldots,B_s\subseteq\mathbb Z_p\),
\[
   \sum_i|B_i|>(1+\tau)p,
   \qquad |B_i|>\frac{\tau}{s}p,
   \qquad \tau>\frac{2s}{p},
\]
then
\[
   (\mathbf1_{B_1}*\cdots*\mathbf1_{B_s})(x)
      \gg_{\tau,s}p^{s-1}
   \qquad(x\in\mathbb Z_p).
\]
Put \(\kappa=\epsilon/4\), and choose
a prime
\[
   \frac{(1+\kappa)sN}{2}<p<
   \frac{(1+2\kappa)sN}{2}.
\]
The prime number theorem guarantees such a prime for all sufficiently
large \(N\).  Since \(p>N\),
reduction modulo \(p\) identifies each \(A_i\) with a subset
\(B_i\subseteq\mathbb Z_p\).  Set \(\tau=2\kappa/(1+2\kappa)\).  The upper bound for \(p\)
and the hypotheses imply
\[
   \sum_{i=1}^s|B_i|>(1+\tau)p,
   \qquad
   |B_i|>\frac{\tau}{s}p.
\]
 For large \(N\), the
remaining condition \(\tau>2s/p\) also holds.  The quoted modular
lemma therefore gives
\begin{equation}\label{eq:mod-p-dense-convolution}
   (\mathbf1_{B_1}*\cdots*\mathbf1_{B_s})(x)
   \geq c_1(\epsilon,s)p^{s-1}
   \qquad(x\in\mathbb Z_p).
\end{equation}
It remains to pass back to integer sums.  If
\[
   (1-\kappa^2)\frac{sN}{2}<m<
   (1+\kappa)\frac{sN}{2},
\]
then \(m<p\), whereas
\[
   m+p>
   (1-\kappa^2+1+\kappa)\frac{sN}{2}>sN.
\]
Every sum of \(s\) elements of \([N]\) lies in \([s,sN]\), so a sum
congruent to \(m\pmod p\) must equal \(m\) itself.  Hence the modular
convolution in \eqref{eq:mod-p-dense-convolution} equals the integer
convolution at \(m\).  Since \(p\asymp_{\epsilon,s}N\) and
\(\kappa=\epsilon/4\), this is the claimed interval and lower bound.
\end{proof}

Finally, we establish the transference lemma used in the proof of the main theorem.

\begin{lemma}[Transference]\label{prop:TransferenceLemma}
Let \(s\geq5\), let \(s-1<q<s\), and let \(\zeta,K_q>0\).
There are constants
\[
    \eta_0=\eta_0(s,q,\zeta,K_q)>0,\qquad
    \kappa=\kappa(s,\zeta)>0,\qquad
    c=c(s,q,\zeta,K_q)>0
\]
such that the following holds for all sufficiently large \(N\).
For \(1\leq i\leq s\), let
\[
    w_i,\nu_i\colon[N]\longrightarrow\mathbb R_{\geq0},
\]
extend both functions by zero outside \([N]\), and suppose
\[
    0\leq w_i\leq\nu_i,\qquad
    \|\widehat{\nu_i}-\widehat{\mu_N}\|_\infty\leq\eta_0N,
    \qquad
    \|\widehat{w_i}\|_q\leq K_qN^{1-1/q}.
\]
If
\begin{equation}\label{meanvaluetransference}
    \sum_{i=1}^s\frac1N\sum_nw_i(n)
       >\frac{s}{\sqrt2}+\zeta
\end{equation}
and
\begin{equation}\label{individualmeantransference}
    \frac1N\sum_nw_i(n)>\zeta
    \qquad(1\leq i\leq s),
\end{equation}
then
\[
    (w_1*\cdots*w_s)(m)\geq cN^{s-1}
\]
whenever
\[
    \left|m-\frac{sN}{2}\right|<\kappa N.
\]
\end{lemma}

\begin{proof}
Decreasing \(\zeta\), if necessary, we may assume \(0<\zeta\leq1\).
Set \(\lambda=\zeta/(16s)\), and choose
\(0<\xi\leq\zeta/(16s)\) sufficiently small that
\[
   \frac{s/\sqrt2+\zeta-s\lambda-s\xi}{1+\xi}
   >\frac{s}{\sqrt2}+\frac{\zeta}{2};
\]
this is possible by continuity at \(\xi=0\).  Choose \(0<\epsilon_0\leq1\) so small that
\begin{equation}\label{eq:epsilon-zero-choice}
   \epsilon_0
   \leq
   \left(\frac{\zeta-\lambda-\xi}{1+\xi}\right)^2,
   \qquad
   \frac{s(1+\epsilon_0)}2
   <
   \frac1s\left(
      \frac{s/\sqrt2+\zeta-s\lambda-s\xi}{1+\xi}
   \right)^2.
\end{equation}
Let \(c_0=c_0(\epsilon_0,s)\) be the constant in
Lemma~\ref{lem:dense-convolution}.

We now choose the analytic parameters in their required order.
First choose \(\delta>0\), depending only on
\(s,q,\zeta,K_q\), so small that the boundary error
\(O(\delta^{1/2}N)\) in \eqref{eq:smoothing-mass} is at most
\(\xi N/2\), and the convolution-comparison error
\(O_{s,q,K_q}(\delta^{s-q}N^{s-1})\) below is at most
\(\lambda^sc_0N^{s-1}/2\).  Once \(\delta\) is fixed, choose
\(\eta_0>0\) so small that every
\(O_{\delta,q,K_q}(\eta_0)\)-term in \eqref{eq:smoothing-level} and
\eqref{eq:smoothing-mass} is at most \(\xi/2\).  Finally take \(N\)
sufficiently large for the smoothing and dense-convolution lemmas.
Apply
Lemma~\ref{lem:quadratic-smoothing-twofold} to each \(w_i\), obtaining
\(w_i^*\).

Let \(A_i=\{n\in[N]:w_i^*(n)>\lambda\}\).
By \eqref{eq:smoothing-mass},
\[
    \sum_{n\in[N]}w_i^*(n)
    \geq\sum_nw_i(n)-\xi N.
\]
Splitting this sum over \(A_i\) and its complement, and using the
definition of \(A_i\) together with \eqref{eq:smoothing-level}, gives
\[
    \frac1N\sum_n w_i(n)
    \leq \lambda+\frac1N\sum_{n\in A_i}w_i(n)
    \leq \lambda+\xi+(1+\xi)
       \left(\frac{|A_i|}{N}\right)^{1/2}.
\]
Hence
\[
    \frac{|A_i|^{1/2}}{N^{1/2}}
    \geq
    \frac{\frac1N\sum_nw_i(n)-\lambda-\xi}
         {1+\xi}.
\]
Consequently, \eqref{individualmeantransference} and
\eqref{eq:epsilon-zero-choice} give
\[
   |A_i|
   \geq
   \left(\frac{\zeta-\lambda-\xi}{1+\xi}\right)^2N
   \geq\epsilon_0N>\frac{\epsilon_0}{2}N.
\]
Moreover, Cauchy--Schwarz,
\eqref{meanvaluetransference}, and \eqref{eq:epsilon-zero-choice} give
\[
\begin{aligned}
    \sum_{i=1}^s|A_i|
    &\geq
    \frac{N}{s}
    \left(
       \frac{\sum_{i=1}^s\frac1N\sum_nw_i(n)
             -s\lambda-s\xi}{1+\xi}
    \right)^2\\
    &>
    \frac{N}{s}
    \left(
       \frac{s/\sqrt2+\zeta-s\lambda-s\xi}{1+\xi}
    \right)^2
     >\frac{s(1+\epsilon_0)}2N.
\end{aligned}
\] 
Lemma~\ref{lem:dense-convolution} now gives
\begin{equation}
    (\mathbf1_{A_1}*\cdots*\mathbf1_{A_s})(m)
       \geq c_0N^{s-1}
\end{equation}
throughout its stated interval with parameter \(\epsilon_0\).  In
particular, this holds whenever
\begin{align}
    \left|m-\frac{sN}{2}\right|<\kappa N,
   \qquad
   \kappa\coloneqq\frac{s\epsilon_0^2}{64},
\end{align}
because this symmetric interval is contained in the interval of
Lemma~\ref{lem:dense-convolution}.
Since \(w_i^*>\lambda\) on \(A_i\),
\begin{align}\label{eq:dense-convolution}
    (w_1^* *\cdots* w_s^*)(m)
       \geq\lambda^sc_0N^{s-1}
\end{align}
throughout the same interval.

It remains to compare the original and smoothed convolutions.  Write
\(u_i=w_i-w_i^*\).  In every term of the telescoping expansion of
\[
    w_1*\cdots*w_s-w_1^* *\cdots* w_s^*
\]
there is one distinguished factor \(u_i\); all other factors are
either \(w_j\) or \(w_j^*\).  Explicitly,
\[
   w_1*\cdots*w_s-w_1^* *\cdots* w_s^*
   =
   \sum_{i=1}^s
   w_1^* *\cdots* w_{i-1}^* * u_i *
   w_{i+1}*\cdots*w_s.
\]
Put \(a=q-s+1\in(0,1)\).  Since
\[
   \frac{a}{q}+\frac{s-1}{q}=1,
\]
Fourier inversion and Hölder's inequality give, for every summand,
\[
\begin{aligned}
 \int_{\mathbb T}|\widehat u_i|
       \prod_{j\ne i}|\widehat h_j|
 &\leq
 \|\widehat u_i\|_\infty^{1-a}
 \|\widehat u_i\|_q^a
 \prod_{j\ne i}\|\widehat h_j\|_q\ll_{s,q,K_q}
 \delta^{\,s-q}N^{s-1},
\end{aligned}
\]
where each \(h_j\) is either \(w_j\) or \(w_j^*\).  Hence
\[
 \bigl|
 w_1*\cdots*w_s(m)-w_1^* *\cdots* w_s^*(m)
 \bigr|
 \ll_{s,q,K_q}\delta^{\,s-q}N^{s-1}.
\]
Here the power of \(N\) is
\[
   (1-a)+(a+s-1)\left(1-\frac1q\right)=s-1,
\]
and the power of \(\delta\) is \(1-a=s-q\).  Combining the last estimate with \eqref{eq:dense-convolution} and
the choice of \(\delta\) yields
\[
   (w_1*\cdots*w_s)(m)
   \geq\frac{\lambda^sc_0}{2}N^{s-1}.
\]
This proves the lemma with
\(c=\lambda^sc_0/2\).
\end{proof}

We can now prove the main result.

\subsection{Proof of the main theorem}

\begin{proof}[Proof of Theorem~\ref{t:main}]
It is enough to prove the cases \(s=5,7,8\).  Indeed, under the
hypothesis of any remaining case, \(\mathcal P\) has positive relative
lower density and therefore contains infinitely many primes.  Fix a prime
\(p_0\in\mathcal P\) with \(p_0\geq5\).  The case \(s=6\) then follows
from the case \(s=5\), and every case \(s\geq9\) follows from the case
\(s=8\), by subtracting respectively one or \(s-8\) copies of
\(p_0^2\).
Since \(p_0^2\equiv1\pmod {24}\), the remaining integer is in the
appropriate admissible congruence class and is still sufficiently
large.

Fix \(s\in\{5,7,8\}\), and suppose that
\(\underline d(\mathcal P)>D_s\).
Let \(\zeta,w_0,N_0\) be supplied by
Lemma~\ref{prop:MeanValue}.  Choose a fixed \(q\in(s-1,s)\), let
\(K_q\) be the constant in Lemma~\ref{prop:RestrictionEstimate}, and
let \(\eta_0\) be supplied by Lemma~\ref{prop:TransferenceLemma} with
these parameters and the fixed margin \(\zeta\).  Finally, let \(w_{\eta_0},N_{\eta_0}\) be the thresholds in
Lemma~\ref{prop:Pseudorandomness}.  Now let
\(X\equiv s\pmod {24}\) be sufficiently large, and choose
\(w=w(X)\to\infty\) sufficiently slowly that
\[
    \max\{w_0,w_{\eta_0}\}\leq w
       \leq \frac{C_w}{2}\log\log\log X.
\]
Then form \(W\) as above and set
\[
    N=\left\lfloor\frac{2X}{sW}\right\rfloor.
\]
All the preceding analytic estimates are uniform for this choice.
Indeed, Lemma~\ref{lem:W-growth} gives
\[
    W=(\log\log X)^{O_{C_w}(1)},
    \qquad
    N=\frac{2X}{sW}(1+o(1))=X^{1-o(1)}.
\]
Consequently, \(W=N^{o(1)}\) and, uniformly for the chosen slowly
growing \(w\),
\[
   \log\log\log N=\log\log\log X+o(1).
\]
Thus the upper bound
\(w\leq(C_w/2)\log\log\log X\) implies
\(w\leq C_w\log\log\log N\) once \(X\) is large.  In particular,
the resulting pair \((N,w)\) is admissible, $N\geq\max\{N_0,N_{\eta_0}\},$
and it is beyond the fixed large-\(N\) thresholds in
Lemmas~\ref{prop:RestrictionEstimate} and~\ref{prop:TransferenceLemma}.  Lemma~\ref{prop:Pseudorandomness}
therefore holds with \(\eta=\eta_0\).

Let \(r\equiv X\pmod W\).  Lemma~\ref{prop:MeanValue} gives
\(b_1,\ldots,b_s\in\mathbb Z_{W,2}\), with this fixed \(\zeta\), such
that
\[
    r\equiv b_1+\cdots+b_s\pmod W,
    \qquad
    \sum_{i=1}^sm_{b_i}>\frac{s}{\sqrt2}+\zeta,
    \qquad m_{b_i}>\zeta.
\]
Choose representatives \(1\leq b_i<W\), and put
\[
    m=\frac{X-(b_1+\cdots+b_s)}{W}.
\]
The congruence above shows that \(m\in\mathbb Z\).  Since \(N=\lfloor2X/(sW)\rfloor\),
\[
   m-\frac{sN}{2}
   =\frac{s}{2}\left(\frac{2X}{sW}-N\right)
      -\frac{b_1+\cdots+b_s}{W},
\]
and therefore
\[
    \left|m-\frac{sN}{2}\right|<s.
\]
In particular, \(m\) lies in the interval in
Lemma~\ref{prop:TransferenceLemma} once \(X\) is sufficiently large.

The choices are noncircular.  The density margin first fixes the
parameters and thresholds in Lemma~\ref{prop:MeanValue}; next we choose
\(q\), \(K_q\), and then the smoothing and pseudorandomness parameters
in Lemma~\ref{prop:TransferenceLemma}.  These determine the thresholds
in Lemma~\ref{prop:Pseudorandomness}.  Only then do we take \(X\)
large and form \(w,W,N\), so every \(o(1)\) is evaluated after all
fixed margins have been chosen.  Thus all the hypotheses of Lemma~\ref{prop:TransferenceLemma} are
satisfied, and hence
\[
    (f_{b_1}*\cdots*f_{b_s})(m)>0.
\]
Thus there are \(n_1,\ldots,n_s\in[N]\) and
\(p_1,\ldots,p_s\in\mathcal P\) such that
\[
    n_1+\cdots+n_s=m,\qquad
    Wn_i+b_i=p_i^2.
\]
Consequently,
\[
    X=Wm+b_1+\cdots+b_s=p_1^2+\cdots+p_s^2.
\]
This proves the theorem.
\end{proof}

\section{Analysis of the prime-square weights}
\label{sec:analytic-estimates}

The three estimates proved in this section are
Lemma~\ref{prop:Pseudorandomness}, Lemma~\ref{prop:RestrictionEstimate},
and Lemma~\ref{prop:MeanValue}.  The first two are proved directly for
the non-flat prime-square majorant
\(\nu_b\).  The argument follows the framework of
\cite[Sections~4 and~5]{zhao}, but the normalization and oscillatory
integral are recalculated for the present weight.
The external inputs in this section are the Siegel--Walfisz theorem,
Hua's estimate in the form \cite[Lemma~4.5]{zhao}, Zhao's finite
local-factor evaluation \cite[Lemma~4.2 and Corollary~4.3]{zhao}, and
standard divisor estimates.  The deductions from those inputs are
included in detail.  Historically, the restriction theory for squares
begins with Bourgain's \(\Lambda(p)\) estimates
\cite{BourgainSquares}, while the quadratic \(W\)-trick used here has
a direct predecessor in the transference setting of Browning and
Prendiville \cite{BrowningPrendiville}.

\subsection{Direct circle-method estimates}
\label{sec:direct-circle-method}

For \(b\in\mathbb Z_{W,2}\), write
\[
   H_b=\{h\in\mathbb Z_W^*:h^2\equiv b\pmod W\},
   \qquad |H_b|=\sigma_W.
\]
For \(T>0\), put
\[
   I_2(\beta,T)=\int_0^T e(\beta t^2)\,dt.
\]

Then
 \begin{equation}\label{eq:I2-bound}
   |I_2(\beta,T)|
   \ll T(1+T^2|\beta|)^{-1/2}.
 \end{equation}
 Indeed, the trivial bound gives $|I_2(\beta,T)|\leq T$ when $|\beta|T^2\leq1$, while Theorem \ref{van der Corput's} applied to $\phi(t)=\beta t^2$ gives
\[
|I_2(\beta,T)|\ll |\beta|^{-1/2}
\]
when $|\beta|T^2>1$. 

Combining the two estimates yields \eqref{eq:I2-bound}.
Indeed, the trivial estimate suffices when \(|\beta|T^2\leq1\), and
Theorem \ref{van der Corput's} gives the complementary
range.

For \((a,q)=1\), define the normalized local factor
\begin{equation}\label{eq:nonflat-local-factor}
\begin{aligned}
   \mathfrak S_b(q,a)
   ={}&\frac{\phi(W)}{\phi(qW)\sigma_W}
       \sum_{h\in H_b}
       \sum_{\substack{\ell\bmod q\\(W\ell+h,qW)=1}}
       e\!\left(
          \frac{((W\ell+h)^2-b)a}{qW}
       \right).
\end{aligned}
\end{equation}
For \(q=1\), the definition gives \(\mathfrak S_b(1,0)=1\).

\begin{lemma}[Circle-method bounds for the non-flat majorant]
\label{lem:nonflat-circle-estimates}
Fix \(A,B\geq1\) and \(\epsilon>0\).  For all sufficiently
large admissible pairs, uniformly over \(b\in\mathbb Z_{W,2}\) and over
\(\alpha=a/q+\beta\in\mathfrak M(q,a)\), one has
\begin{equation}\label{eq:nonflat-major-formula}
\begin{aligned}
   \widehat{\nu_b}(\alpha)
   ={}&\mathfrak S_b(q,a)\,N^{1/2}
       e\!\left(-\frac{b\beta}{W}\right)
       I_2\!\left(\beta,\sqrt{N+b/W}\right)
       +O_{A,B,C_w}(N(\log N)^{-B}).
\end{aligned}
\end{equation}
The local factor satisfies
\begin{equation}\label{eq:nonflat-local-bound}
   \mathfrak S_b(q,a)
   =
   \begin{cases}
      1,&q=1,\\
      0,&1<q<w,\\
      O_\epsilon(q^{-1/2+\epsilon}),&q\geq w.
   \end{cases}
\end{equation}
Consequently,
\begin{equation}\label{eq:nonflat-major-envelope}
   |\widehat{\nu_b}(\alpha)|
   \ll_{A,B,\epsilon,C_w}
      Nq^{-1/2+\epsilon}
      (1+N|\beta|)^{-1/2}
      +N(\log N)^{-B}.
\end{equation}
Moreover, if \(A\) is sufficiently large in terms of \(B\) and
\(C_w\), then
\begin{equation}\label{eq:nonflat-minor-envelope}
   \sup_{\alpha\in\mathfrak m}
   |\widehat{\nu_b}(\alpha)|
   \ll_{B,C_w}N(\log N)^{-B}.
\end{equation}
\end{lemma}

\begin{proof}
We give the argument because the weight \(\nu_b\) has a different
normalization from Zhao's flat weight.  Put $Y=(WN+b)^{1/2}.$
Extending the original range \(n\in[N]\) to \(p\leq Y\) can add
only a prime with \(0\leq p^2-b<W\).  Since \(W\mid p^2-b\), this
forces \(p^2=b\), so at most one term is added.  Hence
\eqref{eq:quadratic-root-count} gives
\begin{equation}\label{eq:nu-root-decomposition}
\begin{aligned}
 \widehat{\nu_b}(\alpha)
  ={}&\frac{\phi(W)}{W\sigma_W}(WN)^{1/2}
      \sum_{h\in H_b}
      \sum_{\substack{p\leq Y\\p\equiv h\, (W)}}
      (\log p)
      e\!\left(\frac{p^2-b}{W}\alpha\right)
      +O\!\left((WN)^{1/2}\log(2W)\right).
\end{aligned}
\end{equation}
By Lemma~\ref{lem:W-growth}, the error is \(N^{1/2+o(1)}\) and is
absorbed by \(N(\log N)^{-B}\) for every fixed \(B\).

Suppose that
\(\alpha=a/q+\beta\in\mathfrak M(q,a)\).  For each \(h\in H_b\),
split the progression \(p\equiv h\pmod W\) into the \(q\) distinct
classes
\[
   p\equiv W\ell+h\pmod{qW},
   \qquad \ell\pmod q.
\]
A class that is not coprime to \(qW\) can contain only a prime divisor
of \(q\), because \((h,W)=1\).  Thus all non-reduced classes together
contribute
\begin{equation}\label{eq:nonreduced-class-error}
   O\!\left((WN)^{1/2}\omega(q)\log Y\right)
   =N^{1/2+o(1)}
\end{equation}
to \eqref{eq:nu-root-decomposition}, uniformly for \(q\leq Q\).

For a reduced residue \(r\pmod{qW}\), we next make the
Siegel--Walfisz error accounting explicit.  We use the theorem \ref{SWT} uniformly for \((a,q)=1\) and \(q\leq(\log x)^C\).  
Set \(Y_0=Y(\log N)^{-B-3A-50}\). For sufficiently large \(N\), one has \(Y_0\geq2\).  Moreover, on the
range \(Y_0\leq t\leq Y\), Lemma~\ref{lem:W-growth} gives
\[
   qW\leq (\log N)^{A+1}\leq(\log t)^{A+2},
   \qquad
   \log t\asymp \log N.
\]
Apply this prime-theta form of the Siegel--Walfisz theorem on this
range with logarithmic saving \(B+4A+60\).  By partial summation we get
\[
\begin{aligned}
  \sum_{\substack{Y_0<p\leq Y\\p\equiv r\, (qW)}}
       (\log p)e(\beta p^2/W)
  ={}&\frac1{\phi(qW)}
       \int_{Y_0}^Y e(\beta t^2/W)\,dt \\
  &\quad
       +O_{A,B,C_w}\!\left(
          Y(\log N)^{-B-4A-60}
          \left(1+\frac{|\beta|Y^2}{W}\right)\right) \\
  ={}&\frac1{\phi(qW)}
       \int_{Y_0}^Y e(\beta t^2/W)\,dt
       +O_{A,B,C_w}\!\left(
          Y(\log N)^{-B-3A-60}\right),
\end{aligned}
\]
since \(|\beta|Y^2/W\ll|\beta|N\leq Q=(\log N)^A\).
The initial segment \(p\leq Y_0\) contributes at most
\(Y_0\log Y\), and extending the main integral to \([0,Y]\) costs at
most \(Y_0\).  Consequently,
\begin{equation}\label{eq:weighted-SW-nonflat}
\begin{aligned}
  \sum_{\substack{p\leq Y\\p\equiv r\, (qW)}}
       (\log p)e(\beta p^2/W)
  ={}&\frac1{\phi(qW)}I_2(\beta/W,Y)
       +O_{A,B,C_w}(Y(\log N)^{-B-3A-49}).
\end{aligned}
\end{equation}

There are at most \(q\sigma_W\) reduced classes.  After multiplication
by the prefactor in \eqref{eq:nu-root-decomposition}, the total of the
errors in \eqref{eq:weighted-SW-nonflat} is
\[
 \begin{aligned}
 &\ll
 \frac{\phi(W)}{W\sigma_W}(WN)^{1/2}
 (q\sigma_W)Y (\log N)^{-B-3A-49}\\
 &\ll q\phi(W)N (\log N)^{-B-3A-49}
 \ll_{A,B,C_w}N(\log N)^{-B},
 \end{aligned}
\]
where the final estimate uses \(q\leq (\log N)^A\) and, by
Lemma~\ref{lem:W-growth}, \(\phi(W)\leq W\leq \log N\) for sufficiently
large \(N\).
Inserting the main terms into \eqref{eq:nu-root-decomposition} therefore
yields
\[
\begin{aligned}
 \widehat{\nu_b}(\alpha)
  ={}&\frac{\phi(W)}{W\phi(qW)\sigma_W}(WN)^{1/2}
      e\!\left(-\frac{b\beta}{W}\right)
      I_2(\beta/W,Y)                                      \\
 &\quad\times
      \sum_{h\in H_b}
      \sum_{\substack{\ell\bmod q\\(W\ell+h,qW)=1}}
      e\!\left(
          \frac{((W\ell+h)^2-b)a}{qW}
      \right)
      +O_{A,B,C_w}(N(\log N)^{-B}).
\end{aligned}
\]
The exact change of variables \(t=W^{1/2}u\) gives
\[
   I_2(\beta/W,Y)
   =W^{1/2}I_2\!\left(\beta,\sqrt{N+b/W}\right),
\]
and hence proves \eqref{eq:nonflat-major-formula}.

To prove \eqref{eq:nonflat-local-bound}, we transfer Zhao's estimate
to our normalization. Let \(S_b^{\mathrm Z}(q,a)\) denote Zhao's
factor in \cite[(10)]{zhao}.  Expanding
\((W\ell+h)^2-b\) shows directly from
\eqref{eq:nonflat-local-factor} that
\[
   \mathfrak S_b(q,a)
   =\frac{\phi(W)}{\phi(qW)}S_b^{\mathrm Z}(q,a).
\]
Zhao uses the same modulus \(W=8\prod_{2<p<w}p\).  His finite
evaluation \cite[Lemma~4.2 and Corollary~4.3]{zhao} gives
\[
   S_b^{\mathrm Z}(q,a)=0
   \quad\text{if }q=2\text{ or }(q,W)\nmid2,
   \qquad
   |S_b^{\mathrm Z}(q,a)|\ll_\epsilon q^{1/2+\epsilon}.
\]
The remaining normalization is not harmless unless its size is recorded.
Prime by prime, for every $\epsilon>0$ we have
\[
   \frac{\phi(W)}{\phi(qW)}
   =q^{-1}\prod_{\substack{p\mid q\\p\nmid W}}
        \left(1-\frac1p\right)^{-1}
   \ll_\epsilon q^{-1+\epsilon}.
\]
Finally, suppose \(1<q<w\).  If \(q=2\), then
\(S_b^{\mathrm Z}(q,a)=0\) directly.  Otherwise, if \(q\) has an odd
prime divisor \(p\), then \(p<q<w\), so \(p\mid W\) and hence
\(p\mid(q,W)\); thus \((q,W)\nmid2\).  If \(q\) has no odd prime
divisor, then \(q\) is a power of \(2\), and since \(q\neq2\) we have
\(q\geq4\).  As \(8\mid W\), this gives \(4\mid(q,W)\), again so that
\((q,W)\nmid2\).  Hence \(\mathfrak S_b(q,a)=0\) for every
\(1<q<w\). This proves every case of
\eqref{eq:nonflat-local-bound}.  Combining the bound with
\eqref{eq:I2-bound} proves \eqref{eq:nonflat-major-envelope}.

It remains to prove the minor-arc estimate.  Let
\(\alpha\in\mathfrak m\).  Dirichlet approximation with upper bound
\(N/Q\) supplies coprime integers \(a,q\) such that
\[
   Q<q\leq N/Q,
   \qquad
   \left|\alpha-\frac aq\right|
      \leq\frac{Q}{qN}\leq\frac1N.
\]
 Put
\[
   d=(a,W),\qquad a_1=a/d,\qquad W_1=W/d.
\]
Since \((a,q)=1\), one has \((a_1,qW_1)=1\), and
\[
   \frac{a}{qW}=\frac{a_1}{qW_1}.
\]
%We now verify the hypotheses of Hua's estimate.  
Apply
\cite[Lemma~4.5]{zhao} with saving parameter \(B+20\) and with the
bound \(W\leq(\log t)^2\); let \(\sigma=\sigma(B+20,2)\) be the
resulting denominator exponent.  Choose
\begin{equation}\label{eq:minor-A-choice}
   A\geq 2B+2\sigma+70
\end{equation}
and put \(t_0=Y(\log N)^{-(B+30)}\).  Uniformly for
\(t_0\leq t\leq Y\), Lemma~\ref{lem:W-growth} gives
\(\log t\asymp \log N\), \(W\leq(\log t)^2\), and
\[
   (\log t)^\sigma\leq qW_1
   \leq\frac{Y^2}{(\log N)^A}
   \leq t^2(\log N)^{-(A-2B-60)}
   \leq t^2(\log t)^{-\sigma}.
\]
Thus, for every \(h\in H_b\), Hua's estimate
\cite[Lemma~4.5]{zhao} gives
\begin{equation}\label{eq:Hua-terminal-bound}
  \sum_{\substack{p\leq t\\p\equiv h\, (W)}}
       e\!\left(\frac{a_1p^2}{qW_1}\right)
  \ll_{B,C_w}\frac{t}{W (\log N)^{B+20}}
  \qquad(t_0\leq t\leq Y).
\end{equation}
Recall that \(\beta=\alpha-a/q\).  The derivative of
\(\log t\,e(\theta t^2/W)\) is bounded by
\begin{equation}\label{eq:minor-partial-summation-derivative}
   \frac1t+\frac{|\beta|t\log t}{W}.
\end{equation}
Partial summation on \([t_0,Y]\), together with
\eqref{eq:Hua-terminal-bound}, gives
\begin{align*}
 &\sum_{\substack{t_0<p\leq Y\\p\equiv h\, (W)}}
       (\log p)e(\alpha p^2/W)                                    \\
 &\qquad\ll
   \frac{Y}{W(\log N)^{B+19}}
   +\int_{t_0}^Y\frac{t}{W(\log N)^{B+20}}
      \left(\frac1t+\frac{|\beta|t\log N}{W}\right)dt                 \\
 &\qquad\ll_{B,C_w}
   \frac{Y}{W(\log N)^{B+18}},
\end{align*}
where in the last step we used \(|\beta|\leq N^{-1}\) and
\(Y^2=WN+b\asymp WN\).  The initial range is bounded trivially by
\[
   \sum_{\substack{p\leq t_0\\p\equiv h\, (W)}}\log p
      \ll Y(\log N)^{-B-29}.
\]
Consequently, for every root class,
\begin{equation}\label{eq:minor-root-class-final}
  \sum_{\substack{p\leq Y\\p\equiv h\, (W)}}
       (\log p)e(\alpha p^2/W)
  \ll_{B,C_w}
       \frac{Y}{W(\log N)^{B+18}}+Y(\log N)^{-B-29}.
\end{equation}
Summing over the \(\sigma_W\) roots and multiplying by the prefactor
in \eqref{eq:nu-root-decomposition}, the two terms in
\eqref{eq:minor-root-class-final} contribute respectively
\[
   \ll \frac{\phi(W)}{W}\frac{N}{(\log N)^{B+18}}
   \quad\text{and}\quad
   \ll \phi(W)N (\log N)^{-B-29}.
\]
Both are \(O_{B,C_w}(N(\log N)^{-B})\) by Lemma~\ref{lem:W-growth}.
Together with the endpoint error in \eqref{eq:nu-root-decomposition},
this proves \eqref{eq:nonflat-minor-envelope}.
\end{proof}

We next record the two elementary estimates for the reference
measure that are needed to compare it with \(\nu_b\).

\begin{lemma}[Fourier estimates for the quadratic reference measure]
\label{lem:mu-fourier-estimates}
For every \(\alpha\in\mathbb T\),
\begin{equation}\label{eq:mu-continuous-approx}
 \left|
   \widehat{\mu_N}(\alpha)
   -N^{1/2}I_2(\alpha,N^{1/2})
 \right|
 \ll N\|\alpha\|_{\mathbb T}.
\end{equation}
Moreover,
\begin{equation}\label{eq:mu-global-decay}
   |\widehat{\mu_N}(\alpha)|
   \ll N\bigl(1+N\|\alpha\|_{\mathbb T}\bigr)^{-1/2}.
\end{equation}
\end{lemma}

\begin{proof}
Since
\[
 \sqrt n-\sqrt{n-1}
 =\int_{n-1}^n\frac{du}{2\sqrt u},
\]
where the integral at zero is improper, the difference in
\eqref{eq:mu-continuous-approx} equals
\[
 N^{1/2}\sum_{n=1}^N
 \int_{n-1}^n
   \bigl(e(n\alpha)-e(u\alpha)\bigr)\frac{du}{2\sqrt u}.
\]
Taking the representative of \(\alpha\) in \([-1/2,1/2]\), the
integrand is
\(O(\|\alpha\|_{\mathbb T}u^{-1/2})\).  Integration proves
\eqref{eq:mu-continuous-approx}.

For the second estimate, put \(\delta=\|\alpha\|_{\mathbb T}\).
The trivial bound gives \(O(N)\).  If \(\delta>N^{-1}\), use
\[
 \sqrt n-\sqrt{n-1}
 =\frac1{2\sqrt n}+O(n^{-3/2}).
\]
The classical bound
\[
   \sum_{n\leq N}n^{-1/2}e(n\alpha)
   \ll\delta^{-1/2}
\]
follows by treating \(n\leq\delta^{-1}\) trivially and applying
summation by parts, with
\(|\sum_{n\leq x}e(n\alpha)|\ll\delta^{-1}\), on the remaining
range.  The \(O(n^{-3/2})\) terms contribute \(O(N^{1/2})\) after
the outer factor \(N^{1/2}\).  Hence
\(|\widehat{\mu_N}(\alpha)|\ll(N/\delta)^{1/2}\), which combines
with the trivial estimate to give \eqref{eq:mu-global-decay}.
\end{proof}

\begin{proof}[Proof of Lemma~\ref{prop:Pseudorandomness}]
Fix \(B\geq2\), choose \(A\) sufficiently large for
Lemma~\ref{lem:nonflat-circle-estimates}, and put \(Q=(\log N)^A\).
On the principal arc \(\mathfrak M(1,0)\), write \(\alpha=\beta\).
By \eqref{eq:nonflat-major-formula} and
\eqref{eq:nonflat-local-bound},
\[
 \widehat{\nu_b}(\beta)
 =N^{1/2}e(-b\beta/W)
   I_2\!\left(\beta,\sqrt{N+b/W}\right)
   +O(N(\log N)^{-B}).
\]
Since \(1\leq b<W\),
\[
   0<\sqrt{N+b/W}-\sqrt N
     =\frac{b/W}{\sqrt{N+b/W}+\sqrt N}
     \leq N^{-1/2}.
\]
The endpoint change in the oscillatory integral therefore costs at
most one.  Moreover,
\[
   |e(-b\beta/W)-1|\ll |\beta|,
   \qquad
   N^{1/2}\left|I_2\!\left(\beta,\sqrt{N+b/W}\right)\right|\ll N,
\]
so removing the displacement phase costs \(O(N|\beta|)=O(Q)\).
Finally, \eqref{eq:mu-continuous-approx} contributes another
\(O(N|\beta|)=O(Q)\).  Thus
\begin{equation}\label{eq:pseudo-principal-arc}
 \sup_{\alpha\in\mathfrak M(1,0)}
 |\widehat{\nu_b}(\alpha)-\widehat{\mu_N}(\alpha)|
 \ll N(\log N)^{-B}+Q.
\end{equation}

If \(\alpha\in\mathfrak M(q,a)\) with \(2\leq q\leq Q\), then
\eqref{eq:nonflat-local-bound} and the trivial estimate in
\eqref{eq:I2-bound} give, for example with \(\epsilon=1/4\),
\[
   |\widehat{\nu_b}(\alpha)|
   \ll Nw^{-1/4}+N(\log N)^{-B}.
\]
Also
\[
   \|\alpha\|_{\mathbb T}
   \geq\frac1q-\frac{Q}{qN}
   \geq\frac1{2q}
\]
for large \(N\), so \eqref{eq:mu-global-decay} gives
\[
   |\widehat{\mu_N}(\alpha)|\ll(Nq)^{1/2}\leq(NQ)^{1/2}.
\]
Finally, on \(\mathfrak m\), the minor-arc estimate and the fact that
\(\mathfrak M(1,0)\) is absent give
\[
   |\widehat{\nu_b}(\alpha)|\ll N(\log N)^{-B},
   \qquad
   |\widehat{\mu_N}(\alpha)|\ll NQ^{-1/2}.
\]
Together with \eqref{eq:pseudo-principal-arc}, these estimates imply
\begin{equation}\label{eq:direct-pseudorandom-error}
 \|\widehat{\nu_b}-\widehat{\mu_N}\|_\infty
 \ll_{A,B,C_w}
 Nw^{-1/4}+NQ^{-1/2}+N(\log N)^{-B}+(NQ)^{1/2}+Q.
\end{equation}
The right-hand side is \(o(N)\), uniformly in \(b\), as
\(w,N\to\infty\) through admissible pairs.  Choose \(w_\eta\)
so that the first term is at most \(\eta N/2\), and then choose
\(N_\eta\) so that the remaining terms are at most
\(\eta N/2\).  This proves the lemma.
\end{proof}

\subsection{Restriction estimate}
\label{sec:restriction}

The restriction argument follows Bourgain's large-spectrum method,
as organized for prime squares by Zhao \cite[Section~5]{zhao}.  The
non-flat majorant has a quadratic oscillatory integral in its major
arcs, so we record the slightly more flexible version of the abstract
envelope lemma that is needed here.

\begin{lemma}[Truncated divisor moments]
\label{lem:truncated-divisor-moments}
Let \(r\geq1\) be fixed, let \(M,Q\geq2\), and define
\[
   D_Q(k)=\sum_{\substack{1\leq d\leq Q\\d\mid k}}1,
   \qquad D_Q(0)=\lfloor Q\rfloor.
\]
If \(Q^r\leq M\), then
\begin{equation}\label{eq:truncated-divisor-moment}
   \sum_{|k|\leq2M}D_Q(k)^r
   \ll_r M\bigl(\log(2Q)\bigr)^{C_r}
\end{equation}
for a constant \(C_r\) depending only on \(r\).
\end{lemma}

\begin{proof}
The contribution of \(k=0\) is at most \(Q^r\leq M\).  For
\(k\ne0\), expansion of the \(r\)-th power and the fact that $[d_1,\cdots, d_r]\le Q^r\le M$ gives
\[
\begin{aligned}
 \sum_{0<|k|\leq2M}D_Q(k)^r
 &\leq
  4M\sum_{1\leq d_1,\ldots,d_r\leq Q}
       \frac1{[d_1,\ldots,d_r]}.
\end{aligned}
\]
All terms are nonnegative, so the last sum is at most
\begin{align}\label{e:eulerprod}
    \prod_{p\leq Q}
 \left(
   \sum_{e_1,\ldots,e_r\geq0}p^{-\max(e_1,\ldots,e_r)}
 \right).
\end{align}
For each prime \(p\), the local factor equals
\[
  1+\sum_{m\geq1}\bigl((m+1)^r-m^r\bigr)p^{-m}
  =1+O_r(p^{-1}).
\]
Using \(\sum_{p\leq Q}p^{-1}\ll\log\log(2Q)\), we can bound \eqref{e:eulerprod} by
\(O_r((\log(2Q))^{C_r})\).  This proves
\eqref{eq:truncated-divisor-moment}.
\end{proof}

Motivated by \cite[Lemma 5.1]{zhao}, we prove Bourgain's envelope lemma. 

\begin{lemma}\label{lem:bourgain-envelope-general}
Fix \(\gamma>1\) and \(\epsilon>0\). Let \(M\geq2\) be an integer
and let \(Q,T\geq2\), with
\(Q^{C_\epsilon}\leq M\) for a sufficiently large constant
\(C_\epsilon\). Define
\[
    \mathcal G_{\gamma,Q,M}(\alpha)
    =
    \sum_{q\leq Q}\frac1q\sum_{a\bmod q}
    \bigl(1+M\|\alpha-a/q\|_{\mathbb T}\bigr)^{-\gamma}.
\]
Suppose \(g\colon\mathbb T\to\mathbb C\) satisfies
\[
    |g(\alpha)|
    \leq
    \max\{\mathcal G_{\gamma,Q,M}(\alpha),T^{-1}\}.
\]
Let \(t_1,\ldots,t_R\in\mathbb T\) be \(M^{-1}\)-separated and
assume
\[
    \sum_{r,s=1}^R|g(t_r-t_s)|\geq\eta R^2,
    \qquad
    \eta\geq2/T.
\]
Then there exists $C_{\gamma,\epsilon}$ such that
\begin{equation}\label{eq:bourgain-envelope-conclusion}
    R\ll_{\gamma,\epsilon}
    \eta^{-1}T^\epsilon
    \bigl(\log(2Q)\bigr)^{C_{\gamma,\epsilon}}.
\end{equation}
\end{lemma}

\begin{proof}
We give the argument together with
Lemma~\ref{lem:truncated-divisor-moments} used at its only arithmetic
step. For brevity, write
\[
    \mathcal G=\mathcal G_{\gamma,Q,M},
    \qquad
    F(\alpha)=(1+M\|\alpha\|_{\mathbb T})^{-\gamma}.
\]
Splitting according to whether \(|g(t_r-t_s)|\leq T^{-1}\), we obtain
\[
\begin{aligned}
    \eta R^2
    &\leq \sum_{r,s=1}^R |g(t_r-t_s)| \\
    &\leq
    \sum_{\substack{1\leq r,s\leq R\\
                    |g(t_r-t_s)|\leq T^{-1}}}
    |g(t_r-t_s)|
    +
    \sum_{\substack{1\leq r,s\leq R\\
                    |g(t_r-t_s)|>T^{-1}}}
    |g(t_r-t_s)| \\
    &\leq \frac{R^2}{T}
    +
    \sum_{\substack{1\leq r,s\leq R\\
                    |g(t_r-t_s)|>T^{-1}}}
    \mathcal G(t_r-t_s) \\
    &\leq \frac{\eta R^2}{2}
    +\sum_{r,s=1}^R\mathcal G(t_r-t_s).
\end{aligned}
\]
Absorbing the first term on the right gives
\begin{equation}\label{eq:large-spectrum-G-sum}
    \eta R^2\ll\sum_{r,s=1}^R\mathcal G(t_r-t_s).
\end{equation}
Let
\[
    \psi(\alpha)
    =
    \left|\frac1{2M+1}\sum_{|k|\leq M}e(k\alpha)\right|^2,
    \qquad
    \sigma(\alpha)=\sum_{r=1}^R\psi(\alpha-t_r).
\]
The Fej\'er-kernel bounds give
\[
    \psi(\alpha)\ll
    \min\{1,(M\|\alpha\|_{\mathbb T})^{-2}\},
    \qquad
    \|\psi\|_1\asymp M^{-1},
\]
and \(\psi\gg1\) on a fixed multiple of the interval
\(\|\alpha\|_{\mathbb T}\leq M^{-1}\). The \(M^{-1}\)-separation of the \(t_r\)'s and the decay of \(\psi\)
give, uniformly in \(\alpha\),
\[
    \sigma(\alpha)
    =\sum_{r=1}^R\psi(\alpha-t_r)
    \ll
    \sum_{j\geq0}(1+j)^{-2}
    \ll1.
\]
Moreover,
\[
    \|\sigma\|_1
    =\sum_{r=1}^R\|\psi(\cdot-t_r)\|_1
    =R\|\psi\|_1
    \asymp R/M.
\]
Hence
\[
    \|\sigma\|_\infty\ll1,
    \qquad
    \|\sigma\|_2^2
    \leq\|\sigma\|_\infty\|\sigma\|_1
    \ll R/M.
\]
Write \(\widetilde\sigma(\alpha)=\sigma(-\alpha)\). Since \(\psi\) is
even,
\[
    (\widetilde\sigma*\sigma)(\alpha)
    =
    \sum_{r,s=1}^R
    (\psi*\psi)\bigl(\alpha-(t_r-t_s)\bigr).
\]
As \(\psi\gg1\) on \(\|\alpha\|_{\mathbb T}\leq cM^{-1}\) for some
absolute \(c>0\), we have
\[
    (\psi*\psi)(\alpha)\gg M^{-1}
    \qquad
    \text{whenever }\|\alpha\|_{\mathbb T}\leq c'M^{-1}
\]
for a sufficiently small \(c'>0\). Moreover,
\[
    F(\alpha)\asymp F(\alpha')
    \qquad\text{if}\qquad
    \|\alpha-\alpha'\|_{\mathbb T}\ll M^{-1},
\]
and hence the same is true for \(\mathcal G\). Therefore, for each
pair \((r,s)\),
\[
\begin{aligned}
    \int_{\mathbb T}\mathcal G(\alpha)
    (\psi*\psi)\bigl(\alpha-(t_r-t_s)\bigr)\,d\alpha
    &\gg
    \frac1M
    \int_{\|\alpha-(t_r-t_s)\|_{\mathbb T}\leq c'/M}
    \mathcal G(\alpha)\,d\alpha \\
    &\gg
    \frac1{M^2}\mathcal G(t_r-t_s).
\end{aligned}
\]
Summing over \(r,s\) and using
\eqref{eq:large-spectrum-G-sum}, we obtain
\begin{equation}\label{eq:bourgain-fourier-start}
\begin{aligned}
    \eta R^2M^{-2}
    &\ll
    \int_{\mathbb T}\mathcal G(\alpha)
    (\widetilde\sigma*\sigma)(\alpha)\,d\alpha \\
    &=
    \sum_{|k|\leq2M}
    \widehat{\mathcal G}(k)|\widehat\sigma(k)|^2 \\
    &\leq
    \sum_{|k|\leq2M}
    |\widehat{\mathcal G}(k)|\,|\widehat\sigma(k)|^2.
\end{aligned}
\end{equation}
As \(\gamma>1\), we have $|\widehat F(k)|\leq\|F\|_1\ll_\gamma M^{-1}.$
Moreover,
\[
\begin{aligned}
    \widehat{\mathcal G}(k)
    &=
    \sum_{q\leq Q}\frac1q
    \sum_{a\bmod q}
    \int_{\mathbb T}F(\alpha-a/q)e(-k\alpha)\,d\alpha \\
    &=
    \widehat F(k)
    \sum_{q\leq Q}\frac1q
    \sum_{a\bmod q}e(-ka/q).
\end{aligned}
\]
By orthogonality of the complete residue sum,
% \[
%     \sum_{a\bmod q}e(-ka/q)
%     =
%     \begin{cases}
%         q,&q\mid k,\\
%         0,&q\nmid k,
%     \end{cases}
% \]
% and therefore
\[
    \widehat{\mathcal G}(k)
    =
    \widehat F(k)
    \sum_{\substack{q\leq Q\\q\mid k}}1,
\]
where the sum has value \(\lfloor Q\rfloor\) at \(k=0\). Thus
\begin{equation}\label{eq:bourgain-divisor-sum}
    \eta R^2M^{-1}
    \ll_\gamma
    \sum_{|k|\leq2M}
    |\widehat\sigma(k)|^2D_Q(k).
\end{equation}

Choose an integer \(r\geq2\), to be fixed in terms of \(\epsilon\),
and choose \(C_\epsilon\geq r\), so that
Lemma~\ref{lem:truncated-divisor-moments} applies. Split the sum in
\eqref{eq:bourgain-divisor-sum} according as
\(D_Q(k)\leq Q_0\) or \(D_Q(k)>Q_0\). The first part is at most
\[
    Q_0\|\sigma\|_2^2\ll Q_0R/M.
\]
For the second part, use
\[
    |\widehat\sigma(k)|
    \leq\|\sigma\|_1\ll R/M,
\]
the inequality
\[
    D_Q(k)\mathbf1_{D_Q(k)>Q_0}
    \leq Q_0^{-(r-1)}D_Q(k)^r,
\]
and Lemma~\ref{lem:truncated-divisor-moments}. This gives
\[
    \sum_{\substack{|k|\leq2M\\D_Q(k)>Q_0}}
    |\widehat\sigma(k)|^2D_Q(k)
    \ll_r
    \frac{R^2}{M}Q_0^{-(r-1)}
    \bigl(\log(2Q)\bigr)^{C_r}.
\]
After multiplying \eqref{eq:bourgain-divisor-sum} by \(M/R\), we
obtain
\[
    \eta R
    \ll_{\gamma,r}
    Q_0+
    Q_0^{-(r-1)}
    \bigl(\log(2Q)\bigr)^{C_r}R.
\]
Choose
\[
    Q_0
    =
    C_{\gamma,r}
    T^{1/(r-1)}
    \bigl(\log(2Q)\bigr)^{C_r/(r-1)}
\]
with a sufficiently large constant. Since \(\eta\geq2/T\), the
second term is absorbed into the left-hand side, and hence
\[
    R\ll_{\gamma,r}
    \eta^{-1}T^{1/(r-1)}
    \bigl(\log(2Q)\bigr)^{C_r/(r-1)}.
\]
Finally choose \(r\) so large that \(1/(r-1)\leq\epsilon\). This
proves \eqref{eq:bourgain-envelope-conclusion}.
\end{proof}
\begin{lemma}[Discrete fourth moment]
\label{lem:nonflat-fourth-moment}
There is \(C_4=C_4(C_w)\) such that, uniformly in
\(b\in\mathbb Z_{W,2}\), every \(h\colon[N]\to\mathbb C\) with
\(|h|\leq\nu_b\) satisfies
\begin{equation}\label{eq:nonflat-discrete-fourth}
   \sum_{j=1}^N
      \left|\widehat h\!\left(\frac jN\right)\right|^4
   \ll_{C_w}N^4(\log N)^{C_4}.
\end{equation}
\end{lemma}

\begin{proof}
By the definition of \(\nu_b\), if \(\nu_b(n)\neq0\), then
\(Wn+b=p^2\) for some prime \(p\), and
\[
   \nu_b(n)
   =
   \frac{|\mathbb Z_{W,2}|}{W}(WN)^{1/2}\log p.
\]
Since \(|\mathbb Z_{W,2}|/W\leq1\) and
\(p^2=Wn+b\leq WN+b\), we have
\[
   \nu_b(n)
   \leq (WN)^{1/2}\log (WN+b).
\]
Lemma~\ref{lem:W-growth} gives \(W\leq(\log N)^C\) and
\(\log(WN+b)\ll\log N\), after increasing \(C=C(C_w)\) if
necessary. Hence
\begin{align}\label{e:supboundnub}
    \sup_n\nu_b(n)
   \ll_{C_w}N^{1/2}(\log N)^C.
\end{align}
For \(|u|<N\), let
\[
   r_b(u)
   =
   |\{(m,n)\in[N]^2:m-n=u,\ 
   \nu_b(m)\nu_b(n)\neq0\}|.
\]
When \(u=0\), we have \(m=n\), and each point in the support of
\(\nu_b\) corresponds to a prime \(x\) with $x^2=Wm+b\leq WN+b.$
Therefore
\[
   r_b(0)
   \leq |\{x\geq1:x^2\leq WN+b\}|
   \leq (WN+b)^{1/2}
   \ll(WN)^{1/2}.
\]
Now suppose \(u\neq0\). For every pair counted by \(r_b(u)\), write
\[
   Wm+b=x^2,\qquad Wn+b=y^2,
\]
with \(x,y\) positive primes. Since \(m-n=u\),
\[
   Wu
   =W(m-n)
   =x^2-y^2
   =(x-y)(x+y).
\]
Thus \(d=x-y\) is a signed divisor of \(Wu\), and once \(d\) is
fixed we must have
\[
   x=\frac12\left(d+\frac{Wu}{d}\right),
   \qquad
   y=\frac12\left(\frac{Wu}{d}-d\right).
\]
Hence each signed divisor \(d\mid Wu\) gives at most one pair
\((x,y)\), and therefore
\[
   r_b(u)\ll\tau(W|u|)\qquad(u\neq0).
\]
The elementary inequalities
\(\tau(W|u|)\leq\tau(W)\tau(|u|)\) and
\(\sum_{u\leq N}\tau(u)^2\ll N(\log N)^3\), together with
Lemma~\ref{lem:W-growth}, therefore give
after increasing \(C\) if necessary,
\begin{equation}\label{eq:support-difference-second-moment}
   \sum_{|u|<N}r_b(u)^2
      \ll_{C_w}N(\log N)^C.
\end{equation}
The term \(u=0\) is admissible because \(r_b(0)^2\ll WN\).

\medskip

Expanding the fourth moment and using orthogonality on
\(N^{-1}\mathbb Z/N\mathbb Z\), we obtain
\[
\begin{aligned}
 \sum_{j=1}^N\left|\widehat h(j/N)\right|^4
 &=N\sum_{\rho\bmod N}
   \left|
      \sum_{\substack{m,n\in[N]\\m-n\equiv\rho\,(\mathrm{mod}\,N)}}
      h(m)\overline{h(n)}
   \right|^2 \leq
 2N\sum_{|u|<N}
   \left|
      \sum_{\substack{m,n\in[N]\\m-n=u}}
      h(m)\overline{h(n)}
   \right|^2,
\end{aligned}
\]
since each residue class modulo \(N\) has at most two integers $-N<u<N$.
  Since \(|h|\leq\nu_b\), for each \(u\),
\[
   \left|
      \sum_{\substack{m,n\in[N]\\m-n=u}}
      h(m)\overline{h(n)}
   \right|
   \leq
   \bigl(\sup_n\nu_b(n)\bigr)^2 r_b(u).
\]
Hence
\[
   \sum_{j=1}^N\left|\widehat h(j/N)\right|^4
   \leq
   2N\bigl(\sup_n\nu_b(n)\bigr)^4
   \sum_{|u|<N}r_b(u)^2.
\]
Combining this with \eqref{e:supboundnub} and 
\eqref{eq:support-difference-second-moment} yields
\[
 \sum_{j=1}^N\left|\widehat h(j/N)\right|^4
 \ll_{C_w}N^4(\log N)^{5C}.
\]
Taking \(C_4=5C\) proves the lemma.
\end{proof}

\begin{proof}[Proof of Lemma~\ref{prop:RestrictionEstimate}]
Fix \(q>4\), and let \(|h|\leq\nu_b\).  Applying
Lemma~\ref{lem:nonflat-circle-estimates} at \(\alpha=0\), with
\(q=1\), \(a=0\), and \(\beta=0\), gives
\[
   \sum_{n\leq N}\nu_b(n)
   =N^{1/2}\sqrt{N+b/W}+O_{C_w}(N(\log N)^{-2})
   =N+O_{C_w}(N(\log N)^{-2}+1).
\]
Consequently, for every sufficiently large admissible pair, uniformly in \(b\) we have
\begin{equation}\label{eq:nonflat-mass-bound}
   \sum_{n\leq N}\nu_b(n)\leq2N.
\end{equation}
% For \(N\) below this fixed threshold there are only finitely many
% possible choices of \(N\), of the modulus \(W\), and of
% \(b\in\mathbb Z_{W,2}\); enlarging the final restriction constant
% absorbs all of them.
We first prove a discrete restriction estimate.  Choose
\[
    4<r_0<r<q
\]
and put \(\gamma=r_0/4>1\).
Let
\[
   \mathcal R_h(u)
   =\left|\left\{1\leq j\leq N:
       \left|\widehat h(j/N)\right|\geq uN\right\}\right|.
\]
We claim that
\begin{equation}\label{eq:restriction-level-bound}
   \mathcal R_h(u)\ll_{q,C_w}u^{-r}
   \qquad(u>0).
\end{equation}
By \eqref{eq:nonflat-mass-bound}, this is immediate for \(u>2\),
so throughout the level-set argument we may assume \(0<u\leq2\).
Let \(C_4\) be the exponent in
Lemma~\ref{lem:nonflat-fourth-moment}.  Choose a large constant
\(B_0=B_0(q,C_w)\) such that
\[
   4+C_4/B_0<r.
\]
If \(u<(\log N)^{-B_0}\), then
\eqref{eq:nonflat-discrete-fourth} gives
\[
   \mathcal R_h(u)
   \ll u^{-4}(\log N)^{C_4}
   \leq u^{-4-C_4/B_0}
   \ll u^{-r}.
\]

Suppose now that \(u\geq (\log N)^{-B_0}\), and let
\(t_1,\ldots,t_R\) be the grid points counted by
\(\mathcal R_h(u)\).  Choose unit complex numbers \(c_j\) so that
\(c_j\widehat h(t_j)=|\widehat h(t_j)|\).  Choose unit complex numbers \(c_j\) so that
\(c_j\widehat h(t_j)=|\widehat h(t_j)|\). Since
\(|\widehat h(t_j)|\geq uN\) for \(1\leq j\leq R\),
\[
   uRN = |\mathcal R_h(u)| uN
   \leq\sum_{j=1}^R|\widehat h(t_j)|
   =\left|\sum_{j=1}^Rc_j\widehat h(t_j)\right|
   =\left|\sum_nh(n)\sum_{j=1}^Rc_je(nt_j)\right|.
\]
Therefore, by Cauchy's inequality, \eqref{eq:nonflat-mass-bound}, and
\(|h|\leq\nu_b\),
\[
\begin{aligned}
  u^2R^2N^2
  &\leq
   \left|\sum_nh(n)\sum_{j=1}^Rc_je(nt_j)\right|^2\\
  &\leq
   \left(\sum_n|h(n)|\right)
   \left(\sum_n\nu_b(n)
      \left|\sum_{j=1}^Rc_je(nt_j)\right|^2\right)\\
  &\ll N\sum_{j,k=1}^R
      |\widehat{\nu_b}(t_j-t_k)|.
\end{aligned}
\]
Hölder's inequality therefore yields
\begin{equation}\label{eq:restriction-TTstar-power}
   u^{r_0}R^2
   \ll_{r_0}\sum_{j,k=1}^R
      \left(N^{-1}|\widehat{\nu_b}(t_j-t_k)|\right)^{r_0/2}.
\end{equation}
We now prove \eqref{eq:restriction-level-bound} in the remaining range
\(u\geq(\log N)^{-B_0}\).  
We shall apply Lemma~\ref{lem:bourgain-envelope-general} to the
powered Fourier transform of \(\nu_b\), at the level-dependent scale
\(U\asymp u^{-r_0}\).
Recall that \(\gamma=r_0/4>1\). Apply Lemma~\ref{lem:nonflat-circle-estimates} with $\epsilon=\frac{r_0-4}{4r_0},$
minor-arc saving \((\log N)^{-D}\), where \(D>2B_0+2\), and major
arcs with denominators at most \((\log N)^A\). Choose \(A\)
sufficiently large in terms of \(D,r_0B_0,C_w\), with
\(A>r_0B_0+10\). Thus, if
\(\alpha=a/v+\beta\in\mathfrak M(v,a)\), then
\[
    N^{-1}|\widehat{\nu_b}(\alpha)|
    \ll
    v^{-1/2+\epsilon}(1+N|\beta|)^{-1/2}
    +(\log N)^{-D},
\]
while on the minor arcs
\[
    N^{-1}|\widehat{\nu_b}(\alpha)|
    \ll(\log N)^{-D}.
\]
At the present level put
\[
    U=\max\{2,\lceil Cu^{-r_0}\rceil\},
\]
where \(C=C(r_0)\) will be chosen sufficiently large. Since
\(0<u\leq2\),
\begin{equation}\label{eq:level-dependent-U}
    Cu^{-r_0}\leq U\leq2Cu^{-r_0}.
\end{equation}
Moreover, \(u\geq(\log N)^{-B_0}\) gives
\[
    U\ll_{r_0}(\log N)^{r_0B_0}\leq(\log N)^A.
\]

We claim that, after multiplication by a fixed constant, the
\(r_0/2\)-th power of the normalized Fourier coefficient satisfies
the Bourgain envelope
\[
    \left(N^{-1}|\widehat{\nu_b}(\alpha)|\right)^{r_0/2}
    \ll
    \max\{\mathcal G_{\gamma,U,N}(\alpha),U^{-1}\},
    \qquad
    \gamma=\frac{r_0}{4}.
\]
Indeed, on a major arc \(\alpha=a/v+\beta\), the preceding estimate
and
\((x+y)^{r_0/2}\ll_{r_0}x^{r_0/2}+y^{r_0/2}\) give
\begin{align}
 \left(N^{-1}|\widehat{\nu_b}(\alpha)|\right)^{r_0/2}
 \ll{}&
 v^{-(r_0+4)/8}
 (1+N|\beta|)^{-r_0/4}
 +(\log N)^{-Dr_0/2} \notag\\
 \leq{}&
 v^{-1}(1+N|\beta|)^{-\gamma}
 +(\log N)^{-Dr_0/2},
 \label{eq:major-arc-powered-envelope}
\end{align}
because \(r_0>4\). If \(v\leq U\), the first term on the right is
one of the summands defining
\(\mathcal G_{\gamma,U,N}(\alpha)\). If \(v>U\), it is at most
\(U^{-1}\).

On the minor arcs,
Lemma~\ref{lem:nonflat-circle-estimates} gives only the error
\(O((\log N)^{-Dr_0/2})\). By
\eqref{eq:level-dependent-U},
\[
    U^{-1}\geq(2C)^{-1}u^{r_0}
    \geq(2C)^{-1}(\log N)^{-r_0B_0},
\]
and our choice of \(D\) gives
\[
    (\log N)^{-Dr_0/2}=o(U^{-1}).
\]
Thus, after dividing by a sufficiently large constant
\(C_0=C_0(q,C_w)\), the function
\[
    g(\alpha)
    =C_0^{-1}
      \left(N^{-1}|\widehat{\nu_b}(\alpha)|\right)^{r_0/2}
\]
satisfies
\begin{equation}\label{eq:nonflat-level-envelope}
    |g(\alpha)|
    \leq
    \max\{\mathcal G_{\gamma,U,N}(\alpha),U^{-1}\}
    \qquad(\alpha\in\mathbb T).
\end{equation}

We next verify the mass hypothesis in
Lemma~\ref{lem:bourgain-envelope-general}. From
\eqref{eq:restriction-TTstar-power}, after incorporating the factor
\(C_0^{-1}\) into the constant, there is \(c_{r_0}>0\) such that
\[
    \sum_{j,k=1}^R|g(t_j-t_k)|
    \geq c_{r_0}u^{r_0}R^2.
\]
Set $\eta=c_{r_0}u^{r_0}.$
Choose \(C\) in the definition of \(U\) large enough that
\(C\geq2/c_{r_0}\). Since \(U\geq Cu^{-r_0}\),
\[
    \frac2U\leq\frac{2}{C}u^{r_0}\leq\eta.
\]
Thus the required lower bound \(\eta\geq2/U\) holds.
The points \(t_1,\ldots,t_R\) lie on the grid
\(N^{-1}\mathbb Z/N\mathbb Z\), so they are \(N^{-1}\)-separated.
Also \(U\) is polylogarithmic in \(N\); hence, for every fixed
\(\vartheta>0\), we have $U^{C_\vartheta}\leq N$
once \(N\) is sufficiently large. We may therefore apply
Lemma~\ref{lem:bourgain-envelope-general} with
\[
    M=N,\qquad Q=T=U.
\]
Using \(\eta\asymp u^{r_0}\) and \(U\asymp u^{-r_0}\), we obtain
\[
\begin{aligned}
    R
    &\ll_{r_0,\vartheta}
      \eta^{-1}U^\vartheta
      \bigl(\log(2U)\bigr)^{C_{r_0,\vartheta}}\\
    &\ll
      u^{-r_0-r_0\vartheta}
      \bigl(\log(4/u)\bigr)^{C_{r_0,\vartheta}}.
\end{aligned}
\]
Choose \(\vartheta>0\) so small that
\[
    r_0\vartheta<\frac{r-r_0}{3}.
\]
For every fixed \(C\),
\[
    \bigl(\log(4/u)\bigr)^C
    \ll_{C,r-r_0}u^{-(r-r_0)/3}
    \qquad(0<u\leq1).
\]
It follows that
\[
    R
    \ll_{q,C_w}
    u^{-r_0-2(r-r_0)/3}
    \ll_{q,C_w}u^{-r}.
\]
For \(1<u\leq2\), the same conclusion follows after changing the
implicit constant. This proves
\eqref{eq:restriction-level-bound} for all \(u>0\).

\medskip

We now pass from the level-set estimate to the discrete
\(L^q\)-bound. Since
\[
    |\widehat h(\alpha)|
    \leq\sum_n|h(n)|
    \leq\sum_n\nu_b(n)
    \leq2N,
\]
dyadic decomposition gives
\[
\begin{aligned}
    \sum_{j=1}^N
       \left|\widehat h(j/N)\right|^q
    &\ll_q
      N^q\sum_{k\leq1}
      2^{kq}\mathcal R_h(2^k)\\
    &\ll_{q,C_w}
      N^q\sum_{k\leq1}2^{k(q-r)}
    \ll_{q,C_w}N^q,
\end{aligned}
\]
since \(r<q\).

Finally, we pass from the grid to \(\mathbb T\). For
\(0\leq\theta\leq1/N\), put
\[
    h_\theta(n)=h(n)e(n\theta).
\]
Then \(|h_\theta|\leq\nu_b\), so the same discrete estimate applies,
and
\[
    \widehat{h_\theta}(j/N)
    =\widehat h(j/N+\theta).
\]
Integrating over \(\theta\in[0,1/N]\), we obtain
\[
\begin{aligned}
    \int_{\mathbb T}|\widehat h(\alpha)|^q\,d\alpha
    &=\int_0^{1/N}\sum_{j=1}^N
      |\widehat h(j/N+\theta)|^q\,d\theta\\
    &\ll_{q,C_w}N^{q-1}.
\end{aligned}
\]
Taking \(q\)-th roots gives
\[
    \|\widehat h\|_q
    \ll_{q,C_w}N^{1-1/q},
\]
which proves the restriction estimate uniformly in \(b\).
\end{proof}

\subsection{Mean-value estimate}

\begin{lemma}\label{lem:distribution-means}
Let \(d=\underline d(\mathcal P)\).  For every
\(\varepsilon>0\), there are
\(w_\varepsilon=w_\varepsilon(\varepsilon,C_w)\) and
\(N_\varepsilon=N_\varepsilon(\mathcal P,\varepsilon,C_w)\) such
that every admissible pair with
\(w\geq w_\varepsilon\) and \(N\geq N_\varepsilon\) satisfies
\[
    \frac1{|\mathbb Z_{W,2}|}
    \sum_{b\in\mathbb Z_{W,2}}m_b
    \geq d-\varepsilon
\]
and, uniformly in \(b\in\mathbb Z_{W,2}\),
\[
    0\leq m_b\leq1+\varepsilon.
\]
see \eqref{e:mbdef}.
\end{lemma}

\begin{proof}
Let \(Y=(WN)^{1/2}\).  Every prime
\[
    \frac{Y}{\log N}<p\leq Y
\]
is larger than \(w\) for sufficiently large \(N\), and hence is
coprime to \(W\).  Choose the representative
\(1\leq b<W\) of \(p^2\pmod W\), and set
\(n=(p^2-b)/W\).  The lower cutoff ensures \(n\geq1\), while
\(p\leq Y\) gives \(n\leq N\).  Thus the prime determines a unique
pair \(b\in\mathbb Z_{W,2}\), \(n\in[N]\).  Using the normalization
in \eqref{function:f_b} and the identity
\[
    \frac{1}{|\mathbb Z_{W,2}|N}
    \frac{|\mathbb Z_{W,2}|}{W}Y
    =\frac1Y,
\]
which follows from \(Y^2=WN\), we obtain
\[
\begin{aligned}
 \frac1{|\mathbb Z_{W,2}|}
 \sum_b m_b
 &\geq
 \frac1Y
 \sum_{\substack{Y/\log N<p\leq Y\\p\in\mathcal P}}\log p\\
 &\geq
 \frac{\log Y-\log\log N}{Y}
 |\mathcal P\cap[1,Y]|
 -\frac{\log Y-\log\log N}{Y}
 |\mathbb P\cap[1,Y/\log N]|.
\end{aligned}
\]
%It suffices to treat \(0<\varepsilon\leq1\), since the lower bound is
%trivial when \(\varepsilon>1\) and the pointwise estimate with tolerance
%one is then stronger than required.  
By the definition of
\(d=\underline d(\mathcal P)\), for all sufficiently large
\(Y\),
\[
   |\mathcal P\cap[1,Y]|
   \geq(d-\varepsilon/4)\pi(Y).
\]
The prime number theorem gives
\[
   \frac{\log Y-\log\log N}{Y}\pi(Y)=1+o(1).
\]
These relations are uniform over admissible pairs: by
Lemma~\ref{lem:W-growth},
\[
   \log Y=\frac12\log N+O_{C_w}(\log\log\log N)
          =\left(\frac12+o(1)\right)\log N.
\]
Thus, after increasing the large-\(N\) threshold, we may arrange
simultaneously that
\[
   \frac{\log Y-\log\log N}{Y}\pi(Y)\geq1-\varepsilon/4
\]
and
\[
   \frac{\log Y-\log\log N}{Y}\pi(Y/\log N)
      \leq\varepsilon/4.
\]
Because \(0\leq d\leq1\), the preceding lower bound is then at least
\[
   (d-\varepsilon/4)(1-\varepsilon/4)-\varepsilon/4
   \geq d-3\varepsilon/4
   \geq d-\varepsilon.
\]
All these estimates are uniform over the admissible choice of \(w\),
once \(N\) is sufficiently large in terms of
\(\mathcal P,\varepsilon\), and \(C_w\).

For the pointwise upper bound, \(0\leq f_b\leq\nu_b\).  Apply
Lemma~\ref{prop:Pseudorandomness} with tolerance \(\varepsilon\) and
evaluate at frequency \(0\).  Since \(\widehat{\mu_N}(0)=N\) by
\eqref{eq:mu-mass}, this gives
\[
   0\leq m_b
   \leq\frac{\widehat{\nu_b}(0)}N
   \leq1+\varepsilon
\]
uniformly in \(b\), provided
\(w\geq w_\varepsilon\) and \(N\) exceeds the corresponding
pseudorandomness threshold.  Taking \(N_\varepsilon\) to be the
maximum of the two large-\(N\) thresholds proves the lemma.
\end{proof}

\begin{proof}[Proof of Lemma~\ref{prop:MeanValue}]
Let \(d=\underline d(\mathcal P)>D_s\).  Choose
\[
   D_s<d_0<d,
\]
and then choose \(\rho>0\) so small that
\(d/(1+\rho)>d_0\).  Next choose \(\varepsilon>0\) satisfying
\[
   \varepsilon<\rho,
   \qquad
   \frac{d-\varepsilon}{1+\rho}>d_0.
\]
Lemma~\ref{lem:distribution-means} gives thresholds \(w_0,N_0\) such
that, for every admissible pair with \(w\geq w_0\) and
\(N\geq N_0\),
\[
    0\leq h(b)\coloneqq\frac{m_b}{1+\rho}\leq1,
    \qquad
    \frac1{|\mathbb Z_{W,2}|}\sum_bh(b)>d_0.
\]
Via the bijection \eqref{eq:quadratic-residue-splitting}, regard \(h\)
as a function on \(\mathbb Z_{W_0,2}\).
Recall \(\gamma_s\) from \eqref{eq:gamma-s} and \(L_s\) from
\eqref{eq:local-scores}.  If \(0\leq t<1\), put
\[
    F_t(b)=\frac{\max\{h(b)-t,0\}}{1-t}.
\]
Since \(h\leq t+(1-t)F_t\),
\[
    \frac1{|\mathbb Z_{W_0,2}|}\sum_bF_t(b)
    >
    \frac{d_0-t}{1-t}.
\]
The last quantity is greater than \(\gamma_s\) whenever
\[
    t<\frac{d_0-\gamma_s}{1-\gamma_s}.
\]
For \(s\in\{5,7,8,9\}\), we have \(d_0>D_s>\gamma_s\), so this
interval contains positive values of \(t\). Apply
Lemma~\ref{lemma:arithmetic-square} modulo \(W_0\), with the
prescribed residue \(r\pmod{W_0}\).  It gives
\(c_1,\ldots,c_s\in\mathbb Z_{W_0,2}\) such that
\[
   c_1+\cdots+c_s\equiv r\pmod{W_0},
   \qquad
   F_t(c_j)>0,
   \qquad
   \sum_{j=1}^sF_t(c_j)>L_s.
\]
The positivity of \(F_t(c_j)\) implies \(h(c_j)>t\) and
\[
   h(c_j)=t+(1-t)F_t(c_j).
\]
Thus the selected score satisfies
\[
    \sum_{j=1}^sh(c_j)>st+(1-t)L_s.
\]
As \(t\) tends from below to this endpoint, the right-hand side
tends to
\[
   \frac{s(d_0-\gamma_s)+L_s(1-d_0)}{1-\gamma_s}
   =
   \begin{cases}
      8d_0-3,&s=5,\\
      (49d_0-7)/6,&s=7,\\
      8d_0,&s=8,\\
      9d_0,&s=9.
   \end{cases}
\]
At \(d_0=D_s\), substituting the values of \(D_s\) from
\eqref{e:csdeltas} shows that each of these expressions equals
\(s/\sqrt2\). Since \(d_0>D_s\), the endpoint
value is strictly larger.  Choose \(t\) below and sufficiently close
to the endpoint that
\[
   st+(1-t)L_s>\frac{s}{\sqrt2}+3\zeta
\]
for some \(\zeta>0\), decreasing \(\zeta\), if necessary, so that
\(\zeta<t/2\).  The selected elements then satisfy
\[
    \sum_{j=1}^sh(c_j)>\frac{s}{\sqrt2}+3\zeta,
    \qquad h(c_j)>t>2\zeta.
\]

Lift each \(c_j\) through
\eqref{eq:quadratic-residue-splitting} to the unique
\(b_j\in\mathbb Z_{W,2}\).  Because \(r\equiv s\pmod {24}\), the
sum of these lifts is congruent to \(r\) both modulo \(W_0\) and
modulo \(24\).  The Chinese remainder theorem therefore gives
\[
    b_1+\cdots+b_s\equiv r\pmod W.
\]
Finally, \(m_{b_j}=(1+\rho)h(c_j)\), so
\[
   \sum_{j=1}^sm_{b_j}>
      \frac{s}{\sqrt2}+3\zeta,
   \qquad
   m_{b_j}>2\zeta.
\]
Replacing \(\zeta\) by a smaller positive constant gives exactly the
conclusion of Lemma~\ref{prop:MeanValue}.
\end{proof}
\section{A density \(7/10\) counterexample}

We construct a set of primes of relative density \(7/10\) for which
infinitely many integers congruent to \(5\pmod{24}\) are not sums of
five squares of primes from the set.  This shows that
\[
    \mathfrak d_5\geq\frac7{10}.
\]

By the Chinese remainder theorem,
\[
    \mathbb Z_{385,2}
    \cong\mathbb Z_{35,2}\times\mathbb Z_{11,2},
    \qquad
    \mathbb Z_{35,2}=\{1,4,9,11,16,29\},
    \qquad
    \mathbb Z_{11,2}=\{1,3,4,5,9\}.
\]
Let
\[
    B=\{1,4,9,29\},
    \qquad
    A=(B\times\mathbb Z_{11,2})\cup\{(11,1)\}.
\]
Thus
\[
    |A|=4\cdot5+1=21,
    \qquad
    |\mathbb Z_{385,2}|=6\cdot5=30.
\]

We claim that
\[
    (3,3)\notin5A.
\]
Suppose that five elements of \(A\) sum to \((3,3)\).  Let \(k\) be
the number of summands equal to \((11,1)\), and let \(a\) be the
number of the remaining summands whose first coordinate is \(1\).
Modulo \(5\), the numbers \(1\) and \(11\) are congruent to \(1\),
while \(4,9,29\) are congruent to \(-1\).  Hence
\[
    2(a+k)-5\equiv3\pmod5.
\]
Since \(0\leq a+k\leq5\), it follows that \(a+k=4\).  Thus exactly
one first coordinate belongs to \(\{4,9,29\}\); denote it by \(x\).

Reducing the first-coordinate sum modulo \(7\), we obtain
\[
    4k+(4-k)+x\equiv3\pmod7,
    \qquad
    x\equiv-1-3k\pmod7.
\]
Since
\[
    4\equiv4,\qquad 9\equiv2,\qquad 29\equiv1\pmod7,
\]
the only possibilities are
\[
    (k,x)=(3,4)\qquad\text{or}\qquad(k,x)=(4,29).
\]

In the first case, the first coordinates are
\[
    1,4,11,11,11.
\]
The three copies of \((11,1)\) contribute \(3\) modulo \(11\), so the
other two second coordinates \(u,v\in\mathbb Z_{11,2}\) would have to
satisfy
\[
    u+v\equiv0\pmod{11}.
\]
This is impossible because \(-1\) is not a quadratic residue modulo
\(11\).  In the second case, the first coordinates are
\[
    29,11,11,11,11.
\]
The four copies of \((11,1)\) contribute \(4\) modulo \(11\), so the
remaining second coordinate would have to be
\[
    3-4\equiv-1\equiv10\pmod{11},
\]
but \(10\notin\mathbb Z_{11,2}\).  Therefore
\[
    (3,3)\notin5A.
\]

Now define
\[
    \mathcal P
    =
    \{p\in\mathbb P:p^2\bmod385\text{ corresponds to an element of }A\}.
\]
Since \(385=5\cdot7\cdot11\), every element of
\(\mathbb Z_{385,2}\) has exactly \(2^3=8\) square roots among the
reduced residue classes modulo \(385\).  Thus \(\mathcal P\) consists
of
\[
    21\cdot8=168
\]
of the
\[
    \phi(385)=240
\]
reduced residue classes modulo \(385\).  By the prime number theorem
in arithmetic progressions,
\[
    \underline d(\mathcal P)
    =
    \lim_{x\to\infty}
    \frac{|\mathcal P\cap[1,x]|}{\pi(x)}
    =
    \frac{168}{240}
    =
    \frac7{10}.
\]

Finally, consider integers satisfying
\[
    n\equiv3\pmod{385},
    \qquad
    n\equiv5\pmod{24}.
\]
By the Chinese remainder theorem, these are precisely the integers
satisfying
\[
    n\equiv773\pmod{9240}.
\]
If such an integer had a representation
\[
    n=p_1^2+\cdots+p_5^2,
    \qquad
    p_1,\ldots,p_5\in\mathcal P,
\]
then reducing modulo \(385\) would give
\[
    (3,3)\in5A,
\]
a contradiction.  Hence every integer
\[
    n\equiv773\pmod{9240}
\]
fails to be a sum of five squares of primes from \(\mathcal P\).
Since \(\underline d(\mathcal P)=7/10\), this proves
\[
    \mathfrak d_5\geq\frac7{10}.
\]

\providecommand{\bysame}{\leavevmode\hbox to3em{\hrulefill}\thinspace}
\providecommand{\MR}{\relax\ifhmode\unskip\space\fi MR }
\providecommand{\MRhref}[2]{%
  \href{http://www.ams.org/mathscinet-getitem?mr=#1}{#2}
}
\providecommand{\href}[2]{#2}

\end{document}